\documentclass[paper,10pt]{article}
\usepackage{fullpage}
\usepackage{latexsym}
\usepackage{amsfonts, amsthm}
\usepackage[all]{xy}
\usepackage{rotating}
\usepackage{amsmath}
\usepackage{amssymb}
\usepackage{amscd}
\usepackage{mathrsfs}
\usepackage[english]{babel}
\usepackage{makeidx}
\usepackage{hyperref}
\usepackage[latin1]{inputenc}
\usepackage{fancyhdr}
\newtheorem{definition}{Definition}[section]
\newtheorem{proposition}[definition]{Proposition} 
\newtheorem{lemma}[definition]{Lemma} 
\newtheorem{theorem}[definition]{Theorem}
 
\newtheorem{remark}[definition]{Remark}
\newtheorem{example}[definition]{Example}
\newtheorem*{theorem*}{Theorem}

\newcommand{\ac}{\`}
\newcommand{\0}{\mathcal{O}}    
\newcommand{\Spf}{\mathrm{Spf}}
\newcommand{\f}[1] {\mathscr{#1}}
\newcommand{\Sp}{\mathrm{Spec}}

\begin{document}
\title{\textsc{On $p$-adic differential equations on semistable varieties}}
\author{Valentina Di Proietto}
\date{}
\maketitle

\textsc{Abstract -} \begin{small} In this paper we prove a comparison theorem between the category of certain modules with integrable connection on the complement of a normal crossing divisor of the generic fiber of a proper semistable variety over a DVR and the category of certain log overconvergent isocystrals on the special fiber of the same open.  
\end{small}

\section{Introduction}
The ``th\'eor\`eme d'alg\'ebrisation'' by Christol and Mebkhout (th\'eor\`eme 5.0-10 of \cite{CtMeIV}) asserts that on a open of a smooth and proper curve all the overconvergent isocrystals (with some non-Liouville conditions) are algebraic. The aim of this paper is to give a generalized version of this result to the case of a variety of arbitrary dimension that is the special fiber of a semistable variety. \\
We first recall the ``th\'eor\`eme d'alg\'ebrisation''  and explain in which sense our result generalizes it.
Let $V$ be a complete discrete valuation ring of mixed characteristic $(0,p)$ with uniformizer $\pi$, let $K$ be its fraction field and let $k$ be the residue field. Christol and Mebkhout consider a proper and smooth curve $X$ over $V$ and an affine open $U$ with complement $D$. They define a functor 
\begin{equation}\label{CM}
\dag:MICLS(U_K/K) \longrightarrow I^{\dag}((U_k,X_k)/\Spf(V))
\end{equation}
where the first category is the category of algebraic modules with connection on $U_K$ which satisfy certain convergent conditions and the second is the category of overconvergent isocrystals as defined by Berthelot in \cite{Be}.\\
They prove that under some non-Liouville conditions the functor $\dag$ is essentially surjective, \emph{i. e.} every overconvergent isocrystal is algebraic. Moreover they notice that, always assuming non-Liouville conditions, $\dag$ is fully faithful if one restricts to the category of algebraic modules with connection on $U_K$ with some convergent conditions that are extendable to module with connections on $X_K$ and logarithmic singularities along $D_K$. The image of this restricted functor turns out to be the category of overconvergent isocrystals with slope zero (probl\ac eme 5.0-14 1) of \cite{CtMeIV} and paragraph 6 of \cite{CtMeII} ).\\
We are interested in the following generalized situation.\\ Let $X$ be a proper semistable variety over $\Sp(V)$, which means that locally for the \'etale topology is \' etale over $\Sp V [x_1, \dots ,x_n,y_1,\dots ,y_n]/(x_1 \dotsm x_r - \pi)$, with $D$ a normal crossing divisor, which \' etale locally is given by the equation $\{y_1\dotsm y_s = 0\}$.
The divisor $X_k\cup D$ induces on $X$ a logarithmic structure that we denote by $M$; similarly, the closed point induces a logarithmic structure on $\Sp(V)$ that we denote by $N$.
We assume we have the following diagram of fine log schemes 
\begin{equation} \label{diagrammacommintro}
\xymatrix{ 
\ (X_k,M)\ \ar@{^(->}[r] \ar[d] &\ (X,M)\  \ar[d]& \ (X_K,M)\ \ar@{_(->}[l] \ar[d] \\  
\ (\Sp (k), N)\  \ar@{^(->}[r]& \  (\Sp (V), N)  & \ \ (\Sp (K),N)\ \ar@{_(->}[l] \\ 
}
\end{equation}
where the two squares are cartesian. The log structures denoted again by $M$ on the special fiber $X_k$ and by $N$ on $k$ are defined in such a way that the closed immersions of fine log schemes $(X_k, M)\hookrightarrow (X,M)$ and $(\Sp (k),N)\hookrightarrow (\Sp (V),N)$ are exact. In an analogous way the log structures on $X_K$ and on $K$ are defined in such a way that the open immersions $(X_K,M)\hookrightarrow (X,M)$ and $(\Sp (K), N)\hookrightarrow (\Sp (V),N)$ are strict. Let us note that the log structure on $K$ constructed in this way is isomorphic to the trivial log structure. \\
We consider on $X$ the open $U$ defined as the complement of the divisor $D$; then there is an open immersion
$$j:U=X\setminus D\hookrightarrow X$$
and it induces on $U$ the log structure $j^*(M)$ that, by abuse of notation, we again denote by $M$.\\ 
So we have a diagram analogous to (\ref{diagrammacommintro}) for $U$, with the same notations for the log structures:
\begin{equation*} \label{diagrammacommintroU}
\xymatrix{ 
\ (U_k,M)\ \ar@{^(->}[r] \ar[d] &\ (U,M)\  \ar[d]& \ (U_K,M)\ \ar@{_(->}[l] \ar[d] \\  
\ (\Sp (k), N)\  \ar@{^(->}[r]& \  (\Sp (V), N)  & \ \ (\Sp (K),N)\ \ar@{_(->}[l] \\ 
}
\end{equation*}
In this situation we consider the category, denoted by $MIC(U_K/K)^{reg}$, of pairs $(E,\nabla)$ where $E$ is a sheaf of coherent $\mathcal{O}_{U_K}$-modules and $\nabla$ is an integrable connection regular along $D_K$.\\
As for the rigid side we look at $I^{\dag}((U_k,X_k)/\Spf(V))^{log, \Sigma}$, the category of overconvergent log isocrystals on the log pair given by $((U_k,M),(X_k,M)/(\Spf (V),N))$ with $\Sigma$-unipotent monodromy along $D_k$, where $\Sigma$ is a subset of $\mathbb{Z}_p^h$, satisfying some non-Liouville hypothesis and $h$ is the number of the irreducible component of $D_k$. A log overconvergent isocrystal is represented by a module with connection on a strict neighborhood $W$ of the tube $]U_k[_{\hat{X}}$ in the tube $]X_k[_{\hat{X}}$. To define $\Sigma$-unipotent monodromy we proceed \'etale locally and we fix an irreducible component $D^{\circ}_{j,k}$ of $D^{\circ}_{k}$, the smooth locus of $D_k$; then the part of the tube $]D^{\circ}_{j,k}[_{\hat{X}}$ which is contained in $W$ is isomorphic to a product of an annulus of small width times certain rigid space associate to $D^{\circ}_{j,k}$, that we think as a base. A log overconvergent isocrystal $\mathcal{E}$ has $\Sigma$-unipotent monodromy along $D^{\circ}_{j,k}$ if $\mathcal{E}$, restricted to the product described above, admits a filtration such that every successive quotient is the pullback of a module with connection on the base twisted by a module with connection on the annulus depending on $\Sigma$. We say that $\mathcal{E}$ has $\Sigma$-unipotent monodromy along $D_k$ if the above condition holds for every irreducible component of $D^{\circ}_{k}$. The category of overconvergent log isocrystals is defined by Shiho in \cite{Sh4} as a log version of the category of overconvergent isocrystals defined by Berthelot in \cite{Be} and the notion of $\Sigma$-unipotent monodromy is introduced by Shiho in \cite{Sh6} as a generalization of the notion of unipotent monodromy introduced by Kedlaya 
in \cite{Ke}. Let us note that the notion of $\Sigma$-unipotent monodromy for a module with connection on an annulus, with $\Sigma$  satisfying some non-Liouville conditions, coincides with the notion of satisfying the Robba condition and having exponent in the sense of Christol and Mebkhout in $\Sigma$ (proposition 1.18 of \cite{Sh7}). \\
Our main result is: 
\begin{theorem}\label{mainintro}
There is a natural algebrization functor 
$$I^{\dag}((U_k,X_k)/\Spf(V))^{log, \Sigma}\longrightarrow MIC(U_K/K)^{reg}.$$
It is a fully faithful functor.
\end{theorem}
Let us note that our functor goes in the opposite direction with respect to Christol and Mebkhout's functor $\dag$.\\

The strategy of our proof is as follows.\\
We denote by $I_{conv}((X_k,M)/(\Spf(V),N))^{lf,\Sigma}$ the category of locally free log convergent isocrystals on  the log convergent site $((X_k,M)/(\Spf(V),N))_{conv}$ with exponents in $\Sigma$.  The category of log convergent isocrystals is defined by Shiho in \cite{Sh1}, as a log version of the category of convergent isocrystal defined by Ogus in \cite{Og} and by Berthelot in \cite{Be}.  We consider the category of log overconvergent isocrystals on $((U_k,M),(X_k,M)/(\Spf (V),N))$ with $\Sigma$-unipotent monodromy and we prove that the restriction functor
\begin{equation}\label{jdagintro}
j^{\dag}:I_{conv}((X_k,M)/(\Spf(V),N))^{lf,\Sigma}\longrightarrow I^{\dag}((U_k,X_k)/\Spf(V))^{log, \Sigma}
\end{equation}
is an equivalence of categories.\\
The equivalence of categories in (\ref{jdagintro}) is a generalization of theorem 3.16 of \cite{Sh6}, since Shiho proves the same result in the case of $X_k$ smooth and $D_k$ a divisor with simple normal crossing, \emph{i. e.} its components are regular and meet transversally. \\
On the other hand we have a fully faithful functor 
$$\tilde{i} : I_{conv}((X_k,M)/(\Spf(V),N))^{lf,\Sigma}\longrightarrow MIC((X_K, M)/K)^{lf, \Sigma}$$
between locally free log convergent isocrystals with exponents in $\Sigma$ and locally free $\mathcal{O}_{X_K}$-modules with integrable connection on $X_K$, logarithmic singularities on $D_K$ and exponents in $\Sigma$. The theorem of algebraic logarithmic extension of \cite{AnBa} I,4 gives us a fully faithful functor
\begin{equation}\label{ABintro}
MIC((X_K, M_D)/K)^{lf, \Sigma}\longrightarrow MIC(U_K/K)^{reg}.
\end{equation}
If we denote by $MICLS(U_K/K)^{reg, \Sigma}$ the essential image of the functor $\tilde{i}$ in $MIC(U_K/K)^{reg}$, we can conclude that we have an equivalence of categories 
$$I^{\dag}((U_k,X_k)/\Spf(V))^{log, \Sigma}\longrightarrow MICLS(U_K/K)^{reg, \Sigma}.$$

Let us now describe in detail the contents of the paper.\\
In the first paragraphs we suppose that $(X, M)\rightarrow (\Sp (V), N)$ is a general log smooth morphism of fine log schemes (not necessarily semistable), with $X$ proper and $X_k$ reduced. We denote by $(\hat{X}, M)\rightarrow (\Spf (V), N)$ the associated morphism of the $p$-adic completions. We recall the definition of modules with integrable log connection on a fine log scheme, \emph{i.e.} the category $MIC((X_K, M)/K)$, the definition of log infinitesimal isocrystals, the category $I_{inf}((\hat{X},M)/(\Spf(V),N))$, and the definition of log convergent isocrystals, the category $I_{conv}((X_k,M)/(\Spf(V),N))$. Using the fact that $X$ is proper Shiho (Theorem 3.15, Corollary 3.2.16 of \cite{Sh1}) shows that $MIC((X_K, M)/K)$ is equivalent to the category $I_{inf}((X_k,M)/(\Spf(V),N))$ of log infinitesimal isocrystals . We prove that the property of being locally free is stable with respect to this equivalence of categories (propositions \ref{primolibero} and \ref{secondolibero}), so that we have an equivalence of categories between $I_{inf}((X_k,M)/(\Spf(V),N))^{lf}$ and $MIC((X_K, M)/K)^{lf}$. We use the definition of log convergent site and isocrystals on it given by Shiho in chapter 5 of \cite{Sh1} and the functor 
$$\Phi : I_{conv}(X_k,M)/(\Spf(V),N))\longrightarrow I_{inf}((\hat{X},M)/(\Spf(V),N)),$$
that he defines between log convergent isocrystals and log infinitesimal isocrystal. We adapt the proof of proposition 5.2.9 of \cite{Sh1} to show that 
\begin{equation}
\begin{split}
\tilde{\Phi} : I_{conv}((X_k,M)/(\Spf(V),N))^{lf}\longrightarrow I_{inf}((\hat{X},M)&/(\Spf(V),N))^{lf}\\
&\cong MIC((X_K, M)/K)^{lf}
\end{split}
\end{equation}
is fully faithful (theorem \ref{ff}). Shiho proves full faithfulness on the subcategory of log convergent isocrystals that are iterated extensions of the unit object. As in Shiho's proof, the key tool used is locally freeness. Finally, we give a characterization of the essential image of the functor $\tilde{\Phi}$ in terms of special stratifications, a suitable modification of the special stratifications introduced by Shiho in \cite{Sh1}. We conclude by proving that requiring a stratification to be special is the same as requiring that the radius of convergence of the stratification is 1. In this last part we do not use log differential calculus, because we prove that we can restrict to the case of trivial log structures (proposition \ref{U denso}).\\ 
Then we restrict to the semistable situation: we suppose that the morphism $(X,M)\rightarrow (\Sp(V),N)$ is as before (\ref{diagrammacommintro}). 
After describing the geometric situation, we introduce the definition of log convergent isocrystals with exponents in $\Sigma$, using log-$\nabla$-modules defined by Kedlaya in \cite{Ke} and used by Shiho in \cite{Sh6}. We first define the notion of exponents in $\Sigma$ \'etale locally using coordinates, then we prove that it is independent on the particular \'etale cover chosen and on the choice of the coordinates (lemma \ref{indepnilpotentresidue}). After that, following Shiho, we recall the notion of $\Sigma$-unipotence for a log-$\nabla$-module defined over a product of a rigid analytic space and a polyannulus and also the extension theorem for $\Sigma$-unipotent log-$\nabla$-modules. Then we analyze log overconvergent isocrystals. Our setting is different from Shiho's (\cite{Sh4}), since in our case the base has a non-trivial log structure. To define a log overconvergent isocrystal with $\Sigma$-unipotent monodromy we proceed \'etale locally. Then we recall the three key propositions that we use in the proof of the main theorem. The first property is called ``generization of monodromy'' and asserts that the property of being $\Sigma$-unipotent is generic on the base in some sense that we make precise (proposition \ref{Generization of monodromy}).
The second property, called ``overconvergent generization", says that the property of being $\Sigma$-unipotent can be extended on strict neighborhoods (proposition \ref{Overconvergent generization}). The third one says that under certain conditions a convergent log-$\nabla$-module with exponents in $\Sigma$ is $\Sigma$-unipotent (proposition \ref{Convergent plus nilpotent implies unipotent}). The proofs of these propositions are given in \cite{Sh6} as a generalization of the one contained in \cite{Ke}. Using these properties we prove that the notion of $\Sigma$-unipotent monodromy for an overconvergent isocrystal is well posed (lemma \ref{indepunipmonodromy}). Finally we prove (theorem \ref{log extension}) that the restriction functor
$$j^{\dag}:I_{conv}((X_k,M)/(\Spf(V),N))^{lf,\Sigma}\longrightarrow I^{\dag}((U_k,X_k)/\Spf(V))^{log, \Sigma}$$
is an equivalence of categories. 
The strategy of the proof is the same as in theorem 3.16 of \cite{Sh6} and theorem 6.4.5 of \cite{Ke}.\\
In the last part we verify that the notion of exponents in $\Sigma$ behaves well with respect to the functor $\tilde{i}$, \emph{i. e.} the functor 
$$I_{conv}((X_k,M)/(\Spf(V),N))^{lf, \Sigma}\longrightarrow MIC((X_K, M_D)/K)^{lf, \Sigma}$$ 
is well defined. Finally we adapt Andr\'e and Baldassarri's theorem of algebraic logarithmic extension to find theorem \ref{mainintro}.\\

\section{Connections with logarithmic poles}
We will recall the definition of log connection on a fine log scheme or on a $p$-adic fine log formal scheme, which is taken from \cite{Sh1} definition 3.1.1; see also \cite{Kz} paragraph 4.\\
We suppose the reader familiar with the language of log schemes introduced in \cite{Ka} and with his version for formal schemes given in \cite{Sh1} chapter 2.\\ 
We denote by $V$ a discrete valuation ring of mixed characteristic $(0, p)$, complete and separated for the $p$-adic topology, $\pi$ a uniformizer of $V$, $K$ its fraction field and $k$ its residue field. By a formal scheme over $\Spf (V)$ or a formal $V$-scheme, we mean a $p$-adic Noetherian and topologically of finite type formal scheme over $V$.\\
According to Shiho's notation we will give the following definition.

\begin{definition}
If $X$ is a scheme, we denote the category of coherent $\mathcal{O}_X$-modules by Coh$(\mathcal{O}_X)$. If $X$ is formal $V$-scheme we denote by Coh$(K\otimes \mathcal{O}_X)$ the category of sheaves of $K\otimes_V\mathcal{O}_X$-modules that are isomorphic to $K\otimes_V F$ for some coherent $\mathcal{O}_X$-module $F$. We will call an object of Coh$(K\otimes \mathcal{O}_X)$ an isocoherent sheaf.
\end{definition}
The category of isocoherent sheaves is introduced in \cite{Og} and (see \cite{Og} remark 1.5) is equivalent to the category of coherent sheaves on $X^{an}$, the rigid analytic space associated to the $p$-adic formal scheme $X$ via Raynaud generic fiber. 
\begin{definition}\label{connections}
Let $f:(X,M)\rightarrow(S,N)$ be a map of fine log schemes (resp. fine log formal $V$-schemes) and let $E$ be a coherent $\mathcal{O}_{X}$-module (resp. $E$ $\in $ Coh$(K\otimes \mathcal{O}_X)$). A log connection on $E$ is an $\mathcal{O}_S$-linear map
$$\nabla: E \rightarrow E\otimes \omega^1_{(X,M)/(S,N)}$$ 
that is additive and satisfies the Leibniz rule:
$$\nabla(ae)=a\nabla(e)+e\otimes da$$
for $a$ $\in$ $\mathcal{O}_X$ and $e$ $\in$ $E.$\\ 
Here $\omega^1_{(X,M)/(S,N)}$ denotes the sheaf of log differentials (resp. the sheaf of log formal differentials).\\
We can extend $\nabla$ to $\nabla_i$ 
$$E\otimes \omega^i_{(X,M)/(S,N)}\xrightarrow{\nabla_i}E\otimes \omega^{i+1}_{(X,M)/(S,N)},$$ 
where $\nabla_i(e\otimes \omega)=e \otimes d\omega+(-1)^i\nabla(e)\wedge \omega$. We say that $\nabla$ is integrable if $\nabla_{i+1} \circ \nabla_i=0$ for all $i$.\\ 
We indicate the category of pairs $(E,\nabla)$ of a coherent sheaf $E$ (resp. an isocoherent sheaf) and an integrable log connection $\nabla$ with $MIC((X,M)/(S,N))$ (resp. $\widehat{MIC}((X,M)/(S,N))$).  If $M$ and $N$ are isomorphic to the trivial log structures we will write $M=\textrm{triv}$ and $N=\textrm{triv}$ and we use the notation $MIC(X/S)$ (resp. $\widehat{MIC}(X/S)$) to denote $MIC((X,\textrm{triv})/(S,\textrm{triv}))$ (resp. $\widehat{MIC}((X,\textrm{triv})/(S,\textrm{triv}))$).
\end{definition}
From proposition 8.9 of \cite{Kz} we know that in the smooth case in characteristic zero every coherent module with integrable connection is locally free; more  precisely we have the following result.
\begin{proposition}\label{8.9Ka}
If $S$ is the spectrum of a field of characteristic $0$ and $X$ is a smooth scheme over $S$, then, for every object $(E,\nabla)$ in $MIC(X/S)$, $E$ is a locally free $\mathcal{O}_X$-module.
\end{proposition}  
This proposition is not true if one admits log-connections; in fact, there is the following example. 
\begin{example}\label{esempiodishiho}
Let $X$ be a curve over $K$ and $D$ a closed point, locally defined by the equation $\{t=0\}$. If we call $M_D$ the log structure on $X$ induced by $D$ we can consider the following log connection
$$d:\mathcal{O}_{X}\rightarrow \omega^{1}_{(X,M_D)/(K,triv)}.$$
If we consider the subsheaf $\mathcal{O}_{X_K}(-D)$, that consists of sections vanishing on $D$, then $d$ induces a log connection on $\mathcal{O}_{X_K}(-D)$. We can see it locally: every section of $\mathcal{O}_{X_K}(-D)$ can be written as a product of $ft$, with $f$ in $\mathcal{O}_{X_K}$. The induced log connection is:
$$\mathcal{O}_{X_K}(-D)\rightarrow \mathcal{O}_{X_K}(-D)\otimes \omega^{1}_{(X,M_D)/(K,triv)}$$
$$ft\mapsto d(ft)=fdt+tdf=ft d\mathrm{log} t+t df$$
We can induce an integrable log connection on the quotient $\mathcal{O}_{X_K}/\mathcal{O}_{X_K}(-D)=\mathcal{O}_D$ which is a skyscraper sheaf.
\end{example}
As in the case where the log structures are trivial (see for example \cite{BeOg} chapter 1), the category of modules with integrable log connections is equivalent to the category of log stratifications: we now describe this equivalence.\\
If $(X,M)\rightarrow(S,N)$ is a morphism of fine log schemes (resp. a morphism of fine log formal schemes), we denote by $(X^n,M^n)$ the $n$-th log infinitesimal neighborhood of $(X,M)$ in $(X,M)\times_{(S,N)}(X,M) $(defined in \cite{Ka} (5.8), \cite{Sh1} remark 3.2.4 as a log version of the $n$-th infinitesimal neighborhood described in \cite{BeOg} chapter 1).
\begin{definition}\label{Stratifications}
Let $(X, M)\rightarrow(S, N)$ be a morphism of fine log schemes (resp. of fine log formal $V$-schemes) and let $E$ be a coherent sheaf (reps. an isocoherent sheaf on $X$). Then, a log stratification (resp. a formal log stratification) on $E$ is a family of morphisms $\epsilon_n:\mathcal{O}_{X^n}\otimes E\rightarrow E\otimes\mathcal{O}_{X^n}$, that satisfy the conditions:
\begin{itemize}
\item[(i)] $\epsilon_n$ is $\mathcal{O}_{X^n}$-linear and $\epsilon_{0}$ is the identity;
\item[(ii)] $\epsilon_n$ and $\epsilon_m$ are compatible via the maps
$$\mathcal{O}_{X^n}\rightarrow \mathcal{O}_{X^m},\; \textrm{for}\; m\leq n;$$
\item[(iii)]
(cocycle condition) if we call $p_{i,j}$ (for $i,j=1,2,3$) the projections from the $n$-th log infinitesimal neighborhood $(X^n(2),M(2)^n)$ of $(X,M)$ in $(X,M)\times_{(S,N)} (X,M)\times_{(S,N)} (X,M)$ to the $n$-th log infinitesimal neighborhood $(X^n,M^n)$ of $(X,M)$ in $(X,M)\times_{(S,N)} (X,M)$ 
$$p_{i,j}:(X^n(2),M^n(2))\rightarrow (X^n,M^n),$$
then for all $n$ 
$$p_{1,2}^*(\epsilon_n)\circ p_{2,3}^*(\epsilon_n)=p_{1,3}^{*}(\epsilon_n).$$ 
\end{itemize}
We denote by $Str((X,M)/(S,N))$ the category of log stratifications (resp. with $\widehat{Str}((X,M)/(S,N))$ the category of log formal stratifications).
\end{definition}
Theorem 3.2.15 of \cite{Sh1} gives us the equivalence of categories that we announced: if $(X,M)\rightarrow(S,N)$ is log smooth morphism of fine log schemes over $\mathbb{Q}$ (resp. a formally smooth morphism of fine log formal $V$-schemes), then
$$Str((X,M)/(S,N))\cong MIC((X,M)/(S,N))$$
$$(\mathrm{resp.} \, \widehat{Str}((X,M)/(S,N))\cong \widehat{MIC}((X,M)/(S,N))).$$
\section{Log infinitesimal isocrystals}
We now define the infinitesimal site and log isocrystals on it. These are Shiho's  log formal analogous of Grothendieck's infinitesimal site and crystals on it defined in \cite{Gr} or \cite{BeOg} chapter 1. All the definitions that follow are taken from chapter 3 of \cite{Sh1}. We define the infinitesimal site only in the case of a morphism of fine log formal schemes, but analogous definition can be given for a morphism of log schemes.
\begin{definition}
Let $(\mathscr{X},M)\rightarrow(\mathscr{S},N)$ be a morphism of fine log formal $V$-schemes. An object of the log infinitesimal site $((\mathscr{X},M)/(\mathscr{S},N))_{inf}$, or by brevity $(\mathscr{X}/\mathscr{S})^{log}_{inf}$ when the log structures are clear, is a $4$-ple $(\f{U},\f{T},L,\phi)$ such that $\f{U}$ is a formal $V$-scheme formally \'{e}tale over $\f{X}$, $(\mathscr{T},L)$ is a fine log formal $V$-scheme over $(\f{S},N)$ and 
$\phi: (\f{U},M)\rightarrow(\f{T},L)$ is a nilpotent exact closed immersion of log formal $V$-schemes over $(\f{S}, N)$. 
A morphism between $(\f{U},\f{T},L,\phi)$ and $(\f{U}',\f{T}',L',\phi')$ is pair of maps $g:(\f{T},L)\rightarrow(\f{T}',L')$ and $f:\f{U}\rightarrow \f{U}'$ such that $\phi'\circ f=g \circ \phi$. The coverings in this site are the coverings of $\f{T}$ for the \'{e}tale topology $\{\f{T}_i\rightarrow\f{T}\}_i$ such that $\f{U}_i=\f{T}_i\times_{\f{T}}\f{U}$. We sometimes denote the $4$-ple $(\f{U},\f{T},L,\phi)$ simply by $\f{T}.$ 
\end{definition}
\begin{definition}\label{loginf}
A log isocrystal on the infinitesimal site $(\mathscr{X}/\mathscr{S})^{log}_{inf}$, or a log infinitesimal isocrystal, is a sheaf $\mathcal{E}$ on $(\mathscr{X}/\mathscr{S})^{log}_{inf}$ such that: 
\begin{itemize}
\item[(i)] for every object $(\f{U},\f{T},L,\phi)$ the Zariski sheaf $\mathcal{E}_{\f{T}}$ induced on $\f{T}$ is an isocoherent sheaf;
\item[(ii)] for every morphism $g:\f{T}\rightarrow\f{T'}$, the map $g^*(\mathcal{E}_{\f{T}'})\rightarrow \mathcal{E}_{\f{T}}$ is an isomorphism.
\end{itemize}
We denote the category of log isocrystals on the infinitesimal site $(\mathscr{X}/\mathscr{S})^{log}_{inf}$ by $I_{inf}(\f{X}/\f{S})^{log}$.
\end{definition}
\begin{definition}\label{locally free}
Let $\f{X}$ be a formal $V$-scheme and let $\mathcal{F}$ be an isocoherent sheaf. We say that it is locally free module if there is a formal affine covering $\{U_i\}_{i\in I}$ of $\mathscr{X}$ such that for every $U_i=\Spf{A_i}$, there exists a finitely generated $A_i$-module $F_i$ such that $\mathcal{F}(U_i)=K\otimes F_i$ is a projective $K\otimes A_i$-module.  
\end{definition}
\begin{definition}\label{locally free2}
A log isocrystal $\mathcal{E}$ on the infinitesimal site $(\mathscr{X}/\mathscr{S})^{log}_{inf}$ is said to be locally free if for every object $(\f{U},\f{T},L,\phi)$ of the infinitesimal site, the sheaf $\mathcal{E}_{\f{T}}$ induced on $\f{T}$ is an isocoherent locally free module. \\
We will denote the subcategory of $I_{inf}(\f{X}/\f{S})^{log}$ consisting of the locally free infinitesimal log isocrystal by $I_{inf}(\f{X}/\f{S})^{log,lf}$.
\end{definition}
Thanks to theorem 3.2.15 of \cite{Sh1} we can see that if $(\f{X}, M)\rightarrow (\f{S}, N)$ is a formally log smooth morphism of fine log formal $V$-schemes, then there exists a canonical equivalence of categories
$$I_{inf}((\f{X},M)/(\f{S},N))\cong \widehat{MIC}((\f{X},M)/(\f{S},N))$$
\section{Log convergent isocrystals}
In this section we define the log convergent site and the isocrystals on it. The following definitions are taken from \cite {Sh1} paragraph 5.1.
\begin{definition} 
For every log formal $V$-scheme $(\mathscr{Y},M)$ we indicate with $\f{Y}_1$ the closed subscheme defined by the ideal $p$ and the associated reduced subscheme of $\f{Y}_1$  by $\f{Y}_0$.
\end{definition}

\begin{definition}
Let $(\mathscr{X},M)\rightarrow(\mathscr{S},N)$ be a morphism of fine log formal $V$-schemes. An enlargement of $(\mathscr{X},M)$ over $(\mathscr{S}, N)$ is a triple $(\mathscr{T},L,z)$, that we will indicated with $\mathscr{T}$, such that $(\f{T},L)$ is a fine log formal $V$-scheme over $(\f{S},N)$ and $z$ is a morphism $(\f{T}_0,L)\rightarrow(\f{X},M)$ over $(\f{S},N)$.
A morphism between two enlargements $(\mathscr{T},L,z)$ and $(\mathscr{T}',L',z')$ is a morphism $g:(\f{T},L)\rightarrow (\f{T}',L')$ such that $z=z'\circ g_0$, where $g_0:(\mathscr{T}_0,L)\rightarrow (\mathscr{T}'_0,L')$ is the map induced by $g$.
\end{definition}
\begin{definition}
We define the log convergent site of $(\mathscr{X},M)\rightarrow(\mathscr{S},N)$ to be the site whose objects are enlargements, morphisms are morphisms of enlargements and coverings are given by the \'etale topology on $\f{T}$. We denote it by $((\mathscr{X},M)/(\mathscr{S},N))_{conv}$ or $(\mathscr{X}/\mathscr{S})^{log}_{conv}$ if there is no ambiguity about the log structures.
\end{definition}
\begin{definition}\label{logconvisocrystals}
A log isocrystal on $(\mathscr{X}/\mathscr{S})^{log}_{conv}$, or a log convergent isocrystal, is a sheaf $\mathcal{E}$ on $(\mathscr{X}/\mathscr{S})^{log}_{conv}$ such that:
\begin{itemize}
\item[(i)] for every enlargement  $(\f{T},L,z)$ the Zariski sheaf $\mathcal{E}_{\f{T}}$ induced on $\f{T}$ is an isocoherent sheaf;
\item[(ii)] for every morphism of enlargements $g:(\f{T},L)\rightarrow(\f{T'},L')$, the map $g^*(\mathcal{E}_{\f{T}'})\rightarrow \mathcal{E}_{\f{T}}$ is an isomorphism.
\end{itemize}
We denote  by $I_{conv}(\f{X}/\f{S})^{log}$ the category of log isocrystals on the log convergent site.
\end{definition}
\begin{definition}\label{locally free1}
A log isocrystal $\mathcal{E}$ on the convergent site $((\mathscr{X},M)/(\mathscr{S},N))_{conv}$ is locally free if for every object $\f{T}$ on the convergent site the sheaf $\mathcal{E}_{\f{T}}$ induced on $\f{T}$ is an isocoherent locally free sheaf. \\
We denote the subcategory of $I_{conv}(\f{X}/\f{S})^{log}$ consisting of the locally free log isocrystals on the convergent site by $I_{conv}(\f{X}/\f{S})^{log, lf}$.
\end{definition}

\section{Relations between algebraic and analytic modules with integrable connections}\label{Relations between algebraic and analytic modules with integrable connections}
In what follows we consider the following situation: we fix $(X,M)\rightarrow(\Sp(V),N)$ a log smooth and proper morphism of fine log schemes. We denote by $(X_K,M)\rightarrow (\Sp(K),N)$ its generic fiber and $(X_k,M)\rightarrow(\Sp(k),N)$ its special fiber, that we suppose to be reduced.  
So we have a commutative diagram  
\begin{equation} \label{diagrammacomm}
\xymatrix{ 
\ (X_k,M)\ \ar@{^(->}[r] \ar[d] &\ (X,M)\  \ar[d]_f& \ (X_K,M)\ \ar@{_(->}[l] \ar[d] \\  
\ (\Sp (k), N)\  \ar@{^(->}[r]& \  (\Sp (V), N)  & \ \ (\Sp (K),N)\ \ar@{_(->}[l] \\ 
}
\end{equation}
The log structures on $X_k$ and $\Sp(k)$, that with an abuse of notation we again call $M$ and $N$, are defined in such a way that the inclusions in $(X,M)$ and $(\Sp(V),N)$ in this diagram are exact closed immersions of log schemes. \\
In the same way we define the log structures $M$ on $X_K$ and $N$ on $\Sp(K)$ as the log structures defined in such a way that the inclusions to $(X,M)$ and $(\Sp(V),N)$ are strict.\\ 
We consider the $p$-adic completion of $(X,M)\rightarrow (\Sp(V),N)$ and we call it $(\hat{X},M) \rightarrow (\Spf(V),N)$. With $(\hat{X},M)$ we mean $\hat{X}=(X_k, \varprojlim \mathcal{O}_{X}/p^n \mathcal{O}_X)$ with the log structure, that we call again $M$ with an abuse of notation, defined as the pull back of $M$ via the canonical morphism $\hat{X}\rightarrow X$.\\
Another way to see this is \cite{ChFo} Definition-Lemma 0.9, where the authors prove that the log structure $M$ over $\hat{X}$ is isomorphic to $\varprojlim_n (M)_n$ with $(M)_n$ the log structure on $\hat{X}_n=(X_k, \mathcal{O}_X/p^n \mathcal{O}_X)$ that is the pull-back of $M$ via the morphism $\hat{X}_n\rightarrow X$.\\
Now we want to construct a fully faithful functor
$$i : I_{conv}((\hat{X},M)/(\Spf(V),N))^{lf}\longrightarrow MIC((X_K, M)/(K,N))^{lf}.$$
In corollary 3.2.16 of \cite{Sh1} we have a useful characterization of algebraic modules with integrable log connection: the following result holds. 
\begin{proposition}\label{derhaminf} Under the above assumptions, there is an equivalence of categories 
$$\Psi:MIC((X_K, M)/(\Sp(K),N))\longrightarrow I_{inf}((\hat{X},M)/(\Spf(V),N)).$$ 
\end{proposition}
As we saw in example \ref{esempiodishiho}, it is not true that every coherent module with integrable connection is locally free, so we will restrict to the category that we call  
$$MIC((X_K,M)/(K,N))^{lf},$$
that consists of pairs $(E,\nabla)$ where $\nabla$ is an integrable log connection and $E$ a locally free $\mathcal{O}_{X_{K}}$-module. \\
In the next two propositions we will see that the functor $\Psi$ of proposition \ref{derhaminf} induces an equivalence of categories 
$$MIC((X_K,M)/(K,N))^{lf}\longrightarrow I_{inf}((\hat{X},M)/(\Spf(V),N))^{lf}.$$
\begin{proposition}\label{primolibero}
For every element $\mathcal{E}=(E,\nabla)$ in $MIC((X_K,M)/(K,N))^{lf}$, the corresponding element $\Psi(\mathcal{E})$ $\in $ $I_{inf}((\hat{X},M)/(\Spf(V),N))$  is a log infinitesimal locally free isocrystal .
\end{proposition} 
\begin{proof}
For the proof we look carefully at the definition of the functor $\Psi$. 
The functor is defined as the composition of three functors each one being an equivalence of category.
The first functor $\Psi_1$ is the one that gives the equivalence of category between $MIC((X_K,M)/(K,N))$ and the category of log stratifications. So given $\mathcal{E}=(E,\nabla)$, with $E$ a locally free $\mathcal{O}_{X_{K}}$-module, $\Psi_1(\mathcal{E})$ is again the $\mathcal{O}_{X_{K}}$-module $E$ with a collection of isomorphisms $\epsilon_n:\mathcal{O}_{X_{K}^n}\otimes E\rightarrow E\otimes\mathcal{O}_{X_{K}^n}$, where $(X_{K}^{n},M^n)$ is the $n$-th log infinitesimal neighborhood of $(X_{K},M)$ in $(X_{K},M)\times_{(K,N)}(X_{K},M)$, that satisfies the conditions of definition \ref{Stratifications}.
By lemma 3.2.7 of \cite{Sh1} we know that all the $\mathcal{O}_{X_{K}^n}$ are free $\mathcal{O}_{X_{K}}$-modules, so $E\otimes\mathcal{O}_{X_{K}^n}$ are locally free $\mathcal{O}_{X_{K}}$-modules.

Now, thanks to example 1.4 of \cite{Og}, that uses a formal version of GAGA principle, we know that the category of coherent $\mathcal{O}_{X_{K}}$-modules on $X_{K}$ is equivalent to the category of isocoherent modules on $\hat{X}$. So we can associate to our $E$ an isocoherent sheaf $\hat{E}$, that we now show to be locally free.
The functor that gives this equivalence is defined locally by extension of scalars: if we suppose that $X=\Sp(B)$ then the functor is
$$E\longmapsto E\otimes_{B_{K}}\hat{B}_{K} $$
where $B_{K}=B\otimes_V K$ and $\hat{B}_K=\varprojlim_n B/p^n \otimes K$. If $E$ is locally free as coherent $\mathcal{O}_{X_K}$-module, then $E$ is a projective-$B_K$ module. This implies that $E\otimes_{B_{K}}\hat{B}_{K}$ is a projective $\hat{B}_{K}$-module too. Indeed, since $E$ is projective,  there exists a $B_K$-module $F$ such that $E\oplus F$ is a free $B_K$-module. Then $(E\oplus F)\otimes_{B_{K}}\hat{B}_{K}$ is a free $\hat{B}_K$-module, which implies that $E\otimes_{B_{K}}\hat{B}_{K}$ is a projective $\hat{B}_K$-module. \\
Moreover if we indicate with $(\hat{X}^n, M^n)$ and $(X_{K}^n,M^n)$ the $n$-th log infinitesimal neighborhoods of the diagonal morphisms $(\hat{X},M)\rightarrow (\hat{X},M)\times_{(\Spf(V),N)}(\hat{X},M)$ and $(X_K,M)\rightarrow (X_K,M)\times_{(K,N)} (X_K,M)$ respectively, then we have an equivalence of categories also between $\textrm{Coh}(\mathcal{O}_{X_{K}^n})$ and $\textrm{Coh}(K\otimes\mathcal{O}_{\hat{X}^n})$ (\cite{Og} (1.4).)\\
This means that we have a functor $\Psi_2$ from $Str((X_K,M)/K)$ to $\widehat{Str}((\hat{X},M)/(\Spf(V),N))$. Moreover $\Psi_2$ is an equivalence of categories which sends locally free objects in locally free objects.

Now we construct the functor $\Psi_3$ between $\widehat{Str}((\hat{X},M)/(\Spf(V),N))$ and $I_{inf}(((\hat{X},M)/(\Spf(V),N))$.\\
Let us take an element $(\f{U},\f{T},L, \phi)$ of the log infinitesimal site $(\hat{X}/\Spf(V))^{log}_{inf}$. We know by definition that $(\f{U},M)\rightarrow(\Spf(V),N)$ is formally log smooth over $(\Spf(V),N)$, because $(\f{U},M)$ is formally \'etale over $(\hat{X},M)$, that is formally log smooth over $(\Spf{V},N)$. Therefore we have a diagram 
\[
\xymatrix{ 
\ (\f{U},M)\ \ar[r] \ar@{^(->}[1,0]_\phi &\ (\f{U},M)\  \ar[d] \\  
\ (\f{T},L)\  \ar[r]& \  (\Spf(V), N)   \\ 
}
\]  
and by proposition 2.2.13 of \cite{Sh1} we know that \'etale locally over $\f{T}$ there exists a morphism $c:(\f{T},L)\rightarrow (\f{U},M)$ that is a section for $\phi:(\f{U},M)\rightarrow(\f{T},L)$.\\ 
If we have $\f{T}'$ \'etale over $\f{T}$, then we call $s:(\f{T'},L)\rightarrow (\hat{X},M)$
the composition of the section $c$ with the morphism $(\f{U},M)\rightarrow (\hat{X},M)$.\\
Then we can define a sheaf over $\f{T}'$ using the pullback map $ s^{*}_{K}: \mathrm{Coh}(K \otimes \mathcal{O}_{\hat{X}})\rightarrow \mathrm{Coh}(K \otimes \mathcal{O_{\f{T}'}})$, and we call the resulting sheaf $\hat{E}_{\f{T}'}=s_K^{*}\hat{E}$. Let us note that $E_{\f{T}'}$ is obviously a locally free isocoherent sheaf on $\f{T'}$, as soon as $\hat{E}$ is.\\
If we have two sections $c$ and $d$ and respectively two morphisms $s$ and $t$, the formally log smooth morphism $s\times t: (\f{T}',L)\rightarrow (\hat{X},M)\times_{(\Spf(V),N)}(\hat{X},M)$ factors through the $n$-th log infinitesimal neighborhood $(\hat{X}^n,M^n)$, for some $n$:
$$s\times t: (\f{T}',L)\stackrel{u}{\longrightarrow} (\hat{X}^n,M^n)\longrightarrow (\hat{X},M)\times_{(\Spf(V),N)}(\hat{X},M),$$
because of the universal propriety of the $n$-th infinitesimal neighborhood.\\
Now, pulling back by $u$ the isomorphisms $\hat{\epsilon}_n:\0_{\hat{X}^n}\otimes \hat{E}\rightarrow \hat{E}\otimes \0_{\hat{X}^n}$ given by the stratification, we obtain an isomorphism $t_{K}^{*}\hat{E}\cong s_{K}^*\hat{E}$.
We want to descend this sheaf to $\f{T}$; this is possible using the cocycle condition and the theorem of faithfully flat descent for isocoherent sheaves of Gabber (\cite{Og} proposition 1.9.).\\
In fact, let us consider the formally \'etale morphism $(\f{T'},L)\rightarrow (\f{T},L)$ and the two projections 
\[
\xymatrix{ 
(\f{T}',L)\times_{(\f{T},L)}(\f{T}',L) \ar[d]_{p_1} \ar@{->}[0,1]^(.68){p_2} & (\f{T'},L) \\  
(\f{T}',L)  & .  \\ 
}
\]
Composing $p_i$ with the morphism $s:(\f{T}',L)\rightarrow (\hat{X},M)$ we find $\pi_i:(\f{T}',L)\times_{(\f{T},L)}(\f{T}',L)\rightarrow (\hat{X},M)$ for $i=1,2$ respectively.\\
As before we have a map
$\pi_1\times\pi_2:(\f{T},L')\times_{(\f{T},L)}(\f{T}',L)\rightarrow (\hat{X},M)\times_{(\Spf(V),N)}(\hat{X},M)$ that factors through the $n$-th log infinitesimal neighborhood and we can deduce an isomorphism $\pi_{1,K}^{*}\hat{E}\cong \pi_{2,K}^{*}\hat{E}$ and consequently an isomorphism $p_{1,K}^{*}\hat{E}_{\f{T}'}\cong p_{2,K}^{*}\hat{E}_{\f{T}'}$, that is a covering datum; to obtain a descent datum we adapt the above argument using the cocycle condition of the stratification.
Now using proposition 1.9 of \cite{Og} we are allowed to descend the sheaf $\hat{E}_{\f{T}'}$ to an isocoherent sheaf on $\f{T}$, that we call $\hat{E}_{\f{T}}$.\\
Let us note that $\hat{E}_{\f{T}}$ defines a log isocrystal on the infinitesimal site: to check property (ii) of definition \ref{loginf} we can use the same arguments as before.\\
Now we prove that $\hat{E}_{\f{T}}$ is a locally free isocoherent module on $\f{T}$. \\
We know that $\hat{E}_{\f{T}}$ is an isocoherent module on $\f{T}$, isomorphic to $K\otimes F$, for some coherent sheaf $F$ of $\mathcal{O}_{\f{T}}$-modules, and that there exists an \'etale covering of $\f{T}$ such that the $\hat{E}_{\f{T}}$, restricted to every element of the covering, is a locally free isocoherent module. We can restrict ourselves to the affine case; we assume that $\f{T}=\Spf(A)$, $\f{T}'=\Spf(B)$ and $\Spf(B)\rightarrow \Spf(A)$ \' etale surjective. We can conclude applying the following lemma.
\end{proof}

\begin{lemma}\label{Milne}
Let $A$ and $B$ be commutative noetherian $V$-algebras and let $M$ be a finitely generated $A\otimes K=A_K$-module.  
If we have a map $A\rightarrow B$ that is faithfully flat and $B_K\otimes_{A_K} M$ is projective, then $M$ is a projective $A_K$-module. 
\end{lemma}
\begin{proof}
We have the following isomorphism for every $A_K$-module $N$:
$$B_K\otimes_ {A_K} \mathrm{Ext}^i_{A_K}(M,N) \cong \mathrm{Ext}^i_{B_K}(B_K \otimes_{A_K} M,B_K \otimes_{A_K} N), $$
for every $i\geq 0 $, because $B_K$ is flat over $A_K$.
We can conclude that $M$ is projective because $B_K \otimes_{A_K} M$ is assumed to be projective. 
\end{proof}
It is true also the viceversa of proposition \ref{primolibero}.
\begin{proposition}\label{secondolibero}
If $\mathcal{E}$ is a log infinitesimal locally free isocrystal, then there exists  an object $(E, \nabla)$ $\in$ $MIC((X_K,M)/(K,N))^{lf}$ such that $\Psi ((E, \nabla))=\mathcal{E}.$ 
\end{proposition}
\begin{proof}
From proposition \ref{derhaminf} we know that there exists an element $(E,\nabla)$ in the category $MIC((X_K,M)/(K,N))$ such that $\Psi((E,\nabla))=\mathcal{E}$. \\
So we have to show that $E$ is a locally free $\mathcal{O}_{X_{K}}$-module.
By proposition \ref{primolibero} we are reduced to prove that the equivalence of categories 
\begin{equation}\label{GAGA}
j:\textrm{Coh}(\mathcal{O}_{X_K})\rightarrow \textrm{Coh}(K\otimes \mathcal{O}_{\hat{X}})
\end{equation}
behaves well with respect to locally free objects, in particular that if $\mathcal{F}$ is a locally free isocoherent module then there exists a locally free sheaf of $\mathcal{O}_{X_K}$-modules $F$ such that $j(F)=\mathcal{F}.$ \\
Let us take $F$ $\in $ Coh$(\mathcal{O}_{X_K})$ such that $\mathcal{F}=j(F)$. We claim that if $F$ is not locally free, then $\mathcal{F}$ is not locally free. Let us assume that $F$ is not locally free. So there exists an open affine, that we can suppose local, $U=\textrm{Spec}(A)$ $\subseteq$ $X_K$ such that $F|_{U}$ is not flat on $U$. By definition of flatness there exists a coherent ideal $\mathcal{I}$ of $A$ such that $\mathcal{I}\otimes F|U\rightarrow F$ is not injective. Let us take a coherent ideal $\mathcal{I}'$ of $\mathcal{O}_{X_K}$ that extends $\mathcal{I}$ (\cite{Ha} ex. II 5.15). Then $\mathcal{I}'\otimes F \rightarrow F$ is not injective. Therefore, since the functor $j$ is faithful, exact and compatible with tensor products the map $j(\mathcal{I}')\otimes \mathcal{F}=j(\mathcal{I}'\otimes F)\rightarrow j(F)=\mathcal{F}$ is not injective. So the functor 
$$-\otimes \mathcal{F}:\textrm{Coh}(K\otimes \mathcal{O}_{\hat{X}})\rightarrow \textrm{Coh}(K\otimes \mathcal{O}_{\hat{X}})$$
is not an exact functor and then $\mathcal{F}$ is not locally free.
\end{proof}

Now, as in \cite{Sh1} paragraph 5.2, we construct a functor 
$$\Phi : I_{conv}((\hat{X},M)/(\Spf(V),N))\longrightarrow I_{inf}((\hat{X},M)/(\Spf(V),N)).$$
We remark that the functor $\Phi$ and its restriction to locally free objects, $\tilde{\Phi}$, that we will mention below, can be constructed for $(\f{X},M)\rightarrow(\f{S},N)$, morphism of fine log formal schemes. \\
Let $(\f{U},\f{T},L, \phi)$ be an object of the infinitesimal site, define an enlargement $\Phi ^*(\f{T})=(\f{T},L,z:(\f{T}_0,L) \cong (\f{U}_0,M)\rightarrow (\hat{X},M))$; this is clearly an element of the log convergent site (the isomorphism $(\f{T}_0, M) \cong (\f{U}_0, L)$ follows from the fact that immersion $(\f{U}, M)\rightarrow (\f{T}, L)$ is nilpotent exact closed immersion.) Let us observe that $\Phi ^*(\hat{X})=(\hat{X}, M, z:(X_k, M)\rightarrow (\hat{X},M)).$ \\
If we have an isocrystal $\mathcal{E}$ on the log convergent site we define 
$$\Phi(\mathcal{E})_{\f{T}}=\mathcal{E}_{\Phi^*(\f{T})}.$$  
We have already seen in proposition \ref{primolibero} and \ref{secondolibero} that $\tilde{\Psi}$, the restriction of the functor $\Psi$ to the locally free objects 
$$\tilde{\Psi} :MIC((X_K, M)/(\Spf(V),N))^{lf} \longrightarrow I_{inf}((\hat{X},M)/(\Spf(V),N))^{lf},$$
is well-defined and is an equivalence of categories. Now we want to show that also $\tilde{\Phi}$, the restriction of the functor $\Phi$ to the locally free objects
$$\tilde{\Phi} : I_{conv}((\hat{X},M)/(\Spf(V),N))^{lf}\longrightarrow I_{inf}((\hat{X},M)/(\Spf(V),N))^{lf},$$
is well-defined. In particular we have the following lemma.
\begin{proposition}\label{secondoliberoperconv}
If $\mathcal{E}$ is a log convergent isocrystal, such that $\Phi(\mathcal{E})$ is a locally free log infinitesimal isocrystal, then $\mathcal{E}$ is also locally free. 
\end{proposition}
\begin{proof}
We can evaluate $\mathcal{E}$ at the enlargement $(\hat{X},M, z:(X_k,M)\rightarrow (\hat{X},M))$ and we find that $\mathcal{E}_{\hat{X}}=\mathcal{E}_{\Phi^*(\hat{X})}=\Phi(E)_{\hat{X}}$; so $\mathcal{E}_{\hat{X}}$ is a locally free isocoherent module. We are going to prove that $\mathcal{E}_{\f{T}}$ is a locally free isocoherent sheaf for every enlargement $\f{T}$. Now taking an enlargement $(\f{T},L,z:(\f{T}_0,L)\rightarrow (\hat{X},M))$, we have the following commutative diagram of log formal schemes
\[
\xymatrix{ 
\ (\f{T}_0,L)\ \ar[r]^{z} \ar[d]^{i} &\ (\hat{X},M)\  \ar[d]^{f} \\  
\ (\f{T},L)\  \ar[r]^{b}& \  (\Spf (V), N).  \\ 
}
\]  
As $f$ is formally log smooth, we know that \'etale locally on $\f{T}$ there is a morphism $c:(\f{T},L)\rightarrow (\hat{X},M)$ such that $c\circ i =z$ and $f \circ c=b$. \\
So let us consider an open for the \'etale topology: $\f{T}'$  formally \'etale over $\f{T}$ and we call again $c$ the morphism $(\f{T}',L)\rightarrow (\hat{X},M)$ induced by the diagram above, where as usual with $L$ we denote the pullback of $L$ by the \'etale map between $\f{T}$ and $\f{T}'$. We construct the enlargement $\f{T}'=(\f{T}',L,z:(\f{T}'_{0},L)\rightarrow(\hat{X},M))$; the morphism $c$ clearly extends to a morphism of enlargements. So by definition of isocrystal we have an isomorphism $c^{*}(\mathcal{E}_{\hat{X}})\cong \mathcal{E}_{\f{T'}}$, from which we know that $\mathcal{E}_{\f{T'}}$ is a locally free isocoherent module.\\
Then we consider the formally \'etale morphism $u:(\f{T}', L)\rightarrow (\f{T},L)$ that extends to a morphism of enlargements; again by definition of isocrystal we have $u^{*}(\mathcal{E}_{\f{T}})\cong \mathcal{E}_{\f{T'}}$. We know that $\mathcal{E}_{\f{T'}}$ is a locally free isocoherent module and we want to prove that $\mathcal{E}_{\f{T}}$ is locally free: we can proceed as in the last part of proposition \ref{primolibero} and conclude.  
\end{proof}
From the definition of the functor $\Phi$ we see also the viceversa of proposition \ref{secondoliberoperconv}; indeed if $\mathcal{E}$ is a locally free log convergent isocrystal, then for every element $\mathscr{T}$ on the log infinitesimal site 
$$\Phi(\mathcal{E})_{\mathscr{T}}=\mathcal{E}_{\Phi^*(\mathscr{T})},$$
by definition of $\Phi$.\\
So we can compose the functors $\tilde{\Phi}$ and $\tilde{\Psi}^{-1}$ and obtain a well defined functor  
$$i:I_{conv}((\hat{X},M)/(\Spf(V),N))^{lf}\rightarrow MIC((X_K, M)/(\Spf(V),N))^{lf}.$$
Our next goal is to prove that $\tilde{\Phi}$ is fully faithful, showing first that this can be proved  \'etale locally.  
\begin{proposition}\label{discesaperi}
If $\tilde{\Phi}:I_{conv}(\hat{X}/\Spf(V))^{log,lf}\rightarrow I_{inf}(\hat{X}/\Spf(V))^{log,lf}$ is fully faithful \'etale locally on $\hat{X}$, then $\tilde{\Phi}:I_{conv}(\hat{X}/\Spf(V))^{log,lf}\rightarrow I_{inf}(\hat{X}/\Spf(V))^{log,lf}$ is fully faithful.
\end{proposition}
\begin{proof}
We suppose that $\coprod_j \hat{X}_j =\hat{X}^e\rightarrow \hat{X}$ is an \'etale covering of $\hat X$. If $\hat{X}'=\hat{X}^e\times_{\hat{X}} \hat{X}^e$ and $\hat{X}''=\hat{X}^e\times_{\hat{X}} \hat{X}^e\times_{\hat{X}} \hat{X}^e$, we have the following diagram:
\[
\xymatrix{ 
\ I_{conv}(\hat{X}/\Spf(V))^{log, lf}\ \ar[r]^{\tilde{\Phi}} \ar[d] &\ I_{inf}(\hat{X}/\Spf(V))^{log, lf}\  \ar[d] \\  
\ I_{conv}(\hat{X}^e/\Spf(V))^{log, lf}\ \ar[r]^{\tilde{\Phi}^{e}} \ar[d] \ar@<1ex>[d] &\ I_{inf}(\hat{X}^e/\Spf(V))^{log, lf}\  \ar[d] \ar@<1ex>[d]\\  
\ I_{conv}(\hat{X}'/\Spf(V))^{log, lf}\ \ar[r]^{\tilde{\Phi}'} \ar[d] \ar@<1ex>[d] \ar@<-1ex>[d]&\ I_{inf}(\hat{X}'/\Spf(V))^{log, lf}\  \ar[d] \ar@<1ex>[d]\ar@<-1ex>[d]\\  
\ I_{conv}(\hat{X}''/\Spf(V))^{log, lf}\ \ar[r]^{\tilde{\Phi}''} &\ I_{inf}(\hat{X}''/\Spf(V))^{log, lf} \\ 
}
\]  
where the vertical arrows are induced by the following morphisms:
$$\hat{X}^e\rightarrow \hat{X},$$
$$p_i :\hat{X}'=\hat{X}^e\times_{\hat{X}} \hat{X}^e\rightarrow \hat{X}^e \,\,\,\textrm{for} \,\,\,i=1,2,$$
$$p_{i,j}:\hat{X}''=\hat{X}^e\times_{\hat{X}} \hat{X}^e\times_{\hat{X}} \hat{X}^e\rightarrow \hat{X}^e\times_{\hat{X}}\hat{X}^e\,\,\, \,\,\,\textrm{for} 1\leq i < j \leq 3$$
where the first is the \'etale morphism defining the \'etale cover, $p_i$ and $p_{i,j}$ are the natural projections.\\
Thanks to \'etale descent for log convergent isocrystals (\cite{Sh1} remark 5.1.7), giving $\mathcal{E}$, a locally free log convergent isocrystal on $\hat{X}$ is the same as giving $\mathcal{E}^{e}$, a locally free log convergent isocrystal on $\hat{X}^e$ and $\alpha_{\mathcal{E}}$ isomorphism between $p_2^*\mathcal{E}^{e}\rightarrow p_1^* \mathcal{E}^{e}$, compatible with the usual cocycle conditions.\\
A morphism $f$ from $\mathcal{E}$ to $\mathcal{F}$ is the same as a morphism $f^{e}$ from $\mathcal{E}^{e}$ and $\mathcal{F}^{e}$ that satisfies 
$\alpha_{\mathcal{F}}\circ p^*_2 f^e = p_1^*f^e \circ \alpha_{\mathcal{E}}$ and the compatibility conditions given by cocycle conditions on $\hat{X}''$. \\
By hypothesis $\tilde{\Phi}^{e}$, $\tilde{\Phi}'$ and $\tilde{\Phi}''$ are fully faithful, so $f^{e}$ induces a unique morphism $\tilde{\Phi}^{e}(f^e)$ between $\tilde{\Phi}^{e}(\mathcal{E}^{e})\rightarrow \tilde{\Phi}^{e}(\mathcal{F}^{e})$ that satisfies the same compatibility conditions on $\hat{X}'$ and $\hat{X}''$. Moreover this association is surjective.\\
Using \'etale descent for log infinitesimal isocrystals proven in \cite{Sh1} remark 3.2.20 we can descend $\tilde{\Phi}^e(f^e)$ to a morphism which coincides with $\tilde{\Phi}(f)$ between $\tilde{\Phi}(\mathcal{E})$ and $\tilde{\Phi}(\mathcal{F})$. 
\end{proof}

Before proving the full faithfulness \'etale locally we recall the construction of the universal enlargement and of convergent stratifications given by Shiho in \cite{Sh1} paragraph 5.2. \\
Let $\ell:(\mathscr{X},M)\hookrightarrow (\mathscr{Y},M')$ be an exact closed immersion of $p$-adic log formal $V$-schemes over $(\f{S},N)$ defined by a sheaf of ideals $\mathcal{I}.$ 
We need a more general notion of enlargement.
\begin{definition} 
We say that the quadruple $(\mathscr{T},L,z,g)$ is an enlargement of $(\mathscr{X},M)$ in $(\mathscr{Y},M')$ if $(\mathscr{T},L,z)$ is an enlargement of $(\mathscr{X},M)$ and $g$ is an $(\mathscr{S},N)$ morphism $(\mathscr{T},L)\rightarrow (\mathscr{Y},M')$ such that the following diagram is commutative
\[
\xymatrix{
(\f{T}_0,L) \ar[d]^{z} \ar[r]^{g_0}  & (\f{Y}_0,M')  \ar[d] \\  
(\f{X},M)\ar@{^{(}->}[r]^{\ell} &   (\f{Y}, M')   \\
}
\]
\end{definition}
We call $\mathcal{B}_{n,\mathscr{X}}(\mathscr{Y})$ the formal blow up of $(\mathscr{Y}, M')$ with respect to $p\mathcal{O}_{\mathscr{Y}}+\mathcal{I}^{n+1}$  and we denote by $T_{n,\mathscr{X}}(\mathscr{Y})$ the open of $\mathcal{B}_{n,\mathscr{X}}(\mathscr{Y})$ defined by $p$:
$$T_{n,\mathscr{X}}(\mathscr{Y}):=\{x \in \mathcal{B}_{n,\mathscr{X}}(\mathscr{Y})|\,\,(p\mathcal{O}_{\mathscr{Y}}+\mathcal{I}^{n+1})\mathcal{O}_{\mathcal{B}_{n,\mathscr{X}}(\mathscr{Y}),x}=p\mathcal{O}_{\mathcal{B}_{n,\mathscr{X}}(\mathscr{Y}),x}\}.$$
We put on $T_{n,\mathscr{X}}(\mathscr{Y})$ the log structure $L_{n,\mathscr{X}}(\mathscr{Y})$ induced by the pull-back of $M'.$
\begin{remark}
We can see from \cite{Og} proposition 2.3 that if $\mathscr{Y}=\Spf (A)$ and $\mathcal{I}^{n+1}=(a_1,\dots a_m)$ then $T_{n,\mathscr{X}}(\mathscr{Y})=\Spf (A\{t_1, \dots t_m\}/(pt_1-a_1,\dots,pt_m-a_m)$ modulo $p$-torsion).
For a local description of the formal blow up see \cite{Bo} paragraph 2.6.
\end{remark}
The map $(T_{n,\mathscr{Xz}}(\mathscr{Y}))_0\rightarrow (\f{Y})_0$ factors through $\f{X}_0$, so that we can equip the pair $(T_{n,\mathscr{X}}(\mathscr{Y}),L_{n,\mathscr{X}}(\mathscr{Y})) $ with two maps $(z_n,t_n)$ in such a way that the quadruple $((T_{n,\mathscr{X}}(\mathscr{Y}),L_{n,\mathscr{X}}(\mathscr{Y}),z_n,t_n)$ is an enlargement of $(\f{X},M)$ in $(\f{Y},M')$ and the set $\{(T_{n,\mathscr{X}}(\mathscr{Y}),L_{n,\mathscr{X}}(\mathscr{Y}), z_n, t_n)\}_{n\in \mathbb{N}}$ is an inductive system of enlargements that is universal in the sense that for every enlargement $(T',L',z',t')$ of $(\mathscr{X},M)$ in $(\mathscr{Y},M')$ there exists a unique morphism to the inductive system given by $\{(T_{n,\mathscr{X}}(\mathscr{Y}),L_{n,\mathscr{X}}(\mathscr{Y}), z_n, t_n)\}_{n\in \mathbb{N}}$ (proposition 5.2.4 of \cite{Sh1}).\\

The fiber product of fine log formal schemes is a log formal scheme that is not necessarily fine, but, thanks to proposition 2.1.6 of \cite{Sh1}, there is a functor $(-)^{int}$ that sends a log formal scheme to a fine log formal scheme that is right adjoint to the natural inclusion of fine formal schemes in the category of log formal schemes.
If $f:(\f{X},M)\rightarrow (\f{S},N)$ is a morphism of fine formal schemes, we denote the fiber product in the category of log formal schemes of $(\f{X},M)$ and $(\f{X},M)$ over $(\f{S},N)$ by $(\f{X},M)\times_{(\f{S},N)}(\f{X},M)$ (resp. the fiber product in the category of log formal schemes of $(\f{X},M)$ with itself three times over $(\f{S},N)$ with $(\f{X},M)\times_{(\f{S},N)}(\f{X},M)\times_{(\f{S},N)}(\f{X},M)$) and the fine log formal scheme associated to this with $((\f{X},M)\times_{(\f{S},N)}(\f{X},M))^{int}$ (resp. $((\f{X},M)\times_{(\f{S},N)}(\f{X},M)\times_{(\f{S},N)}(\f{X},M))^{int}$).\\
We want to construct a log formal scheme that  we will indicated by $(\f{X}(1),M(1))$ (resp. $(\f{X}(2),M(2))$) such that it factors the diagonal embedding $\Delta^{int} : (\f{X},M) \rightarrow ((\f{X},M)\times_{(\f{S},N)}(\f{X},M))^{int}$
(resp. $\Delta^{int} : (\f{X},M) \rightarrow ((\f{X},M)\times_{(\f{S},N)}(\f{X},M)\times_{(\f{S},N)}(\f{X},M))^{int}$) in an exact closed immersion $(\f{X},M)\hookrightarrow (\f{X}(1),M(1))$ (resp. the closed immersion $(\f{X},M)\hookrightarrow (\f{X}(2),M(2))$) followed by a formally log \'etale morphism: following \cite{Sh1} or \cite{Ka} proposition 4.10, we can do this if the morphism $f:(\f{X},M)\rightarrow (\f{S},N)$ has a global chart.\\
We indicate with $(P_{\f{X}}\rightarrow M, Q_V\rightarrow N, Q\rightarrow P )$ a chart of $f$, with $\alpha(1)$ (resp. $\alpha(2)$) the homomorphism induced by the map $P\oplus_Q P\rightarrow P$ (resp. $P\oplus_Q P \oplus_Q P \rightarrow P$) and with $R(1)$ the set $(\alpha(1)^{gp})^{-1}(P)$ (resp. with $R(2)$ the set $(\alpha(2)^{gp})^{-1}(P)$).\\
With this notation we define $\f{X}(1)=(\f{X}\times_{\f{S}}\f{X})\times_{\Spf (\mathbb{Z}_p\{P\oplus_Q P\})}\Spf({\mathbb{Z}_p\{R(1)\}})$ (resp. $\f{X}(2)=(\f{X}\times_{\f{S}}\f{X}\times_{\f{S}}\f{X})\times_{\Spf (\mathbb{Z}_p\{P\oplus_Q P\oplus_Q P\})}\Spf({\mathbb{Z}_p\{R(2)\}})$) equipped with the log structure $M(1)$ (resp $M(2)$) defined as the log structure induced by the canonical log structure on $\Spf({\mathbb{Z}_p}\{R(1)\})$ (resp. on  $\Spf({\mathbb{Z}_p}\{R(2)\})$).
Thanks to proposition 4.10 of \cite{Ka} we have that $(\f{X}(1),M(1))$ (resp. $(\f{X}(2),M(2))$) factors the diagonal embedding as we wanted.\\
Using the fact that $(\f{X},M)\hookrightarrow (\f{X}(i),M(i))$ are exact closed immersions of log formal schemes
for $i=1,2$, we define $\{(T_{\f{X},n}(\f{X}(i)), L_{\f{X},n}(\f{X}(i)),z_n(i),t_n(i))\}_{n\in \mathbb{N}}$ which is the universal system of enlargements associated to this closed immersions. For simplicity of notation we will denote by $(T_n(i),L_n(i))$ the $n$-th universal enlargement $(T_{\f{X},n}(\f{X}(i)), L_{\f{X},n}(\f{X}(i)),z_n(i),t_n(i))$.\\
The natural maps
$$p_i:(\f{X},M)\times_{(\f{S},N)}(\f{X},M)\rightarrow(\f{X},M)\;\;\;\;\;\;\;\;\;\textrm{for}\,\,\,i=1,2,$$
$$p_{i,j}:(\f{X},M)\times_{(\f{S},N)}(\f{X},N)\times_{(\f{S},N)}(\f{X},N)\rightarrow(\f{X},N)\times_{(\f{S},N)}(\f{X},N)\;\;\;\;\;\;$$
$$ \textrm{for}\;\; 1\leq i < j\leq 3$$
$$\Delta: (\f{X},M) \rightarrow (\f{X},M)\times_{(\f{S},N)}(\f{X},M)$$
induce compatible morphisms of enlargements:
$$q_{i;n}:(T_n(1),L_n(1))\rightarrow (\f{X},M)$$
$$q_{i,j;n}: (T_n(2),L_n(1))\rightarrow (T_n(1),L_n(1))$$
$$\Delta_n: (\f{X},M)\rightarrow (T_n(1),L_n(1)).$$
With the same notation we can give the following definition.
\begin{definition}\label{convstra}
A log convergent stratification on $(\f{X},M)$ is an isocoherent sheaf $\mathcal{E}_{\f{X}}$ on $\f{X}$  and a compatible family of isomorphisms 
$$\epsilon_{n}:q_{2;n}^*\mathcal{E}_{\f{X}}\rightarrow q_{1;n}^*\mathcal{E}_{\f{X}}$$
such that every $\epsilon_n$ satisfies 
$$\Delta_n^*(\epsilon_n)=\textrm{id};$$
$$q_{1,2;n}^*(\epsilon_n)\circ q_{2,3;n}^*(\epsilon_n)= q_{1,3;n}^*(\epsilon_n).$$
We denote the category of log convergent stratifications by $Str'((\f{X}, M)/(\f{S},N)).$
\end{definition}
As in the case of log infinitesimal isocrystals we can establish an equivalence of categories between log convergent stratifications and convergent log  isocrystals: this is the statement of proposition 5.2.6 of \cite{Sh1}.
\begin{proposition}\label{stratconvisoconv}
If the log formal scheme $(\f{X},M)$ is formally log smooth over the log formal scheme $(\Spf(V),N)$, then the category $I_{conv}((\f{X},M)/(\Spf(V),N))$ is equivalent to $Str'((\f{X},M)/(\Spf(V),N))$.
\end{proposition} 
 
Now that we have introduced all the machinery we can finish the proof of full faithfulness of the functor $\tilde{\Phi}$. We need the following result, whose proof is essentially the same as the one of proposition 5.2.9 of \cite{Sh1}.\\

\begin{theorem}\label{ff}
The functor 
$$\tilde{\Phi} : I_{conv}((\hat{X},M)/(\Spf(V),N))^{lf}\longrightarrow I_{inf}((\hat{X},M)/(\Spf(V),N))^{lf}$$
is fully faithful \'etale locally.
\end{theorem} 
\begin{proof}
Since fine log formal schemes have charts \'etale locally and the statement is of \'etale local nature, we can suppose that the morphism $f:(\hat{X},M)\rightarrow (\Spf(V),N)$ has a chart globally.
For $\mathcal{E}$ and $\mathcal{F}$ in $I_{conv}((\hat{X},M)/(\Spf(V),N))^{lf}$, we can define 
$\mathcal{H}om(\mathcal{E}, \mathcal{F})$ $\in$ $I_{conv}((\hat{X},M)/(\Spf(V),N))^{lf}$ by
$$\mathcal{H}om(\mathcal{E}, \mathcal{F})_{\f{T}}=\mathcal{H}om(\mathcal{E}_{\f{T}}, \mathcal{F}_{\f{T}}).$$
The global sections of $\mathcal{H}om(\mathcal{E}, \mathcal{F})$ are isomorphic to Hom$(\mathcal{E},\mathcal{F})$.
The same holds for $I_{inf}((\hat{X},M)/(\Spf(V),N))^{lf}$.\\
So we are reduced to prove that there is an isomorphism 
$$H^0((\hat{X},M)/(\Spf(V),N)_{conv},\mathcal{E})\longrightarrow H^{0}((\hat{X},M)/(\Spf(V),N)_{inf},\tilde{\Phi}(\mathcal{E}))$$
for every $\mathcal{E}$ in  $I_{conv}((\hat{X},M)/(\Spf(V),N))^{lf}$.\\
As we noticed before the morphism $f:(\hat{X}, M)\rightarrow (\Spf(V),N)$ has a chart globally and so we can construct the scheme $(\hat{X}(1),(M)(1))$ that we described above.\\
The equivalence in proposition \ref{stratconvisoconv} associates to $\mathcal{E}$ $\in$ $I_{conv}((\hat{X},M)/(\Spf(V),N))^{lf}$ a log convergent stratification  $(\mathcal{E}_{\hat{X}},\epsilon_n)$ given by a locally free isocoherent sheaf $\mathcal{E}_{\hat{X}}$ on $\hat{X}$ and isomorphisms
$$\epsilon_n:(K\otimes \mathcal{O}_{T_n(1)})\otimes \mathcal{E}_{\hat{X}}\longrightarrow \mathcal{E}_{\hat{X}}\otimes (K\otimes \mathcal{O}_{T_n(1)})$$
for all $n$.\\
So the set $H^0((\hat{X},M)/(\Spf(V),N)_{conv},\mathcal{E})$ can be characterized in terms of log convergent stratifications as follows:
$$H^0((\hat{X},M)/(\Spf(V),N)_{conv},\mathcal{E})=\{\,\,e \,\,\in \,\, \Gamma(\hat{X},\mathcal{E}_{\hat{X}})|\,\,\,\epsilon_n(1\otimes e)=e\otimes 1\,\,\forall n\}.$$ 
To see this let us remember how Shiho (\cite{Sh1} prop 5.2.6) associates to a log convergent isocrystal $\mathcal{E}$ a convergent stratification $(\mathcal{E}_{\hat{X}}, \epsilon_n)$: to define the isomorphism $\epsilon_n:q_{2,n}^{*}\mathcal{E}_{\hat{X}}\rightarrow q_{1,n}^{*}\mathcal{E}_{\hat{X}}$, he uses the fact that $q_{2,n}^{*}\mathcal{E}_{\hat{X}}\cong \mathcal{E}_{T_n(1)}\cong q_{1,n}^{*}\mathcal{E}_{\hat{X}}$, \emph{i.e.} the fact that $\mathcal{E}$ is an isocrystal. A global section of the log convergent isocrystal $\mathcal{E}$ is a collection of $(\varphi_{\f{T}})_{\mathscr{T}\in (\hat{X}/\Spf(V))^{log}_{conv}}$ with $\varphi_{\f{T}}\in \mathcal{E}_{\f{T}}(\f{T})$, with the property that if there is a morphism of enlargements $\f{T}\rightarrow \f{T}'$, then $\varphi_{\f{T}'}$ is sent to $\varphi_{\f{T}}.$ Hence if $(e_{\f{T}})_{\f{T}}$ is a global section and we send it to $e_{\hat{X}}\in \Gamma(\hat{X}, \mathcal{E}_{\hat{X}})$ then $e_{\hat{X}} $ is such that $\epsilon_n(1\otimes e_{\hat{X}})=e_{\hat{X}} \otimes 1$ for every $n$: this last condition is equivalent to say that a global section of $\mathcal{E}$ is compatible with $q_{i,n}:T_n(1)\rightarrow \hat{X}$ for $i=1,2.$ Moreover let us take $e_{\hat{X}}$ which verifies that  $\epsilon_n(1\otimes e_{\hat{X}})=e_{\hat{X}} \otimes 1$ for every $n$, this means than one can define $e_{T_n(1)}$ for every $n$ in a compatible way with respect to the map $q_{i,n}:T_n(1)\rightarrow \hat{X}$ for $i=1,2;$ using the universality of $T_n(1)$ one can construct a global section of $\mathcal{E}$ (look again at proof of proposition 5.2.6 of \cite{Sh1}). \\  
Let $J$ be the sheaf of ideals that defines the closed immersion $ \hat{X}\hookrightarrow \hat{X}(1)$; we denote by $\mathcal{O}_{\hat{X}(1)^{an}}$ the sheaf $\varprojlim_m K\otimes \mathcal{O}_{\hat{X}(1)}/J^m$. By proposition 5.2.7 (2) of \cite{Sh1} we know that there is an injective map $K\otimes \mathcal{O}_{T_n(1)}\rightarrow \mathcal{O}_{\hat{X}(1)^{\textrm{an}}}$. If we tensor the isomorphisms $\epsilon_n$ of the convergent stratification $(\mathcal{E}_{\hat{X}},\epsilon_n)$ with this map we obtain a map
$$\epsilon': \mathcal{O}_{\hat{X}(1)^{\textrm{an}}}\otimes \mathcal{E}_{\hat{X}}\longrightarrow \mathcal{E}_{\hat{X}}\otimes \mathcal{O}_{\hat{X}(1)^{\textrm{an}}}$$
that coincides with the limit of the isomorphisms of the stratification induced by $\tilde{\Phi}(\mathcal{E})$ through the equivalence of categories 
$$I_{inf}((\hat{X},M)/(\Spf(V),N))\cong \widehat{Str}((\hat{X},M)/(\Spf(V),N)).$$
So we can characterize the set $H^0((\hat{X},M)/(\Spf(V),N)_{inf}$ as follows
$$H^0((\hat{X},M)/(\Spf(V),N)_{inf}, \tilde{\Phi}(\mathcal{E}))=\{\,\,e \,\,\in \,\, \Gamma(\hat{X},\mathcal{E}_{\hat{X}})|\,\,\,\epsilon '(1\otimes e)=e\otimes 1\}.$$ 
This means that the claim is reduced to prove that the following diagram
\begin{equation}\label{primotriangolo}
\xymatrix{ 
\mathcal{O}_{\hat{X}(1)^{\textrm{an}}}\otimes \mathcal{E}_{\hat{X}} \ar[rr]^{\epsilon '}  &  & 
  \mathcal{E}_{\hat{X}}\otimes \mathcal{O}_{\hat{X}(1)^{\textrm{an}}} \\  
 & \mathcal{E}_{\hat{X}}
 \ar[ul] \ar[ur] &   \\
}
\end{equation}
 is commutative if and only if this is commutative
\begin{equation}\label{secondotriangolo}
\xymatrix{ 
(K\otimes \mathcal{O}_{T_n(1)})\otimes \mathcal{E}_{\hat{X}} \ar[rr]^{\epsilon_n}  &  & 
  \mathcal{E}_{\hat{X}}\otimes (K\otimes \mathcal{O}_{T_n(1)}) \\  
 & \mathcal{E}_{\hat{X}}
 \ar[ul] \ar[ur] &  . \\
}
\end{equation}
Knowing that the following diagram is commutative
\begin{equation}\label{quadrato}
\xymatrix{\mathcal{O}_{\hat{X}(1)^{\textrm{an}}}\otimes \mathcal{E}_{\hat{X}} \ar[r]^{\epsilon'}  & \mathcal{E}_{\hat{X}}\otimes \mathcal{O}_{\hat{X}(1)^{\textrm{an}}} \\  
(K\otimes \mathcal{O}_{T_n(1)})\otimes \mathcal{E}_{\hat{X}} \ar[r]^{\epsilon_n}  \ar[u] & 
  \mathcal{E}_{\hat{X}}\otimes (K\otimes \mathcal{O}_{T_n(1)}) \ar[u]\\  
}
\end{equation}
and putting together (\ref{secondotriangolo}) and (\ref{quadrato}), we can conclude that if (\ref{secondotriangolo}) is commutative then (\ref{primotriangolo}) is commutative.\\
Let us suppose instead that (\ref{primotriangolo}) is commutative: then using the fact that $\mathcal{E}_{\hat{X}}$ is flat 
and that the map $K\otimes \mathcal{O}_{T_n(1)}\rightarrow \mathcal{O}_{\hat{X}(1)^{\textrm{an}}}$ is injective we can conclude that also (\ref{secondotriangolo}) is commutative. 
\end{proof}
Shiho in \cite{Sh1} proposition 5.2.9 proves that the functor $\Phi$ is fully faithful \'etale locally when it is restricted to the nilpotent part of $I_{conv}((\hat{X},M)/(\Spf(V),N))$ and $I_{inf}((\hat{X},M)/(\Spf(V),N))$. Our proof is essentially the same, because the key property of nilpotent objects used in Shiho's proof is that the nilpotent isocrystals are locally free.\\
Putting together theorem \ref{ff} and proposition \ref{discesaperi} we obtain the following 
\begin{theorem}
The functor $\tilde{\Phi}$ is fully faithful.
\end{theorem}

\section{Characterizations of log convergent isocrystals in terms of stratifications}\label{Characterizations of log convergent isocrystals in terms of stratifications}
We want to describe the essential image of the functor $\tilde{\Phi}$. As for the case of the proof of full faithfulness it will be enough to describe it \'etale locally, then, with descent argument we can conclude as in proposition \ref{discesaperi}. So we can suppose that $(\hat{X},M)\rightarrow(\Spf(V),N)$ has a chart globally.\\
To avoid log differential calculus we will prove that we can restrict to the case of trivial log structures. We need the following lemma whose proof is essentially the same as proposition 5.2.11 of \cite{Sh1}.
\begin{proposition}\label{U denso}
Let  $f:(\hat{X},M) \rightarrow (\Spf(V),N)$ be a formally smooth morphism of fine log schemes that admits a chart. 
Let $U$ be a dense open subset of $\hat{X}$ and set $j$ the open immersion $j:U\hookrightarrow \hat{X}$. If $\mathcal{E}_i \in I_{inf}((\hat{X},M)/(\Spf(V),N))^{lf}$  there exists $\mathcal{E}_c$ in the category $I_{conv}((U,M)/(\Spf(V),N))^{lf}$ such that $\tilde{\Phi}(\mathcal{E}_c)=j^*\mathcal{E}_i$, then $\mathcal{E}_i\in$ $\tilde{\Phi}( I_{conv}((\hat{X},M)/(\Spf(V),N))^{lf}).$ 
\end{proposition}
\begin{proof}
Let us consider the sheaf 
$$\mathcal{F}=z_{n*}\mathcal{H}om((K\otimes \mathcal{O}_{T_{n}(1)})\otimes \mathcal{E}_{i,\hat{X}},\mathcal{E}_{i,\hat{X}} \otimes (K\otimes \mathcal{O}_{T_{n}(1)})).$$
The following sequence is exact:
$$0\rightarrow \mathcal{F}\xrightarrow{a} (\mathcal{F}\otimes \mathcal{O}_{\hat{X}(1)^{an}})\oplus j_*j^*(\mathcal{F})\xrightarrow{b} j_*j^*(\mathcal{F}\otimes \mathcal{O}_{\hat{X}(1)^{an}}),$$
because $\mathcal{F}$ is a projective $z_{n*}(K\otimes \mathcal{O}_{T_{n}(1)})$-module, and the following sequence is exact by proposition 5.2.8 of \cite{Sh1} and \cite{Og}  lemma 2.14
$$0\rightarrow z_{n*}(K\otimes \mathcal{O}_{T_{n}(1)})\rightarrow  \mathcal{O}_{\hat{X}(1)^{an}}\oplus j_*j^*z_{n*}(K\otimes \mathcal{O}_{T_{n}(1)})\rightarrow j_*j^* \mathcal{O}_{\hat{X}(1)^{an}}.$$
Viewing $\mathcal{E}_c$ as a convergent stratification we have a map
$$\epsilon'_n: j_*j^*z_{n*}((K\otimes \mathcal{O}_{T_{n}(1)})\otimes j^*(\mathcal{E}_{i,\hat{X}}))\rightarrow j_*j^*z_{n*} (j^*(\mathcal{E}_{i,\hat{X}})\otimes(K\otimes \mathcal{O}_{T_{n}(1)}));$$ 
on the other hand the log infinitesimal isocrystal $\mathcal{E}_i$ induces a stratification 
$$\epsilon' :\mathcal{O}_{\hat{X}(1)^{an}}\otimes \mathcal{E}_{i,\hat{X}}\rightarrow \mathcal{E}_{i,\hat{X}}\otimes \mathcal{O}_{\hat{X}(1)^{an}}.$$
The pair $(\epsilon', \epsilon'_n)$ lies in Ker$(b)$, because $b(\epsilon', \epsilon'_n)=j^{*}j_{*}\epsilon'-\epsilon'_n$ where we consider $\epsilon'_n$ as an element of   $j_*j^*(\mathcal{F}\otimes \mathcal{O}_{\hat{X}(1)^{an}})$ using the injective map $K\otimes \mathcal{O}_{T_n(1)}\rightarrow \mathcal{O}_{\hat{X}(1)^{\textrm{an}}}$.  Then there exists an $\epsilon_n$ $\in$ $\mathcal{F}$ such that $a(\epsilon_n)=(\epsilon', \epsilon'_n)$ that defines a convergent stratification, i.e. a convergent isocrystal $\bar{\mathcal{E}_c}$. Moreover one can verify that $\tilde{\Phi}(\bar{\mathcal{E}_c})=\mathcal{E}_i $.   
\end{proof}
Now we want to apply proposition $\ref{U denso}$ choosing as $U$ the subset 
$$\hat{X}_{\hat{f}-triv}=\{x\in \hat{X}\;|\;(\hat{f}^{*}N)_{\bar{x}}=(M)_{\bar{x}}\}.$$
Let us prove that it is open and dense in $\hat{X}$. Clearly $\hat{X}_{\hat{f}-triv}$ is homeomorphic to $X_{k,f_k-triv}$ so it will be sufficient to prove that $X_{k,f_k-triv}$ is open dense in $X_k$. But this follows from proposition 2.3.2 of \cite{Sh1} because the special fiber is reduced.\\
Now applying proposition \ref{U denso} we can restrict ourselves to the case in which $\hat{f}^{*}N=M$. As Shiho notices the hypothesis $\hat{f}^{*}N=M$ gives an equivalence of categories 
$$I_{conv}((\hat{X},M)/(\Spf(V),N))\cong I_{conv}((\hat{X},\textrm{triv})/(\Spf(V),\textrm{triv})),$$
where the notation triv means that the log structure is trivial and 
$$I_{inf}((\hat{X},M)/(\Spf(V),N))\cong I_{inf}((\hat{X},\textrm{triv})/(\Spf(V),\textrm{triv})).$$
So we are reduce to the case of trivial log structures, as we wanted. We will characterize the essential image using certain type of stratification that we call special.
\begin{definition}\label{special}
Let $(E,\epsilon_n)$ be an object of $\widehat{Str}(\hat{X}/\Spf(V))$ and let $\tilde {E}$ be a coherent $p$-torsion-free $\mathcal{O}_{\hat{X}}$-module such that $K\otimes \tilde{E}=E$; we say that $(E,\epsilon_n)$ is special if there exists a sequence of integers $k(n)$ for $n \in \mathbb{N}$ such that:
\begin{itemize}
\item[(i)]$k(n)=o(n)$ for $n\rightarrow \infty$,
\item[(ii)]the restriction of the map $p^{k(n)}\epsilon_n$ to $p_{2,n}^{*}(\tilde{E})$ has image contained in $p_{1,n}^{*}(\tilde{E})$ and the restriction of the map $p^{k(n)}\epsilon_n^{-1}$ to $p_{1,n}^{*}(\tilde{E})$ has image contained in $p_{2,n}^{*}(\tilde{E})$.
\end{itemize}
\end{definition}
This definition is a small modification of definition of special stratification given by Shiho (\cite{Sh1} Definition 5.2.12). Our definition of special is weaker then Shiho's definition: every special object in the sense of Shiho is special in our sense and allows us to characterize the essential image of the functor $\tilde{\Phi}$.\\
Let us see now that the definition of special is well-posed.
\begin{proposition}
Definition \ref{special} is independent on the choice of the $p$-torsion-free sheaf $\tilde{E}$.
\end{proposition}
\begin{proof}
Suppose that $(E,\epsilon_n)$ is special and that the conditions in definition \ref{special} are verified for a given coherent $p$-torsionfree $\mathcal{O}_{\hat{X}}$-module $\tilde{E}$ such that $K\otimes \tilde{E}=E$. We take an other $p$-torsion-free $\mathcal{O}_{\hat{X}}$-module $\tilde{F}$ such that $K\otimes \tilde{F}=E$ and we want to prove that the same conditions are verified. From \cite{Og} proposition 1.2 we know the following isomorphisms
$$K\otimes \mathrm{Hom}_{\mathcal{O}_{\hat{X}}}(\tilde {E},\tilde{F})\cong\mathrm{Hom}_{K\otimes \mathcal{O}_{\hat{X}}}(K\otimes \tilde{E},K\otimes \tilde{F})\cong \mathrm{End}_{K\otimes \mathcal{O}_{\hat{X}}}(E).$$ 
If we take the identity as endomorphism of $E$ we know that there exists a power of $p$, say $p^{a}$, such that the multiplication by $p^a$ is a morphism from $\tilde{E}$ to $\tilde{F}$; moreover this morphism is injective because $\tilde{F}$ is $p$-torsion-free.\\
In the same way we can prove that there exists a $b$ such that the multiplication by $p^b$ is an injective morphism between $\tilde{F}$ and $\tilde{E}$.\\
If we consider now the morphism $p^{k(n)+a+b}\epsilon_n$ then we have that the restriction of this to $p_{2,n}^{*}(\tilde{F})$ goes to $p_{1,n}^{*}(\tilde{F})$. \\
Arguing analogously for $\epsilon_{n}^{-1}$ we are done.  
\end{proof}
Let us see, first, that every object in the essential image of the functor $\Phi$ is special. \\
Following Shiho \cite{Sh1}, proposition 3.2.14 and proposition 5.2.6, both in the case of trivial log structures, we have the equivalences of categories
$$I_{conv}(\hat{X}/\Spf(V))={Str'}(\hat{X}/\Spf(V)),$$
$$I_{inf}(\hat{X}/\Spf(V))=\widehat{Str}(\hat{X}/\Spf(V)).$$
The functor $\Phi$ induces the functor 
$$\alpha:{Str'}(\hat{X}/\Spf(V))\rightarrow \widehat{Str}(\hat{X}/\Spf(V)).$$
From now on we may work locally.\\
We are reduced to the situation where $\hat{X}$ is formally smooth and we call $dx_1, \dots dx_l$ a basis of $\Omega^{1}_{\hat{X}/\Spf(V)}$. Let us call $\xi_1, \dots, \xi_l$ the dual basis of $dx_1,\dots dx_l$ where $\xi_j=1\otimes x_j-x_j \otimes 1$; we will indicate $(\xi_1, \dots ,\xi_l)$ with $\boldsymbol{\xi}$ and an $l$-ple of natural numbers $(\beta_1,\dots \beta_l)$ with $\boldsymbol{\beta}$. We will use multi-index notations denoting $\prod_j \xi_j^{\beta_j}$ by $\boldsymbol{\xi}^{\boldsymbol{\beta}}$ and $\beta_1+\dots +\beta_l$ by $|\boldsymbol{\beta}|.$\\
We call $\hat{X}^{n}$ the $n$-th infinitesimal neighborhood of $\hat{X}$ in $\hat{X}\times_{\Spf{V}}\hat{X}$. By a formal version of proposition 2.6 of \cite{BeOg} we know that $\mathcal{O}_{\hat{X}^{n}}$ is a free $\mathcal{O}_{\hat{X}}$-module generated by $\{\boldsymbol{\xi}^{\boldsymbol{\beta}}:|\boldsymbol{\beta}| \leq n\}$, so that we can write $\mathcal{O}_{\hat{X}^n}=\mathcal{O}_{\hat{X}}[[\boldsymbol{\xi}]]/(\boldsymbol{\xi}^{\boldsymbol{\beta}},\,\, |\boldsymbol{\beta}|=n+1)$.\\
We want to give a local description also for the universal system of enlargement $\{T_n\}_n$ of $\hat{X}$ in $\hat{X}\times_{\Spf{V}}\hat{X}$.  
By \cite{Og} remark 2.6.1 we know that $T_n$ is isomorphic to the $n$-th universal enlargement of $\hat{X}$ in $\hat{X}\times_{\Spf{V}}\hat{X}_{|\hat{X}}$, the formal completion of $\hat{X}\times_{\Spf{V}}\hat{X}$ along $\hat{X}$. Using this and the local description given in the proof of proposition 2.3 of \cite{Og}   
we can write $\mathcal{O}_{T_n}=\mathcal{O}_{\hat{X}}\{\boldsymbol{\xi},\boldsymbol{\xi}^{\boldsymbol{\beta}}/p\,\,\,(|\boldsymbol{\beta}|=n+1)\}$.
By universality of blowing up there exists a unique map $\psi_n$ such that the following diagram is commutative
\begin{equation}\label{phin}
\xymatrix{ 
\ar@{}|(.7)\cal[dr]&\hat{X}^n\ar@{-->}[dl]_{\psi_n}\ar[d]\\ 
T_n\ar[r]&\hat{X}\times_{\Spf{V}} \hat{X}.
} 
\end{equation} 
The functor $\alpha$ is induced by the pull back of $\psi_n$ and in local coordinates is given by
$$\mathcal{O}_{T_n}=\mathcal{O}_{\hat{X}}\{\boldsymbol{\xi},\boldsymbol{\xi}^{\boldsymbol{\beta}}/p\,\,\,(|\boldsymbol{\beta}|=n+1)  \}\rightarrow \mathcal{O}_{\hat{X}}[|\boldsymbol{\xi}|]/(\boldsymbol{\xi}^{\boldsymbol{\beta}},\,\, |\boldsymbol{\beta}|=n+1)$$
and sends $\frac{\boldsymbol{\xi}^{\beta}}{p}$ with $|\boldsymbol{\beta}|=n+1$ to 0. 
\begin{proposition}\label{immagineessenziale}
If $(E,\epsilon_n)$ is in $\widehat{Str}(\hat{X}/\Spf(V))$ and it is in the image of the functor $\alpha$, then it is special. 
\end{proposition}
\begin{proof}

Let $(E,\epsilon'_n)$ be an element of $Str'(\hat{X}/\Spf(V))$ such that $\alpha(E,\epsilon'_n)=(E,\epsilon_n)$, with
$$q_{2,n}^*E\xrightarrow{\;\epsilon'_n\;}q_{1,n}^{*}E,$$
where $q_{i,n}$ are the projections from $T_n$ to $\hat{X}$, that exist by definition of convergent stratification. \\
We note that $q_{1,n}^{*}E=E\otimes_{\mathcal{O}_{\hat{X}}}\mathcal{O}_{\hat{X}}[|\boldsymbol{\xi}|]\{\boldsymbol{\xi}^{\boldsymbol{\beta}}/p\,\, ,( |\boldsymbol{\beta}|=n+1)\}$ is embedded in $\prod_{\boldsymbol{\beta}} E\boldsymbol{\xi}^{\boldsymbol{\beta}}$ and so we have a map, that we call $\epsilon'$, which is the composition of 
$$E\rightarrow q_{2,n}^*E\xrightarrow{\;\epsilon'_n\;}q_{1,n}^{*}E\rightarrow \prod_{\boldsymbol{\beta}} E\boldsymbol{\xi}^{\boldsymbol{\beta}}.$$
Let us note that $\epsilon'$ does not depend from $n$, as a consequence of the fact that the maps $\epsilon'_n$ coming from the convergent stratification are compatible.
Using  the isomorphisms $\epsilon_n$ that define the stratification $(E,\epsilon_n)$ 
$$ p_{2,n}^*E\xrightarrow{\;\epsilon_n\;}p_{1,n}^*E=\prod_{|\boldsymbol{\beta}|\leq n}E\boldsymbol{\xi}^{\boldsymbol{\beta}},$$
where $p_{i,n}$ are the projections $p_{i,n}:X^n\rightarrow X$,  we can define a map
\begin{equation}\label{daEaq2}
E\rightarrow p_{2,n}^*E\xrightarrow{\;\epsilon_n\;}p_{1,n}^*E=\prod_{|\boldsymbol{\beta}|\leq n}E\boldsymbol{\xi}^{\boldsymbol{\beta}}.
\end{equation}
The fact that $\alpha(E,\epsilon'_n)=(E,\epsilon_n)$ means that, if we call $pr$ the projection
$$\prod_{\boldsymbol{\beta}} E\boldsymbol{\xi}^{\boldsymbol{\beta}}\xrightarrow{pr} \prod_{|\boldsymbol{\beta}|\leq n}E\boldsymbol{\xi}^{\boldsymbol{\beta}},$$
the map in (\ref{daEaq2}) coincides with the map
$$E\xrightarrow{\;\epsilon'\;}\prod_{\boldsymbol{\beta}} E \boldsymbol{\xi}^{\boldsymbol{\beta}}\xrightarrow{\;pr\;}\prod_{|\boldsymbol{\beta}|\leq n}E\boldsymbol{\xi}^{\boldsymbol{\beta}}.$$
Let $\tilde{E}$ be $p$-torsion-free $\mathcal{O}_{\hat{X}}$-module such that $K\otimes \tilde{E}=E$; so $q_{1,n}^*(\tilde{E})=\tilde{E}\otimes\mathcal{O}_{\hat{X}}[|\boldsymbol{\xi}|]\{\boldsymbol{\xi}^{\boldsymbol{\beta}}/p\,\,\,(|\boldsymbol{\beta}|=n+1) \}$ and this is embedded in $\prod_{\boldsymbol{\beta}} \tilde{E}\boldsymbol{\xi}^{\boldsymbol{\beta}}/p^{\left\lfloor \frac{|{\boldsymbol{\beta}}|}{n+1}\right\rfloor}$, \emph{i.e.} 
\begin{equation}
\xymatrix
{
q_{1,n}^*(E)=E\otimes\mathcal{O}_{\hat{X}}[|\boldsymbol{\xi}|]\{\boldsymbol{\xi}^{\boldsymbol{\beta}}/p\,\,\,(|\boldsymbol{\beta}|=n+1) \} \ar@{^(->}[r]&\prod_{\boldsymbol{\beta}} E\boldsymbol{\xi}^{\boldsymbol{\beta}}\\
q_{1,n}^*(\tilde{E})=\tilde{E}\otimes\mathcal{O}_{\hat{X}}[|\boldsymbol{\xi}|]\{\boldsymbol{\xi}^{\boldsymbol{\beta}}/p\,\,\,(|\boldsymbol{\beta}|=n+1) \}\ar@{^(->}[r]\ar@{^(->}[u]
& \prod_{\boldsymbol{\beta}} \tilde{E}\boldsymbol{\xi}^{\boldsymbol{\beta}}/p^{\left\lfloor \frac{|{\boldsymbol{\beta}}|}{n+1}\right\rfloor}\ar@{^(->}[u].\\
}
\end{equation}
The $\mathcal{O}_{\hat{X}}$-module $\tilde{E}$ is finitely generated and let $e_1, \dots, e_l$ be a set of generators; then for every $i=1, \dots l $ there exists $a_i$ such that
$$p^{a_i}\epsilon'(e_i)\subset \prod_{\boldsymbol{\beta}} \tilde{E}\frac{\boldsymbol{\xi}^{\boldsymbol{\beta}}}{p^{\left\lfloor \frac{|\boldsymbol{\beta}|}{n+1} \right\rfloor}}$$   
Thus there exists $a=:\mathrm{max}_{1\leq i\leq l}a_{i}$ in $\mathbb{N}$ such that 
$$p^a\epsilon'(\tilde{E})\subset \prod_{\boldsymbol{\beta}} \tilde{E}\frac{\boldsymbol{\xi}^{\boldsymbol{\beta}}}{p^{\left\lfloor \frac{|\boldsymbol{\beta}|}{n+1} \right\rfloor}}$$
and, if we call $\pi_{\boldsymbol{\beta}}$ the projection $\prod_{\boldsymbol{\beta}} E\boldsymbol{\xi}^{\boldsymbol{\beta}}\rightarrow E$, then 
$$p^{a+{\left\lfloor \frac{|\boldsymbol{\beta}|}{n+1} \right\rfloor}}\pi_{\boldsymbol{\beta}} \circ \epsilon'(\tilde{E})\subset \tilde{E};$$
therefore there exists a sequence $b_n(\boldsymbol{\beta})$ that tends to infinity when $|\boldsymbol{\beta}|$ goes to infinity such that 
\begin{equation}\label{conbn}
p^{\left\lfloor \frac{|\boldsymbol{\beta}|}{n} \right\rfloor}\pi_{\boldsymbol{\beta}} \circ \epsilon'(\tilde{E})\subset p^{b_n(\boldsymbol{\beta})}\tilde{E}.
\end{equation}
If we define now 
$$a(k):=\mathrm{min}\{a\in \mathbb{N}|\;p^a\pi_{\boldsymbol{\beta}}\circ \epsilon'(\tilde{E})\subset \tilde{E}\; \mathrm{for}\,\,\mathrm{all} \; \boldsymbol{\beta}\,\,\,\textrm{such that }\,\, |\boldsymbol{\beta}|\leq k\},$$
then $p^{a(k)}\epsilon_k(\tilde{E})\subset \prod_{|\boldsymbol{\beta}|\leq k}\tilde E \boldsymbol{\xi}^{|\boldsymbol{\beta}|}, $
which means that $p^{a(k)}\epsilon_k$ sends $p_{2,k}^*(\tilde{E})$ into $p_{1,k}^*(\tilde{E})$. 
So we are left to prove that $a(k)=o(k)$ for $k\rightarrow \infty$ .\\ 
We notice, from the definition of $a(k)$, that $a(k)$ is a non decreasing sequence.
If $a(k)$ is bounded, then we are done. Arguing by contradiction, then $a(k)\rightarrow \infty$ for $k\rightarrow \infty$. This means that there exists a sequence $\{k_i\}_i$ such that 
\begin{equation*}
\begin{split}
0<a(k_1)=a(k_1+1)=\dots =a(k_2-1)<\\
a(k_2)=a(k_2+1)=\dots =a(k_3-1)<\\
\dots\\
a(k_i)=a(k_i+1)=\dots =a(k_{i+1}-1)<\\
a(k_{i+1})=a(k_{i+1}+1)=\dots=a(k_{i+2}-1)<\dots .
\end{split}
\end{equation*}
Then 
$$p^{a(k_i)}\pi_{\boldsymbol{\beta}}\circ \epsilon '(\tilde{E})\subseteq \tilde{E},$$
for every $\boldsymbol{\beta}$ such that $|\boldsymbol{\beta}|\leq k_i$.
Let us prove that  
\begin{equation}\label{nonincl}
p^{a(k_i)}\pi_{\boldsymbol{\beta}}\circ \epsilon '(\tilde{E}) \nsubseteq p\tilde{E}
\end{equation}
for some $\boldsymbol{\beta}$ with $|\boldsymbol{\beta}|=k_i.$
Let us suppose that this is not true; this means that 
\begin{equation}\label{-1}
p^{a(k_i)-1}\pi_{\boldsymbol{\beta}}\circ \epsilon '(\tilde{E}) \subseteq \tilde{E}
\end{equation}
for every $\boldsymbol{\beta}$ such that $|\boldsymbol{\beta}|=k_i.$
Moreover for $\boldsymbol{\beta}$ with $|\boldsymbol{\beta}|<k_i$ we have 
$$p^{a(k_i)-1}\pi_{\boldsymbol{\beta}}\circ \epsilon '(\tilde{E}) \subseteq p^{a(k_{i-1})}\pi_{\boldsymbol{\beta}}\circ \epsilon' (\tilde{E})  \subseteq           \tilde{E}.$$
Hence we have 
$$p^{a(k_i)-1}\pi_{\boldsymbol{\beta}}\circ \epsilon '(\tilde{E}) \subseteq \tilde{E}$$
for all $\boldsymbol{\beta}$ with  $|\boldsymbol{\beta}|\leq k_i$
and this contradicts the definition of $a(k_i)$'s, so (\ref{nonincl}) holds.
If we now put together the formula (\ref{conbn}) with $|\boldsymbol{\beta}|=k_i$ and (\ref{nonincl}), we find that 
$$\lim_{i \to \infty}\left(\left\lfloor \frac{k_i}{n}\right\rfloor-a(k_i)\right)=\infty,$$
so that there exists $i_0$ such that 
$$0\leq \frac{a(k_i)}{k_i}\leq \frac{1}{n}$$
for all $i\geq i_0$. Then for any $k\geq k_{i_0}$ we can find some $k_i$ with $k_i\leq k \leq k_{i+1}-1$ and then 
$$0\leq \frac{a(k)}{k}\leq \frac{a(k_i)}{k_i}\leq \frac{1}{n}.$$ 
Hence we have that $\limsup_{k} \frac{a(k)}{k}\leq \frac{1}{n}.$
Since this is true for any $n$, we have that $a(k)=o(k).$

\end{proof}
Now we want to prove the converse: that every special object is in the image of functor $\alpha$. This is proven by Shiho in proposition 5.2.13 of \cite{Sh1} for his special objects, but the proof works also in our case.
\begin{proposition}
If $(E,\epsilon_n )$ is a special stratification on $\hat{X}$, then there exists $(E', \epsilon'_n)$ $\in$ $Str'(\hat{X},\Spf(V))$ such that $\alpha( (E', \epsilon'_n))=(E,\epsilon_n ).$
\end{proposition}

Hence we have a complete characterization in term of stratifications of the differential equations coming from a log-convergent isocrystal.\\ 

\begin{remark}
An example of the situation studied before appears in \cite{Ba} proposition 3.13 and corollary 3.14.
\end{remark}
We want to describe the property of being special in term of radius of convergence. 
We will use the formalism given in \cite{LS} in the local situation described before proposition \ref{immagineessenziale}. If we have an element $(E,\epsilon_n) \in \widehat{Str}(\hat{X}/\Spf(V))$, then we can take the inverse limit of the map that we considered in the proof of the proposition \ref{immagineessenziale}
$$E\rightarrow p_{2,n}^*E\xrightarrow{\;\epsilon_n\;}p_{1,n}^*E=\prod_{|\boldsymbol{\beta}|\leq n}E\boldsymbol{\xi}^{\boldsymbol{\beta}},$$
and we obtain
$$\theta:E\rightarrow \varprojlim p_{2,n}^*E\xrightarrow{\;\varprojlim \epsilon_n\;}\varprojlim p_{1,n}^*E=\prod_{\boldsymbol{\beta}}E\boldsymbol{\xi}^{\boldsymbol{\beta}}.$$  
According to definition 4.4.1 of \cite{LS} we can say that a section $s\in \Gamma(\hat{X},E)$ is $\eta$-convergent, with $\eta<1$ for the stratification $(E,\epsilon_n)$ if
$$\theta(s)\in \Gamma(\hat{X},E\otimes \mathcal{O}_{\hat{X}}\left\{\frac{\boldsymbol{\xi}}{\eta}\right\} ).$$
\begin{definition}\label{radius}
The radius of convergence of the section $s$ for the stratification $(E,\epsilon_n)$ is defined as
$$R(s)=\mathrm{sup}\{\eta| \;s \;\mathrm{is} \;\eta\mathrm{-convergent}\}.$$
And the radius of convergence of the stratification $(E,\epsilon_n)$ is 
$$R((E,\epsilon_n), \hat{X})=\mathrm{inf}_{s\in \Gamma(\hat{X},E)}R(s).$$
\end{definition}
\begin{proposition}
A stratification $(E, \epsilon_n)$ is special if and only if its radius of convergence is equal to 1.
\end{proposition}
\begin{proof}
We know that $\widehat{Str}(\hat{X}/\Spf(V))$ is equivalent to the category $\widehat{\mathrm{MIC}}(\hat{X}/\Spf(V))$. By lemma 5.2.15 of \cite{Sh1} we can write the map $\theta$ locally. Following the notation that we recalled before proposition \ref{immagineessenziale} we denote by $\{D_{\boldsymbol{\beta}}\}_{0\leq|\boldsymbol{\beta}|\leq n}$ the dual basis of $\{\boldsymbol{\xi}^{\boldsymbol{\beta}}\}_{0\leq|\boldsymbol{\beta}|\leq n}$ in $\mathcal{D}\textrm{iff}^n(\mathcal{O}_{\hat{X}},\mathcal{O}_{\hat{X}})$, the differential operators of order $\leq n$ and in particular we indicate with $D_{(i)}:=D_{(0\dots,1,\dots,0)}$ with $1$ at the $i$-th place .\\
With this notation
$$\theta(e)=\sum_{\boldsymbol{\beta}}\frac{1}{\boldsymbol{\beta}!}\nabla_{\boldsymbol{\beta}}(e)\otimes \xi^{\boldsymbol{\beta}}$$
with $\nabla_{\boldsymbol{\beta}}:=(\textrm{id}\otimes D_{(1)}\circ \nabla )^{\beta_1}\circ \dots \circ (\textrm{id}\otimes D_{(l)}\circ \nabla )^{\beta_l}.$
Given an $\tilde{E}$, as in the definition \ref{special}, then the fact of being special can be translated as follows: there exists a sequence of integers $a(n)$ such that $a(n)=o(n)$ for $n\rightarrow \infty$ and that 
$$p^{a(n)}\frac{\nabla_{\boldsymbol{\beta}}(e)}{\boldsymbol{\beta}!}\in \tilde{E}$$
for $e$ in $\tilde{E}$ and for any multi index $\boldsymbol{\beta}$ such that $|\boldsymbol{\beta}|\leq n$ . 
Let us see that the radius of convergence of a section $e$ $\in \tilde{E}$ is $1$, because $e \in \tilde{E}$ is $\eta$-convergent for every $\eta$, i.e. for every $\eta$
$$\theta(e)=\sum_{\boldsymbol{\beta}}\frac{1}{\boldsymbol{\beta}!}\nabla_{\boldsymbol{\beta}}(e)\otimes \boldsymbol{\xi}^{\boldsymbol{\beta}} \in \Gamma(\hat{X},E\otimes \mathcal{O}_{\hat{X}}\{\frac{\boldsymbol{\xi}}{\eta}\} ).$$
To prove it we have to show that if we denote by $\|\,\,\,\|$ the $p$-adic Banach norm on $E$ such that $\|\tilde{E}\|=1$
$$\|\frac{1}{\boldsymbol{\beta}!}\nabla_{\boldsymbol{\beta}}(e)\|\eta^{|\boldsymbol{\beta}|}\rightarrow 0, \forall \;\eta<1.$$
This is clearly true because, fixed an $n$, the following estimate holds:
$$\|\frac{1}{\boldsymbol{\beta}!}\nabla_{\boldsymbol{\beta}}(e)\|\eta^{|\boldsymbol{\beta}|}=p^{a(n)}\|p^{a(n)}\frac{1}{\boldsymbol{\beta}!}\nabla_{\boldsymbol{\beta}}(e)\|\eta^{|\boldsymbol{\beta}|}\leq p^{a(n)}\eta^{|\boldsymbol{\beta}|}$$
for every $\boldsymbol{\beta}$  such that $|\boldsymbol{\beta}|\leq n$, because $p^{a(n)}\frac{1}{\boldsymbol{\beta}!}\nabla_{\boldsymbol{\beta}}(e)\in \tilde{E}$.\\
This means that $\|\frac{1}{\boldsymbol{\beta}!}\nabla_{\boldsymbol{\beta}}(e)\|\eta^{|\boldsymbol{\beta}|}\rightarrow 0$ since
$$0 \leq \|\frac{1}{\boldsymbol{\beta}!}\nabla_{\boldsymbol{\beta}}(e)\|\eta^{|\boldsymbol{\beta}|}\leq p^{a(n)}\eta^{|\boldsymbol{\beta}|}$$
and $p^{a(n)}\eta^{|\boldsymbol{\beta}|}\rightarrow 0$ because $a(n)=o(n)$. So we can say that 
$$R(e)=1\;\forall e \in \tilde{E},$$
and if we take $s\in E$, then there exists a positive integer $k$ such that $p^{k}s \in \tilde{E}$, so that $R(p^{k}s)=1$; moreover we know that $s$ is $\eta$-convergent if and only if $p^{k}s$ is $\eta $ convergent so 
$$R(p^{k}s)=R(s)$$ 
and we can conclude that our stratification has radius of convergence $1$.\\
The converse is also true: if $(E,\epsilon_n)$ is such that $R((E,\epsilon_n),\hat{X})=1$, then $(E,\epsilon_n)$ is special. We choose an $\tilde{E}$ coherent $\mathcal{O}_{\hat{X}}$-module $p$-torsion free such that $K\otimes \tilde{E}=E$; for every $e \in \tilde{E}$ let $a(n,e)$ the minimal integer such that 
$$p^{a(n,e)} \frac{\nabla_{\boldsymbol{\beta}}(e)}{\boldsymbol{\beta}!}\in \tilde{E}$$
for any $\boldsymbol{\beta}$ with $|\boldsymbol{\beta}|\leq n$. Since $R((E,\epsilon_n),\hat{X})=1$, then 
$$\textrm{max}_{|\boldsymbol{\beta}| \leq n}\left( \|\frac{\nabla_{\boldsymbol{\beta}}(e)}{\boldsymbol{\beta}!}\|\right)\eta^n\leq \textrm{max}_{|\boldsymbol{\beta}| \leq n}\left( \|\frac{\nabla_{\boldsymbol{\beta}}(e)}{\boldsymbol{\beta}!}\|\eta^{\frac{|\boldsymbol{\beta}| }{2}}\right)\eta^{\frac{n}{2}}\leq (\textrm{const}) \eta^{\frac{n}{2}},$$
so that for $n\rightarrow \infty$ $p^{a(n,e)}\eta^{n}\rightarrow 0$ for any $\eta< 1$. This means that $a(n,e)=o(n)$ for any $e$ $\in $ $\tilde{E}.$
Now if $e$ $\in$ $\tilde{E}$, then we can write $e=\sum_{i}^{l}f_i e_i$ where $f_i$ $\in$ $\mathcal{O}_{\hat{X}}$ and $e_i$'s are generator of $\tilde{E}$
 which is finitely generated $\mathcal{O}_{\hat{X}}$ module, and we put $a(n):=\textrm{max}_{1\leq i \leq l }a(n,e_i)$ (let us note that $a(n)=o(n)$). If we denote by $d_{\boldsymbol{\beta}}$ the operator $\nabla_{\boldsymbol{\beta}}$ for the trivial stratification $(K\otimes \mathcal{O}_{\hat{X}},\textrm{id})$, then for any $f\in \mathcal{O}_{\hat{X}}$ we have $\frac{d_{\boldsymbol{\beta}}(f)}{\boldsymbol{\beta}!}$ $\in$ $\mathcal{O}_{\hat{X}}$ for any $\boldsymbol{\beta}$. Therefore, for any $e=\sum_{i}^{l}f_i e_i$ $\in $ $\tilde{E}$, we have 
 $$p^{a(n)}\frac{\nabla_{\boldsymbol{\beta}}(e)}{\boldsymbol{\beta}!}=\sum_{i=1}^l \sum_{0\leq \boldsymbol{\gamma} \leq \boldsymbol{\beta}}\frac{d_{\boldsymbol{\gamma}}f_i}{\boldsymbol{\gamma}!}\left(p^{a(n)} \frac{\nabla_{\boldsymbol{\beta}-\boldsymbol{\gamma}}e_i}{(\boldsymbol{\beta}-\boldsymbol{\gamma})!}\right)\,\,\,\in\,\,\,\tilde{E} .$$
Hence $(E,\epsilon_n)$ is special.
\end{proof}
\section{Description of the semistable case}\label{Description of the semistable case}
In what follows we suppose that $X$ is proper semistable variety over $V$, which means that locally for the \'{e}tale topology there is an \'{e}tale map 
$$X\xrightarrow{\acute{e}t} \Sp \frac{V[x_1,\dots,x_n,y_1,\dots,y_m]}{x_1\dotsm x_r -\pi}.$$
We call $M_{X_k}$ the log structure on $X$ induced by the special fiber $X_{k}$ that is a divisor with normal crossing, so locally for the \'etale topology it admits a chart given by
$$\mathbb{N}^r\rightarrow \frac{V[x_1,\dots,x_n,y_1,\dots,y_m]}{x_1\dotsm x_r -\pi}$$
that sends $e_i$ to $x_i$, where $e_i=(0,\dots,1,\dots,0)$ with $1$ at the i-th place.\\
We consider the log structure $N$ induced by the closed point of $\textrm{Spec} (V)$ that has a chart given by
$$\mathbb{N}\rightarrow V,$$
that maps $1$ to $\pi.$
This is explained in \cite{Ka} example 2.5 (1) and example 3.7 (2).\\
We also consider a normal crossing divisor $D$ on $X$ that locally for the \'{e}tale topology is defined by the equation $\{ y_1\dotsm y_s=0\}$ and we indicate by $M_D$ the log structure induced by $D$ on $X$.\\
We consider on $X$ the log structure $M=M_{X_k}\oplus M_D$, that corresponds to the log structure induced by the divisor with normal crossing $X_k\cup D$ in X (let us remark that with the notation $M_{X_k}\oplus M_D$ we indicate the sum in the category of log structures); the structural morphism extends to a log smooth morphism of log schemes $(X,M)\rightarrow(\Sp(V),N)$. Moreover the special fiber is reduced, hence the hypothesis stated at the beginning of section \ref{Relations between algebraic and analytic modules with integrable connections} are satisfied.\\
If we denote by $\hat{D}$ the $p$-adic completion of $D$, then we have a diagram 
\begin{equation*}
\xymatrix{
D_k\ar@{^(->}[r] \ar@{^(->}[d]&\hat{D}\ar@{^(->}[d]\\
X_k \ar@{^(->}[r] \ar[d] & \hat{X} \ar[d]\\
\Sp (k )\ar@{^(->}[r] & \Spf (V). \\
}
\end{equation*}
We suppose that \'etale locally on $\hat{X}$ we have the following diagram
\begin{equation}\label{etale locally}
\xymatrix{ 
\ \hat{D}=\bigcup_{j=1}^{s}D_{j}\ \ar@{^(->}[r] \ar[d] &\ \hat{X} \ar[d]  \\  
\ \bigcup_{j=1}^s\{y_j=0\}\  \ar@{^(->}[r]& \  \Spf (V\{x_1,\dots,x_n,y_1,\dots,y_m\}/(x_1\dotsm x_r-\pi) )  \\ 
}  
\end{equation}
which is cartesian with the vertical maps that are \'etale and the horizontal maps closed immersions.\\
If $\hat{X}_{sing}$ and $\hat{D}_{sing}$ are the singular loci of $\hat{X}$ and $\hat{D}$ respectively, then we will use the following notations:
$$\hat{X}^{\circ}=\hat{X}-(\hat{X}_{sing}\cup \hat{D}_{sing})$$
$$\hat{D}^{\circ}=\hat{D}-(\hat{X}_{sing}\cup \hat{D}_{sing})=\hat{X}^{\circ}\cap \hat{D}.$$
When we consider the situation \'etale locally and fix a diagram (\ref{etale locally}), we have a decomposition of the formal schemes $\hat{X}-\hat{X}_{sing}$, $\hat{X}^{\circ}$ and $\hat{D}^{\circ}$ which will be useful later. First let $\hat{X}_i^{\circ}$ be the open formal scheme of $\hat{X}$ defined by pullback of the open formal scheme of $\Spf (V\{x_1,\dots,x_n,y_1,\dots,y_m\}/(x_1\dotsm x_r-\pi))$ on which all the $x_{i'}$'s for $i'\neq i$ are invertible and let $\hat{X}^{\circ}_{i,j}$ be the open formal subscheme of $\hat{X}^{\circ}_i$ defined by \'etale pullback of the open formal subscheme of $ \Spf (V\{x_1,\dots,x_n,y_1,\dots,y_m\}/(x_1\dotsm x_r-\pi))$, where $x_{i'}$ are invertible $\forall i'\neq i, 1 \leq i \leq r$ and $y_{j'}$ are invertible $\forall j'\neq j, 1 \leq j \leq s$.\\
Moreover we will indicate with $\hat{D}^{\circ}_{i,j}$ the set $\hat{X}^{\circ}_{i,j}\cap\hat{D}=\hat{X}^{\circ}_{i,j}\cap\hat{D}_j$, that is the open formal subscheme of $\hat{D}_j$ defined by pullback of the open formal subscheme of $\Spf (V\{x_1,\dots,x_n,y_1,\dots,\hat{y}_j,\dots ,y_m\}/(x_1\dotsm x_r-\pi))$, where all the $x_{i'}$ and the $y_{j'}$ are invertible for all $ i'\neq i, 1 \leq i \leq r$,$\forall j'\neq j, 1 \leq j \leq s$.
In the previous line $\hat{y}_j$ in $\Spf (V\{x_1,\dots,x_n,y_1,\dots,\hat{y}_j,\dots ,y_m\}/(x_1\dotsm x_r-\pi))$ means that the coordinate $y_j$ is missing.\\
With this notations we have the following relations:
$$\coprod_{i}\hat{X}_i^{\circ}=\hat{X}-\hat{X}_{sing},\;\;\;\;\;\;\coprod_{i,j}\hat{X}^{\circ}_{i,j}=\hat{X}^{\circ},$$
$$\coprod_{i,j}\hat{D}^{\circ}_{i,j}=\hat{D}^{\circ}.$$
Note that this decomposition is defined only if we work \'etale locally and we fix a diagram as (\ref{etale locally}).
If we denote by the subscript $_K$ the rigid analytic space associate to a formal scheme, then the sets  $\hat{D}^{\circ}_{i,j;K}$ and $\hat{X}^{\circ}_{i,j;K}$ can be described as follows:
$$\hat{X}^{\circ}_{i,j;K}=\{P\in \hat{X}_{K}| \,\forall i'\neq i \,\,\,|x_{i'}(P)|=1 \,\,\,,\forall j'\neq j \,\,\,|y_{j'}(P)|=1\,\},$$
$$\hat{D}^{\circ}_{i,j;K}=\{P\in \hat{X}_{K}| \,\forall i'\neq i \,\,\,|x_{i'}(P)|=1 \,\,\,,\forall j'\neq j \,\,\,|y_{j'}(P)|=1\,\,\,,y_j(P)=0\}.$$
Finally we will denote by $\hat{U}$ the open formal subscheme complement of $\hat{D}$ in $\hat{X}$.

\section{Log convergent isocrystals with exponents in $\Sigma$} 
We consider now the category of locally free log convergent isocrystals on $\hat{X}$, that we denote, as before, by $I_{conv}((\hat{X},M)/(\Spf(V),N))^{lf}$. By remark 5.1.3 of \cite{Sh1} we know that there is an equivalence of categories between $((X_k,M)/(\Spf(V),N))_{conv}$, the log convergent site on the special fiber, and $((\hat{X},M)/(\Spf(V),N))_{conv}$, the log convergent site on the lifting, hence an equivalence of categories between $I_{conv}((\hat{X},M)/(\Spf(V),N))^{lf}$ and $I_{conv}((X_k,M)/(\Spf(V),N))^{lf}$, locally free isocrystals on equivalent sites.\\ 
As we saw in section \ref{Relations between algebraic and analytic modules with integrable connections}, through the functor $\Phi$ we can associate to a locally free convergent log isocrystal $\mathcal{E}$ on $\hat{X}$ a locally free infinitesimal log isocrystal $\tilde{\Phi}(\mathcal{E})$. Using the terminology of \cite{Ke} and \cite{Sh6}, in the local situation as in (\ref{etale locally}), $\tilde{\Phi}(\mathcal{E})$ induces a log-$\nabla$-module $E$ on $\hat{X}_K$ with respect to $y_1,\dots y_s$, that means a locally free coherent module $E$ on $\hat{X}_K$ and an integrable connection
$$\nabla: E\rightarrow E\otimes \omega^1_{\hat{X}_K/K},$$
where $\omega^1_{\hat{X}_K/K}$ is the coherent sheaf on $\hat{X}_K$ associated to the isocoherent sheaf $K\otimes \omega^1_{(\hat{X},M)/(\Spf(V),N)}$ on $\hat{X}$. If we are in the situation of (\ref{etale locally}) we can write $\omega^1_{\hat{X}_K/K}$ more explicitly:
if we denote by $\Omega^1_{{\hat{X}_K}/K}$ the sheaf of continuous classical 1-differentials on the rigid analytic space $\hat{X}_K$, then
$$\omega^1_{\hat{X}_K/K}=(\Omega^1_{\hat{X}_K}/K \oplus \bigoplus_{j=1}^s \mathcal{O}_{\hat{X}_K}\textrm{dlog}y_j)/L,$$
where $L$ is the coherent sub $\mathcal{O}_{\hat{X}_K}$ generated by $(dy_j,0)-(0,y_j \textrm{dlog}y_j)$ for $1\leq j \leq s.$
Fixed a $j'$ $\in$ $\{1,\dots, s\}$, i.e a component $\hat{D}_{j';K}=\{y_j'=0\}$ of $\hat{D}_K$, then there is a natural immersion of
$$\Omega^1_{\hat{X}_K}/K \oplus \bigoplus_{j\neq j'}\mathcal{O}_{\hat{X}_K}\textrm{dlog}y_j\rightarrow \omega^{1}_{\hat{X}_K/K}$$
and we call $M_{j'}$ the image.
The endomorphism $\textrm{res}_j$ is obtained by tensoring with $\mathcal{O}_{D_{j;K}}$ the following map 
$$E\rightarrow E \otimes \omega^1_{\hat{X}_K/K} \rightarrow E\otimes \omega^1_{\hat{X}_K/K}/M_j$$
and is called the residue of $E$ along $\hat{D}_{j;K}.$ Thanks to proposition 1.5.3 of \cite{BaCh} we know that there exists a minimal and monic polynomial $P_j$ $\in$ $K[T]$ such that $P_j(res_j)=0$. The exponents of $(E,\nabla)$ along $\hat{D}_{j;K}$ are the roots of $P_j$.\\
We fix a set $\Sigma=\prod_{h=1}^k\Sigma_h$ $\in$ $\mathbb{Z}_p^k$, where $k$ is the number of the irreducible components of $\hat{D}=\bigcup_{h=1}^k\hat{D}^h$ in $\hat{X}$.\\ 
If there exists an \'etale covering $\coprod_l\phi_l:\coprod_l \hat{X}_l \rightarrow \hat{X}$ such that every $\hat{X}_l$ has a diagram as in (\ref{etale locally}), then we can define a function of sets $h:\{1,\dots,r\}\times\{1,\dots,s\}\rightarrow \{1,\dots,k\}$ as follows: with the notation as in the previous paragraph $\phi_l(\hat{D}^{\circ}_{i,j,l})$ is contained in one irreducible component of $\hat{D}$, which we denote by $\hat{D}^{h(i,j)}$. We denote by $\Sigma_{h(i,j)}$ the factor of $\Sigma$ corresponding to the component $\hat{D}^{h(i,j)}$.
\begin{definition}\label{residueanalytic} 
A locally free convergent isocrystal $\mathcal{E}$ has exponents along $\hat{D}_K$ in $\Sigma$ if there exists an \'etale covering $\coprod_l\phi_l:\coprod_l \hat{X}_l \rightarrow \hat{X}$ such that every $\hat{X}_l$ has a diagram
\begin{equation}\label{etale locally X_l}
\xymatrix{ 
\ \hat{D}_l=\bigcup_{j=1}^{s}\hat{D}_{j,l}\ \ar@{^(->}[r] \ar[d] &\ \hat{X}_l \ar[d]  \\  
\ \bigcup_{j=1}^s\{y_{l,j}=0\}\  \ar@{^(->}[r]& \  \Spf V\{x_{l,1},\dots,x_{l,n},y_{l,1},\dots,y_{l,m}\}/(x_{l,1}\dots x_{l,r}-\pi)   \\ 
}  
\end{equation}
as in (\ref{etale locally}) with $\hat{D}_l:=\phi_l^{-1}(\hat{D})$. Moreover, for every $j\in\{1,\dots, s\}$, for every $l$, the log-$\nabla$-module $E_l$ on $\hat{X}_{l;K}$ induced by $\mathcal{E}$   has exponents along $\hat{D}_{j,l;K}$ in $\cap_{i=1}^r\Sigma_{h(i,j)}$, if $\phi_l(\hat{D}^{\circ}_{i,j,l})\subset \hat{D}^{h(i,j)}$.\\
We denote the category of locally free log convergent isocrystals with exponents in $\Sigma$ by $I_{conv}((\hat{X},M)/(\Spf(V),N))^{\Sigma}$ or $I_{conv}(\hat{X}/\Spf(V))^{log,\Sigma}$.

\end{definition}
In the next lemma we prove that the definition of isocrystals with exponents along $\hat{D}_K$ in $\Sigma$ is well posed.
\begin{lemma}\label{indepnilpotentresidue}
The notion of locally free log convergent isocrystal with exponents in $\Sigma$ is independent on the choice of the \'etale covering and the diagram as in (\ref{etale locally}), which are chosen in definition \ref{residueanalytic}.
\end{lemma}
\begin{proof}
Let us suppose that $\mathcal{E}$ is a log convergent isocrystal with exponents along $\hat{D}_{K}$ in $\Sigma$. It is sufficient to prove that for any \'etale morphism $\phi: \hat{X}^{\prime}\rightarrow \hat{X}$, such that for $\hat{X}^{\prime}$ there exists a diagram
\begin{equation}\label{etale locally X'}
\xymatrix{ 
\ \hat{D}^{\prime}=\bigcup_{j'=1}^{s'}\hat{D}^{\prime}_{j'}\ \ar@{^(->}[r] \ar[d] &\ \hat{X}^{\prime} \ar[d]  \\  
\ \bigcup_{j'=1}^{s'}\{y^{\prime}_{j'}=0\}\  \ar@{^(->}[r]& \  \Spf V\{x^{\prime}_1,\dots,x^{\prime}_{n'},y^{\prime}_1,\dots,y^{\prime}_{m'}\}/(x^{\prime}_1\dots x^{\prime}_{r'}-\pi)   \\ 
}  
\end{equation}
as in (\ref{etale locally}), with $\hat{D}^{\prime}:=\phi^{-1}(\hat{D})$, the log-$\nabla$-module $E^{\prime}$ on $\hat{X}^{\prime}_K$ induced by $\mathcal{E}$ has exponents along $\hat{D}^{\prime}_{j';K}$ in $\cap_{i'=1}^{r'}\Sigma_{h(i',j')}$, if $\phi(\hat{D}^{' \circ}_{i',j'})\subset \hat{D}^{h(i',j')}$ .\\

By hypothesis $\mathcal{E}$ has exponents along $\hat{D}_K$ in $\Sigma$, hence there exists an \'etale covering $\coprod_l\phi_l:\coprod_l \hat{X}_l \rightarrow \hat{X}$ such that every $\hat{X}_l$ has a diagram
\begin{equation}\label{etale locally X_l}
\xymatrix{ 
\ \hat{D}_l=\bigcup_{j=1}^{s}\hat{D}_{j,l}\ \ar@{^(->}[r] \ar[d] &\ \hat{X}_l \ar[d]  \\  
\ \bigcup_{j=1}^s\{y_{l,j}=0\}\  \ar@{^(->}[r]& \  \Spf V\{x_{l,1},\dots,x_{l,n},y_{l,1},\dots,y_{l,m}\}/(x_{l,1}\dots x_{l,r}-\pi)   \\ 
}  
\end{equation}
as in (\ref{etale locally}) with $\hat{D}_l:=\phi_l^{-1}(\hat{D})$ such that for every $l$ the log-$\nabla$-module $E_l$ on $\hat{X}_{l;K}$ induced by $\mathcal{E}$ has exponents along $\hat{D}_{j,l;K}$ in $\cap_{i=1}^{r}\Sigma_{h(i,j)}$, if $\phi_l(\hat{D}^{ \circ}_{i,j})\subset \hat{D}^{h(i,j)}$ .\\

Let us denote by $\hat{X}^{\prime}_l$ the fiber product $\hat{X}^{\prime}\times_{\hat{X}} \hat{X}_l$ and by $\hat{D}^{\prime}_l,\,\, \hat{D}^{\prime}_{j',l},\,\, \hat{D}''_{j,l}$ the inverse image of $\hat{D}_l,\,\, \hat{D}'_{j'},\,\, \hat{D}_{j,l}$ on $\hat{X}'_l$ respectively. With this notation we have  two diagrams on $\hat{X}'_l$:
\begin{equation}\label{etale locally X'_l indotto da X'}
\xymatrix{  
\ \hat{D}'_l=\bigcup_{j'=1}^{s'}D'_{j',l}\ \ar@{^(->}[r] \ar[d] &\ \hat{X}'_l \ar[d]  \\  
\ \bigcup_{j'=1}^{s'}\{y'_{j'}=0\}\  \ar@{^(->}[r]& \  \Spf V\{x'_{1},\dots,x'_{n'},y'_{1},\dots,y'_{m'}\}/(x'_{1}\dots x'_{r'}-\pi)   \\ 
}  
\end{equation}
and 
\begin{equation}\label{etale locally X'_l indotto da X_l}
\xymatrix{ 
\ \hat{D}'_l=\bigcup_{j=1}^{s}\hat{D}''_{j,l}\ \ar@{^(->}[r] \ar[d] &\ \hat{X}'_l \ar[d]  \\  
\ \bigcup_{j=1}^s\{y_{j,l}=0\}\  \ar@{^(->}[r]& \  \Spf V\{x_{1,l},\dots,x_{n,l},y_{1,l},\dots,y_{m,l}\}/(x_{1,l}\dots x_{r,l}-\pi)   .\\ 
}  
\end{equation}
The diagram (\ref{etale locally X'_l indotto da X'}) is induced by (\ref{etale locally X'}) through $p_2$, the projection on the second factor $\mathrm{pr}_2:\hat{X}^{\prime}\times_{\hat{X}} \hat{X}_l\rightarrow  \hat{X}_l$; and the diagram (\ref{etale locally X'_l indotto da X_l}) is induced by (\ref{etale locally X_l}) through $p_1$, the projection on the first factor $\mathrm{pr}_2:\hat{X}^{\prime}\times_{\hat{X}} \hat{X}_l\rightarrow  \hat{X}'$.\\
The log-$\nabla$-module $E'_l$ induced by $\mathcal{E}$ on $\hat{X}'_{l;K}$ has exponents  along $\hat{D}''_{j,l;K}$ which are contained in the set of exponents of $E_l$ along $\hat{D}_{j,l;K}$. This happens because the residue of $E'_l$ along $\hat{D}''_{j,l;K}$, denoted  by $\mathrm{res}^{''}_{j,l}$ is the image of the residue of $E_l$ along $\hat{D}_{j,l;K}$, denoted by $\mathrm{res}_{j,l}$, via the map
$$\textrm{End}_{\mathcal{O}_{\hat{D}_{j,l;K}}}({E_l}|_{\hat{D}_{j,l;K}})\rightarrow \textrm{End}_{\mathcal{O}_{\hat{D}^{ \prime \prime}_{j,l;K}}}({E'_l}|_{\hat{D}^{\prime \prime }_{j,l;K}}) ,$$ 
which is induced by the projection $\mathrm{pr}_{2}$.
If $P_{j,l}$ is the minimal and monic polynomial such that $P_{j,l}(res_{j,l})=0$, then $P_{j,l}(res^{''}_{j,l})=0$, so if we denoted by $P^{''}_{j,l}$ the minimal and monic polynomial such that $P^{''}_{j,l}(res^{''}_{j,l})=0$, then $P^{''}_{j,l}\mid P_{j,l}$. So the roots of $P^{''}_{j,l}$ are contained in the roots of $P_{j,l}$, which means that the exponents of $E'_l$ along $\hat{D}''_{j,l}$ are a subset of the set of exponents of $E_l$ along $\hat{D}_{j,l;K}$. Since for every $(i,j)$ such that $\phi_l \circ \mathrm{pr}_2(\hat{D}^{\prime \prime \circ}_{i,j,l})$ is contained in $\hat{D}^{h(i,j)}$ also $\phi_l(\hat{D}^{\circ}_{i,j,l})$ is contained in $\hat{D}^{h(i,j)}$ and viceversa, then we proved that for every $(i,j)$ such that $\phi_l \circ \mathrm{pr}_2(\hat{D}^{\prime \prime \circ}_{i,j,l})$ is contained in $\hat{D}^{h(i,j)}$ the exponents along $\hat{D}''_{j,l;K}$ are contained in $\cap_{i=1}^r\Sigma_{h(i,j)}$ .\\  
Now we want to look at the exponents of $E'_l$ along $\hat{D}'_{j',l;K}$. \\

Let us put $\hat{X}^{\prime \circ}_l:=\hat{X}'_l\cap \hat{X}^{\circ}$, $\hat{D}^{\prime \circ }_l:=\hat{D}'_l\cap \hat{X}^{\circ}$, $\hat{D}^{\prime \circ }_{l,j'}:=\hat{D}'_{l,j'}\cap \hat{X}^{\circ}$ and $\hat{D}^{\prime \prime \circ }_{l,j}:=\hat{D}''_{l}\cap \hat{X}^{\circ}$. Since the map
$$\textrm{End}_{\mathcal{O}_{\hat{D}'_{l,j';K}}}(E'_l|_{\hat{D}'_{l,j';K}})\rightarrow \textrm{End}_{\mathcal{O}_{\hat{D}^{\prime \circ }_{l,j';K}}}(E'_l|_{\hat{D}^{\prime \circ }_{l,j';K}})$$
is injective because of the local freeness it is enough to look at the exponents of $E'_l|_{\hat{X}^{\prime \circ }_{l;K}}$ along $\hat{D}^{'\circ}_{l,j';K}$.

Let us note that $\hat{D}^{'\circ}_{l}$ is a relative normal crossing divisor in a smooth formal $V$-scheme; if $\hat{D}^{'\circ}_{l}=\bigcup_t \hat{C}_{t,l}$ is the decomposition of $\hat{D}^{'\circ}_{l}$ in irreducible components, from \cite{NaSh} proposition A.0.3 and proposition A.0.7, we can deduce that $\hat{D}^{'\circ}_{i',j',l}$ and $\hat{D}^{'' \circ}_{i,j,l}$, that are irreducible components of $\hat{D}^{'\circ}_{l}$, correspond to some $\hat{C}_{t,l}$'s. 
Thanks to what we have proven before we know that $E^{'}_l|_{\hat{X}^{' \circ}_{l;K}}$ has  exponents along $\hat{D}^{''\circ}_{j,l;K}$ in $\cap_{i=1}^r\Sigma_{h(i,j)}$ where $(i,j)$ are such that $\phi_l \circ \mathrm{pr}_2(\hat{D}^{\prime \prime \circ}_{i,j,l})$ is contained in $\hat{D}^{h(i,j)}$.We now consider $\hat{D}^{'\circ}_{j',l}$, then $\hat{D}^{'\circ}_{i',j',l}$ will coincide with some $C_{t,l}$'s, so will correspond to some  $\hat{D}^{''\circ}_{i,j,l}$.
If $C_{t,l}$ is such that $\phi \circ\mathrm{pr}_1(C_{t,l})\subset \hat{D}^h$ then also $\phi_{l} \circ\mathrm{pr}_2(C_{t,l})\subset \hat{D}^h$ by the commutativity of the following diagram
\begin{equation}
\xymatrix{
\  \hat{X}^{\prime}\times_{\hat{X}} \hat{X}_l \ar[d]^{pr_{1}} \ar[r]^{pr_{2}} & \hat{X}_{l} \ar[d]_{\phi_{l}} \\
 \hat{X}^{\prime} \ar[r]^{\phi} & \hat{X}
 }.
\end{equation}
So we can conclude that the exponents of $E'_l$ along $\hat{D}^{'}_{j',l;K}$ are contained in $\cap_{i'=1}^{r'}\Sigma^{h(i',j')}$ with $(i',j')$ such that $\phi \circ \mathrm{pr}_1(\hat{D}^{'\circ}_{i',j',l}) \subset \hat{D}^{h(i',j')}$.\\
Finally we prove that $E'$ has exponents along $\hat{D}'_{j';K}$ in $\cap_{i'=1}^{r'}\Sigma_{h(i',j')}$ with $(i',j')$ such that $\phi(\hat{D}^{' \circ}_{i',j'})\subset \hat{D}^{h(i',j')}$. Having a surjective \'etale map $\coprod_l \hat{X}'_l\rightarrow \hat{X}'$, the thesis is reduced to prove that the induced map
\begin{equation}\label{coprod}
\textrm{End}_{\mathcal{O}_{\hat{D}^{'}_{j';K}}}(E'\otimes \mathcal{O}_{\hat{D}^{\prime }_{j';K}}) \rightarrow \textrm{End}_{\mathcal{O}_{\coprod_l\hat{D}^{'}_{j',l;K}}}(\coprod_l{E'_l\otimes \mathcal{O}_{\coprod_l\hat{D}^{\prime }_{j',l;K}}} )
\end{equation}
is injective. \\
If $\phi(\hat{D}^{' \circ}_{i',j'})\subset \hat{D}^{h(i',j')}$, then $\phi \circ \mathrm{pr}_1(\hat{D}^{'\circ}_{i',j',l}) \subset \hat{D}^{h(i',j')}$ for every $l$ and  $\phi \circ \mathrm{pr}_1(\coprod_l\hat{D}^{'\circ}_{i',j',l})$ is contained in $\hat{D}^{h(i',j')}$.
One can see that the residue of $E'$ along $\hat{D}^{'}_{j';K}$ goes via the map in (\ref{coprod}) into the residue of $\coprod_l{E'_l}$ along $\coprod_l\hat{D}^{'}_{j',l;K}$.\\
$$\Gamma(\hat{D}'_{j';K},\mathcal{O}_{\hat{D}'_{j';K}})\rightarrow\Gamma(\coprod_l \hat{D}'_{j',l;K}, \mathcal{O}_{\coprod_l\hat{D}'_{j',l;K}} )$$
is injective, since $\coprod_l \hat{X}'_l\rightarrow \hat{X}'$ is \'etale surjective and then faithfully flat.
\end{proof}
\section{Log-$\nabla$-modules on polyannuli}
We recall in this section the notion of log-$\nabla$-modules on some particular rigid space defined and used by Kedlaya in \cite{Ke} and by Shiho in \cite{Sh6}. \\
An aligned interval is a interval $I$ contained in $[0,\infty)$ such that any end point is contained in $\Gamma^{*}$  with $\Gamma^{*} $ the multiplicative divisible closure of the image of the absolute value $|\,\,\,|:K^*\rightarrow \mathbb{R}^+$. An aligned interval is said to be quasi open if is open at any non zero end point.
For an aligned interval we define a polyannulus as the rigid analytic space $A^n_K(I)=\{(t_1,\dots,t_n) \in \mathbb{A}^{n,rig}_K| |t_i|\in I \,\forall \,\, i=1,\dots,n\}.$\\ 
For example we will consider $A^1_K([\lambda,1])$, the rigid annulus with coefficients in $K$ and radii $\lambda$ and $1$, for $\lambda \in \Gamma^{*}$ or $A^n_K([0,0])$, the polyannulus in $n$ coordinates of null radius.\\
If $Y$ is a smooth rigid analytic space and $y_1,\dots, y_s$ are global sections such that they are smooth and meet transversally, then for a subset $\Sigma=\prod_{j=1}^{s}\Sigma_{j}$ $\subset$ $\bar{K}^{s}$ we denote by $\mathrm{LNM}_{Y,\Sigma}$ the category of log-$\nabla$-module on $Y$ such that all the exponents along $\{y_{j}=0\}$ are contained in $\Sigma_{j}$ for every $j=1,\dots,s.$\\
If $Y$ is a smooth rigid analytic space and $y_1,\dots, y_s$ are global sections such that they are smooth and meet transversally, then we set $\omega_{Y\times A^n_K([0,0])/K}=\omega^1_{Y/K}\oplus \bigoplus_{i=1}^{n}\mathcal{O}_{Y}\textrm{dlog} t_i.$ We define a log-$\nabla$-module $(E,\nabla)$ on  $Y\times A^n_K([0,0])/K$ with respect to $y_1\dots, y_s, t_1,\dots t_n$  as a log-$\nabla$-module $(E,\nabla)$ on $Y$ with respect to $y_{1},\dots, y_{s}$ with $n$ commuting endomorphisms $\partial_i=t_i\frac{\partial}{\partial t_i}$ of $(E,\nabla)$ for $i=1,\dots n$. If we fix $\Sigma=\prod_{j=1}^{s}\Sigma_j\times \prod_{i=1}^{n}\Sigma_{i}$ $\subset$ $\bar{K}^{s+n}$, we can define a log-$\nabla$-module $(E,\nabla)$ with respect to $y_1\dots y_s, t_1,\dots, t_n$ on $Y\times A^n_K([0,0])/K$ with exponents in $\Sigma$ if the log-$\nabla$-module $(E,\nabla)$ on $Y$ has exponents along $\{y_{j}=0\}$ in $\Sigma_{j}$ and if the commuting endomorphisms $\partial_i=t_i\frac{\partial}{\partial t_i}$ have eigenvalues in $\Sigma_{i}$ for every $i=1,\dots, n$. Following Shiho we denote the category of locally free log-$\nabla$-modules on $Y\times A^n_K([0,0])/K$ with exponents in $\Sigma$ by $\mathrm{LNM}_{Y\times A^n_K([0,0]),\Sigma}$.\\ 
If $I$ is an aligned interval and  $\xi:=(\xi_1,\dots, \xi_n)$ $\in$ $\bar{K}^n$, the log-$\nabla$-module denoted by $(M_{\xi},\nabla_{M_{\xi}})$ is the log-$\nabla$-module on $A_K^n(I)$ given by $(\mathcal{O}_{A_K^n(I)}, d+\sum_{j=1}^n\xi_j\textrm{dlog}t_j).$
We will define now the notion of $\Sigma$-unipotence for log-$\nabla$-modules on a product of a smooth rigid analytic space and a polyannulus (\cite{Sh6} definition 1.3). 
\begin{definition}
Let $Y$ be a smooth rigid analytic space, $y_1,\dots,y_s$  global sections whose zero loci are smooth and meet transversally, $I$ an aligned interval and $\Sigma=\prod_{j=1}^{s+n}\Sigma_j$ $\subset$ $\bar{K}^{s+n}$. We say that $\mathcal{E}=(E,\nabla)$ $\in$ $\mathrm{LNM}_{Y\times A^n_K(I)}$ is $\Sigma$-unipotent if after some finite extension of $K$ there exists a filtration 
$$0\subset\mathcal{E}_1\subset \dots \subset \mathcal{E}_n=\mathcal{E}$$
of subobjects such that every successive quotient $\mathcal{E}_i/\mathcal{E}_{i-1}\cong \pi_1^*{\mathcal{F}}\times \pi_2^*(M_{\xi}, \nabla_{M_{\xi}})$, where $\pi_1$ denotes the first projection, $\pi_2$ the second, $\mathcal{F}$ is a log-$\nabla$-module $\in$ $\mathrm{LNM}_{X,\prod_{j=1}^s \Sigma_j}$ and$(M_{\xi},\nabla_{M_{\xi}})$ denotes the log -$\nabla$-module we defined before with $\xi$ $\in$ $\prod_{j=s+1}^{s+n}\Sigma_j$.\\
We will denote by $\mathrm{ULNM}_{Y\times A^n_K(I), \Sigma}$ the categories of $\Sigma$-unipotent log-$\nabla$-modules on $Y\times A^n_K(I).$
\end{definition} 
\begin{remark}(Remark 1.16 of \cite{Sh6})
We note that in the case of $I=[0,0]$ we have that $\mathrm{ULNM}_{Y\times A^n_K([0,0]), \Sigma}\cong \mathrm{LNM}_{Y\times A^n_K([0,0]), \Sigma}$: every object in $\mathrm{LNM}_{Y\times A^n_K([0,0]), \Sigma}$ is $\Sigma$-unipotent. Let us take $\mathcal{E}$ in $\mathrm{LNM}_{Y\times A^n_K([0,0]), \Sigma}$, seeing as a log-$\nabla$-module $(E, \nabla)$ in $\mathrm{LNM}_{Y, \prod_{j=1}^s\Sigma_j}$ endowed with $\partial_j$ for $j=1,\dots, n$, commuting endomorphisms. To prove that it is $\Sigma$-unipotent we proceed by induction on the rank of $E$. We consider $E_1=\cap_{i=1}^n Ker(\partial_i-\xi_i)$, for some $(\xi_1, \dots, \xi_n)\in\prod_{j=s+1}^{s+n}\Sigma_j$. The submodule $E_1$ is  non zero, it is $\Sigma$-constant and $E/E_1$ is $\Sigma$-unipotent by induction hypothesis. Hence $\mathcal{E}$ is $\Sigma$-unipotent.  
\end{remark}
Shiho (\cite{Sh6} Definition 1.5) defines a functor 
$$\mathcal{U}_I:\mathrm{LNM}_{Y\times A^n_K([0,0]),\Sigma}\rightarrow \mathrm{LNM}_{Y\times A^n_K(I),\Sigma }.$$
It associates to a log-$\nabla$-module $\mathcal{E}$ a log-$\nabla$-module $\mathcal{U}_I(\mathcal{E})$ defined as the sheaf $\pi_1^*{\mathcal{E}}$ and the connection 
$$v\mapsto \pi_1^*(\nabla)v+\sum_{i=1}^n\pi_1^*(N_i)(v)\otimes \frac{dt_i}{t_i},$$
where $N_i$ for $i=1,\dots n$ are the commuting endomorphisms attached to $\mathcal{E}$ with eigenvalues on $\Sigma_i$.\\
We recall here the definition of a non Liouville number, which we will use in the sequel.
\begin{definition}
An element $\alpha$ in $\bar{K}$ is said to be $p$-adically non-Liouville if both the power series $\sum_{n\neq \alpha}\frac{x^n}{\alpha -n}$ and $\sum_{n\neq \alpha}\frac{x^n}{n-\alpha }$ have radius of convergence equal to $1$.\\
\end{definition}
As in definition 1.8 of \cite{Sh6} we can define the following.
\begin{definition}
A set $\Sigma$ $\subset$ $\bar{K}$ is called (NID) (resp. (NLD)) if for any $\alpha$, $\beta$ $\in$ $\Sigma$, $\alpha-\beta$ is a non zero integer (resp. is p-adically non-Liouville). A set $\Sigma=\prod_{j=1}^{s}\Sigma_{j}$ $\subset$ $\bar{K}^{s}$ is called (NID) (resp. (NLD)) if for any $j=1,\dots,s$ $\Sigma_{j}$ is (NID) (resp. (NLD)).\\
\end{definition}
We will use the following result (\cite{Ke} 3.3.4, \cite{Ke} 3.3.6, \cite{Sh6} corollary 1.15 and \cite{Sh6} corollary 1.16)
\begin{theorem}\label{unipotentiequivalenti}
Let $Y$ be as before, $I$ a quasi open interval and $\Sigma=\prod_{j=1}^{s+n}\Sigma_{j}$ $\subset$ $\bar{K}^{s+n}$ which is (NID) and (NLD) then the restriction of the functor $\mathcal{U}_I$ to the $\Sigma$-unipotent log-$\nabla$-modules
$$\mathcal{U}_I:\mathrm{ULNM}_{Y\times A^n_K([0,0]),\Sigma}\rightarrow \mathrm{ULNM}_{Y\times A^n_K(I),\Sigma}$$
is an equivalence of categories. If $I$ is an interval of length $\geq 0$, but not necessarily quasi open, then the functor $\mathcal{U}_I$ is fully faithful.
\end{theorem}  

\section{Log overconvergent isocrystals}
Before defining the category of log overconvergent isocrystals, we recall the notion of log tubular neighborhood given by Shiho in \cite{Sh2} definition 2.2.5 with some restrictive hypothesis and in \cite{Sh3} paragraph 2 in full generality. This is the log version of the tubular neighborhood defined by Berthelot in  \cite{Be}.\\
Given a closed immersion of fine log formal schemes $i:(Z,M_Z)\hookrightarrow (\f{Z},M_{\f{Z}})$, there exists a fine log formal scheme $(\f{Z}^{ex},M_{\f{Z}^{ex}})$ and an associated homeomorphic closed  exact immersion $i^{ex}:(Z,M_Z)\hookrightarrow (\f{Z}^{ex},M_{\f{Z}^{ex}})$ such that the functor that associates $i^{ex}$ to $i$  is a right adjoint functor to the inclusion functor from the category of homeomorphic closed immersions of log formal schemes in the category of closed immersions of log formal schemes. The functor $i\mapsto i^{ex}$ is called the exactification functor and its existence is proven in \cite{Sh3} proposition-definition 2.10.\\
\begin{definition}
Let $(Z,M_Z)\hookrightarrow(\f{Z},M_{\f{Z}})$ be a closed immersion of log formal schemes, then the log tubular neighborhood $]Z[^{log}_{\f{Z}}$ of $(Z,M_Z)$ in $(\f{Z},M_{\f{Z}})$ is defined as the rigid analytic space  $\f{Z}^{ex}_K$ associated to the formal scheme $\f{Z}^{ex}$. We can define the specialization map
$$\textrm{sp}: ]Z[^{log}_{\f{Z}}\rightarrow \hat{\f{Z}},$$
where $\hat{\f{Z}}$ is the completion of $\f{Z}$ along $Z$, as the composition of the usual specialization map $ ]Z[^{log}_{\f{Z}}=\f{Z}^{ex}_{K}\rightarrow \f{Z}^{ex}$ with the map $\f{Z}^{ex}\rightarrow \hat{\f{Z}}$ induced by the morphism $\f{Z}^{ex}\rightarrow \f{Z}$. 
\end{definition}
We can notice that, if the closed immersion $i:(Z,M_Z)\hookrightarrow(\f{Z},M_{\f{Z}})$ is exact, then $\f{Z}^{ex}=\f{Z}_K$ and the log tubular neighborhood $]Z[^{log}_{\f{Z}}$ coincides with the classical tubular neighborhood.\\ 
We define the category of log overconvergent isocrystals for log pairs. Log pairs are defined in paragraph 4 of \cite{Sh4} and  in 2.1 of \cite{ChTs} in the case of trivial log structures. A log pair is a pair $((X,M_X),(\bar{X},M_{\bar{X}}))$ of fine log schemes in characteristic $p$ endowed with a strict open immersion $(X,M_X)\hookrightarrow (\bar{X},M_{\bar{X}})$. A morphism of log pairs $f:((X,M_X),(\bar{X},M_{\bar{X}}))\rightarrow((Y,M_Y),(\bar{Y},M_{\bar{Y}}))$ is a morphism of log schemes $f :\bar{X}\rightarrow \bar{Y}$ that verifies $f(X)\subset Y$.  A log pair $((X,M_X),(\bar{X},M_{\bar{X}}))$ over a log pair $((S,M_S),(\bar{S},M_{\bar{S}}))$ is a log pair endowed with the structural morphism $f:((X,M_X),(\bar{X},M_{\bar{X}}))\rightarrow((S,M_S),(\bar{S},M_{\bar{S}}))$.  We assume that all log pairs are log pairs over a given log pair $((S,M_{S}),(S,M_{S}))$. In paragraph 4 of \cite{Sh4} there is a definition of log overconvergent isocrystals for log pairs over a log pair $((S,M_{S}),(S,M_{S}))$ with $M_{S}$ isomorphic to the trivial log structure; we will give analogous definition in the case of non necessarily trivial $M_S$.\\
A log triple is a triple $((X,M_X),(\bar{X},M_{\bar{X}}),(\f{P},M_{\f{P}}))$ which consists of a log pair $((X,M_{X}),(\bar{X},M_{\bar{X}}))$ and a log formal scheme $(\f{P},M_{\f{P}})$ over a log formal scheme $(\f{S},M_{\f{S}})$ endowed with a closed immersion $(\bar{X},M_{\bar{X}})\hookrightarrow (\f{P},M_{\f{P}}).$ Morphisms of log triples are defined in the natural way, as well as a log triple over an other log triple. We will work only with triples $((X,M_X),(\bar{X},M_{\bar{X}}),(\f{P},M_{\f{P}}))$ over a fixed log triple $((S,M_S),(S,M_{S}),(\f{S},M_{\f{S}}))$.\\ 
As in the classical case, for a log triple $((X,M_X),(\bar{X},M_{\bar{X}}),(\f{P},M_{\f{P}}))$  we can define a strict neighborhood $W$ of $]X[^{log}_{\f{P}}$ in $]\bar{X}[^{log}_{\f{P}}$ to be an admissible open of $]\bar{X}[^{log}_{\f{P}}$ such that $\{W,]\bar{X}[^{log}_{\f{P}}-]{X}[^{log}_{\f{P}}\}$ is an admissible covering of $]\bar{X}[^{log}_{\f{P}}$. Given a sheaf of $O_W$ modules $\mathcal{E}$ we define the sheaf of overconvergent sections as the sheaf $j^{\dag}_W \mathcal{E}=\varinjlim_{W'} \alpha_{W',]\bar{X}[^{log}_{\f{P}}*}\alpha_{W',W}^{-1}\mathcal{E},$ where $W'$ varies among the strict neighborhoods of $]X[^{log}_{\f{P}}$ in $]\bar{X}[^{log}_{\f{P}}$ that are contained in $W$ and $\alpha_{T,T'}:T\hookrightarrow T'$ is the natural inclusion. If $W=]\bar{X}[^{log}_{\f{P}}$, then we will denote the sheaf of overconvergent sections by $j^{\dag}\mathcal{E}$.\\  
We suppose that there exists a commutative diagram 
\begin{equation} \label{diaconj}
\xymatrix{ 
\ (\bar{X},M_{\bar{X}})\ \ar[r]^{j} \ar[d]^{g} &\ (\f{P},M_{\f{P}})\  \ar[d]^{h} \\  
\ (\bar{Y}, M_{\bar{Y}})\  \ar[r]^{i}& \  (\f{Y} , M_{\f{Y}})   \\ 
}
\end{equation}
where $j$ is a closed immersion and $h$ is formally log smooth. If we denote by $(\f{P}(1), M_{\f{P}(1)})$ (resp. $(\f{P}(2), M_{\f{P}(2)})$) the fiber product of $(\f{P}, M_{\f{P}})$ with itself over $(\f{Y},M_{\f{Y}})$ (reps. the fiber product of $(\f{P}, M_{\f{P}})$ with itself over $(\f{Y},M_{\f{Y}})$ three times), then the projections and the diagonal induce the following maps:
$$p_{i}:]\bar{X}[_{\f{P}(1)}^{log}\rightarrow ]\bar{X}[_{\f{P}}^{log}\,\,\, \textrm{for} \,\,\,i=1,2,$$
$$p_{i,j}:]\bar{X}[_{\f{P}(2)}^{log}\rightarrow ]\bar{X}[_{\f{P}(1)}^{log}\,\,\, \textrm{for}  \,\,\,1\leq i,j \leq 3,$$
$$\Delta:]\bar{X}[^{log}_{\f{P}}\rightarrow ]\bar{X}[_{\f{P}(1)}^{log}.$$
\begin{definition}
With the previous notation a log overconvergent isocrystal is a pair $(\mathcal{E},\epsilon)$ consisting of a coherent $j^{\dag}\mathcal{O}_{]\bar{X}[_{\f{P}}^{log}}$-module $\mathcal{E}$ and $\epsilon$ a $j^{\dag}\mathcal{O}_{]\bar{X}[_{\f{P}(1)}^{log}}$-linear isomorphism $\epsilon:p_1^*{\mathcal{E}}\rightarrow p_2^*{\mathcal{E}}$ that satisfies $\Delta^*(\epsilon)=\textrm{id}$ and $p_{1,2}^*(\epsilon)\circ p_{2,3}^*(\epsilon)=p_{1,3}^*(\epsilon).$ We denote by $I^{\dag}(((X,\bar{X})/\f{Y},\f{P})^{log})$ the category of log overconvergent isocrystals on $((X,M_X),(\bar{X},M_{\bar{X}})/\f{Y})$ over $(\f{P},M_{\f{P}})$. We say that $(\mathcal{E},\epsilon)$ is a locally free log overconvergent isocrystal if $\mathcal{E}$ is a locally free $j^{\dag}\mathcal{O}_{]\bar{X}[_{\f{P}}^{log}}$-module and we indicate the category of locally free log overconvergent isocrystals with $I^{\dag}(((X,\bar{X})/\f{Y},\f{P})^{log})^{lf}.$
\end{definition}
In the case of trivial log structures the previous definition coincides with the definition of overconvergent isocrystals given by Berthelot \cite{Be}.
\begin{remark}
Shiho in \cite{Sh4} definition 4.2 defines the category of log overconvergent isocrystals also in a more general situation, but for our purposes the definition we gave is sufficient.
 \end{remark}
Given a log pair $((X, M_X),(\bar{X}, M_{\bar{X}}))$, we assume the existence of  a diagram
\begin{equation}\label{suk}
(\bar{X}, M_{\bar{X}})\xrightarrow{g}(\bar{Y}, M_{\bar{Y}})\xrightarrow{i}(\f{Y}, M_{\f{Y}}),
\end{equation}
where $(\bar{Y}, M_{\bar{Y}})$ is a log scheme over $(S,M_{S})$, $(\f{Y}, M_{\f{Y}})$ is a $p$-adic log formal scheme over $(\f{S},M_{\f{S}})$ and $i$ is a closed immersion.\\
Coming back to the setting discussed in section \ref{Description of the semistable case}, we consider the log triple $((U_k,M),(X_k,M),(\hat{X}, M))$ over $((\Sp (k), N),(\Sp (k), N),(\Spf (V),N))$ and the following commutative diagram
\begin{equation*}
\xymatrix{ 
\ (U_k,M)\ \ar@{^(->}[r] \ar[d] &\ (X_k,M)\  \ar[d]_{f_k} \ar@{^(->}[r] & \ (\hat{X}, M)\ \ar[d] \\  
\ (\Sp (k), N)\  \ar@{^(->}[r]& \  (\Sp (k), N)  \ar@{^(->}[r] & \ \ (\Spf (V), N)\ .\\ 
}
\end{equation*}
The diagram as in (\ref{suk}) is given by
\begin{equation*}
(U_k,M)\xrightarrow{g}(\Sp k,N)\xrightarrow{i}(\Spf V,N)
\end{equation*}
and the commutative diagram as in (\ref{diaconj}) is 
\begin{equation*} 
\xymatrix{ 
\ (X_k,M)\ \ar[r]^{j} \ar[d]^{g} &\ \ (\hat{X},M)) \ar[d]^{h} \\  
\ (\Sp (k), N)\  \ar[r]^{i}& \  (\Spf (V), N) . \\ 
}
\end{equation*}
Let us note that since the immersion $(U_k,M)\hookrightarrow(\hat{X},M)$ is strict and the closed immersion $(X_k,M)\hookrightarrow(\hat{X},M)$ is exact, the log tubes in these cases coincide with the classical tubes:
$$]X_k[^{log}_{\hat{X}}=]X_k[_{\hat{X}}=\hat{X}_{K}$$
$$]U_k[^{log}_{\hat{X}}=]U_k[_{\hat{X}}=\hat{U}_{K}.$$ 
Now we want to give a description of integrable connections associated to locally free log overconvergent isocrystals in our case. By proposition 2.1.10 of \cite{Be} we know that there is an equivalence of categories between Coh$(j^{\dag}\mathcal{O}_{]X_k[_{\hat{X}}})$, the category of $j^{\dag}\mathcal{O}_{]X_k[_{\hat{X}}}$-coherent modules, and the inductive limit category of coherent modules over strict neighborhoods of $]U_k[_{\hat{X}}$ in $]X_k[_{\hat{X}}$. Thanks to remark after proposition 2.1.10 of \cite{Be}, if $(\mathcal{E}, \epsilon)$ is a locally free log overconvergent isocrystal, then $\mathcal{E}$ is a locally free $j^{\dag}\mathcal{O}_{]X_k[_{\hat{X}}}$ module, which means that there exists a strict neighborhood $W$ of  $]U_k[_{\hat{X}}$ in $]X_k[_{\hat{X}}$  and a locally free $\mathcal{O}_{W}$-module $E$, such that $j^{\dag}_W E=\mathcal{E}$ .\\
The log overconvergent isocrystal $(\mathcal{E},\epsilon)$ induces an integrable connection on $\mathcal{E}$
$$\nabla :\mathcal{E}\rightarrow \mathcal{E}\otimes_{j^{\dag}\mathcal{O}_{]X_k[_{\hat{X}}}}j^{\dag}\omega^{1}_{{]X_k[_{\hat{X}}}/K} $$
where $\omega^{1}_{{]X_k[_{\hat{X}}}/K}$ is the restriction of $K\otimes \omega^1_{(\hat{X},M)/(\Spf(V),N)}$ to $]X_k[_{\hat{X}}$; 
moreover given a strict neighborhood $W$ of $]U_k[_{\hat{X}}$ in $]X_k[_{\hat{X}}$, as we saw before, there exists $E$ on $W$ such that  $j^{\dag}_W E=\mathcal{E}$ and there exists an integrable connection
$$\nabla: E\rightarrow E\otimes \omega^1_{W/K}$$
which induces the above connection on $\mathcal{E}$, where $\omega^1_{W/K}$ is the restriction of $K\otimes \omega^1_{(\hat{X},M)/(\Spf(V),N)}$ to W. \\
If \'etale locally we are in the situation described in (\ref{etale locally}), then $W$ contains a subspace of the form 
$$\{P \in \hat{X}_K\,\,\,|\,\,\,\forall j \,\,\,|y_j(P)| \geq \lambda\}$$
for some $\lambda$ $\in $ $(0,1)\cap \Gamma^{*}$ with $\Gamma^{*} $ the divisible closure of the image of the absolute value $|\,\,\,|:K^*\rightarrow \mathbb{R}^+$.\\
Therefore we can restrict $E$ to the space 
$$\{P \in \hat{X}^{\circ}_{i,j;K}\,\,\,|\,\,\, \lambda \leq |y_j(P)| < 1 \}.$$
\begin{proposition}
There is an isomorphism $$\phi: \hat{D}_{j,i;K}^{\circ}\times A^1_K([\lambda,1)) \rightarrow\{P \in \hat{X}^{\circ}_{i,j;K}\,\,\,|\,\,\, \lambda \leq |y_j(P)| < 1 \},$$ where $A^1_K([\lambda,1)):=\{t\in \mathbb{A}^1_K\,\,\,|\,\,\, |t|\,\,\,\in\,\,\,[\lambda,1) \}$.
\end{proposition}
\begin{proof}
If we can prove that $\hat{D}_{j,i;K}^{\circ}\times A^1_K([0,1))\cong \hat{X}^{\circ}_{i,j;K}$, then the isomorphism of the proposition will be clear. To prove this we will apply lemma 4.3.1 of \cite{Ke}: we consider $A=\Gamma(\hat{X}^{\circ}_{i,j;K}, \mathcal{O}_{X_K})$,
and as $B$ the ring $A/y_j A=\Gamma(\hat{D}_{i,j;K}^{\circ},\mathcal{O}_{X_K})$. We can apply lemma 4.3.1 of \cite{Ke} because $\Gamma(\hat{D}_{i,j;K}^{\circ},\mathcal{O}_{X_K})$ is formally smooth over $K$
and we can conclude that 
$$\Gamma(\hat{X}_{i,j;K}^{\circ},\mathcal{O}_{X_K})\cong\Gamma(\hat{D}_{i,j;K}^{\circ},\mathcal{O}_{X_K})[|y_j|]$$
i.e. the isomorphism that we wanted.
\end{proof}
We fix a set $\Sigma=\prod_{h=1}^k\Sigma_h$ $\in$ $\mathbb{Z}_p^k$, where $k$ is the number of the irreducible components of $\hat{D}=\bigcup_{h=1}^k\hat{D}^h$ in $\hat{X}$, with the same notations as before definition \ref{residueanalytic}.\\ 
\begin{definition}\label{unipotentmonodromy}
A log overconvergent isocrystal $\mathcal{E}$ has $\Sigma$-unipotent monodromy if there exists an \'etale covering $\coprod_l\phi_l:\coprod_l \hat{X}_l \rightarrow \hat{X}$ such that every $\hat{X}_l$ has a diagram
\begin{equation}
\xymatrix{ 
\ \hat{D}_l=\bigcup_{j=1}^{s}\hat{D}_{j,l}\ \ar@{^(->}[r] \ar[d] &\ \hat{X}_l \ar[d]  \\  
\ \bigcup_{j=1}^s\{y_{j,l}=0\}\  \ar@{^(->}[r]& \  \Spf V\{x_{1,l},\dots,x_{n,l},y_{1,l},\dots,y_{m,l}\}/(x_{1,l}\dots x_{r,l}-\pi)   \\ 
}  
\end{equation}
as in (\ref{etale locally}) with $\hat{D}_l:=\phi_l^{-1}(\hat{D})$ such that for every $l$  the restriction of the log-$\nabla$-module $E_l$ on $\hat{X}_{l;K}$ to 
$$ \hat{D}^{\circ}_{i,j,l;K}\times A^1_K[\lambda,1),$$
is $\cap_{i=1}^r\Sigma_{h(i,j)}$-unipotent , $\forall$ $(i,j)$ such that $\phi_l(\hat{D}_{i,j,l}^{\circ})\subset \hat{D}^{h(i,j)}$.\\
We denote the category of log overconvergent isocrystals with $\Sigma$-unipotent monodromy by $I^{\dag}(((U_k,X_k)/\Spf(V))^{log, \Sigma}$ or $I^{\dag}((U_k, M),(X_k,M))/(\Spf(V),N))^{\Sigma}$.
\end{definition}
\begin{remark}
In definition \ref{unipotentmonodromy} we do not ask any locally freeness hypothesis, because every object in the category $I^{\dag}((U_k, M),(X_k,M))/(\Spf(V),N))$ is such that $\mathcal{E}$ is locally free. This is clear because $(\mathcal{E},\epsilon)$ induces on a strict neighborhood $W$ an $\mathcal{O}_{W}$-module $E$, such that $j^{\dag}_W E=\mathcal{E}$ endowed with an integrable connection. As $K$ is of characteristic $0$ we can conclude that $E$ is locally free and $\mathcal{E}$ is locally free.
\end{remark}
\section{Unipotence, generization and overconvergent generization}
In this section we recall the three propositions that we will use in the proof of the main theorem. They are proven by Shiho in \cite{Sh6} as generalization of the analogous propositions proven by Kedlaya in \cite{Ke} (assuming $\Sigma=0$). We will write the statements only in the cases that we need, that are simplified versions of the propositions given in paragraph 2 of \cite{Sh6}. 
The first property that we consider is called by Shiho and Kedlaya generization property for monodromy (\cite{Ke} proposition 3.4.3 and \cite{Sh6} proposition 2.4 ).
\begin{proposition}\label{Generization of monodromy}
Let $A$ be an affinoid algebra such that $Y=\textrm{Spm}(A)$ is smooth and endowed with sections $y_1,\dots, y_s$ that are smooth and meet transversally.
Suppose that there exists $A\subseteq L$ such that $L$ is an affinoid algebra over $K$, $\textrm{Spm}(L)$ is smooth, all the $y_i$'s are invertible in $L$ and the spectral norm on $L$ restricts to the spectral norm on $A$. Let $I$ be a quasi open interval contained in $[0,1)$ and $A^n_L(I)$ defined as $\textrm{Spm}(L)\times A^n_K(I)$. Let $\Sigma \subset \mathbb{Z}_p^{n+s}$ be a set which is (NLD) and (NID); if $E$ $\in $ $LNM_{Y\times A_K^n(I),\Sigma}$ is such that the induced object $F$ $\in$ $LNM_{A^n_L(I),\Sigma}$ is unipotent, then $E$ is $\Sigma$-unipotent.
\end{proposition} 
The second result that we need is called overconvergent generization and describes the property of extension of unipotence on strict neighborhoods  (proposition 2.7 of \cite{Sh6} and proposition 3.5.3 of \cite{Ke}).
\begin{proposition}\label{Overconvergent generization}
Let $P$ be a $p$-adic formal affine scheme topologically of finite type over $V$. Let $Y_k\subseteq P_k$ be an open dense subscheme of the special fiber of $P$ such that $P$ is formally smooth over $V$ in a neighborhood of $Y_k$. Let $W$ be a strict neighborhood of $]Y_k[_P$ in $P_K$, $I \subset [0,1)$ a quasi open interval and $\Sigma$ a subset of $\mathbb{Z}_p^n$. Given $E\in $ $LNM_{W\times A^n_K(I),\Sigma}$ such that the restriction of $E$ to $]Y_k[_P\times A^n_K(I)$ is $\Sigma$-unipotent, then for every closed interval $[b,c]\subseteq I$ there exists a strict neighborhood $W'$ of $]Y_k[_P$ in  $P_K$ such that $W'$ is contained in $W$ and such that the restriction of $E$ to $W'\times A_K^n[b,c]$ is $\Sigma$-unipotent.
\end{proposition}
The third property that we need states that, under certain assumptions, a log-$\nabla$-module with exponents in $\Sigma$ that is convergent is $\Sigma$-unipotent (proposition 2.12 of \cite{Sh6} and lemma 3.6.2 of \cite{Ke}).
Before giving the statement we recall what is a log-$\nabla$-module with exponents in $\Sigma$ that is convergent (\cite{Sh6}, definition 2.9).
\begin{definition}
Let $Y$ be a smooth affinoid rigid space endowed with $y_1, \dots, y_s \in \Gamma(Y, O_Y)$ whose zero loci are smooth and meet transversally, let $a \in (0,1] \cap \Gamma^*$ and let $E$ be a log-$\nabla$-module on $X \times A^n_K([0,a))$ with respect to $y_1,\dots, y_s,t_1,\dots,t_n$. Then $E$ is called log convergent if, for any $a' \in (0,a) \cap \Gamma^*$, $\eta \in (0,1)$ and $v \in \Gamma (Y \times A^n_K ([0, a']), E)$, the multisequence
$$b_{i_1,\dots,i_n}:=\left\{ \frac{1}{i_1!, \dotsm i_n!}\left(\prod_{j=1}^n\prod_{l=0}^{i_j-1} \left(t_j\frac{\partial}{\partial t_j}-l            \right)\right)(v)\right\}_{i_1, \dots, i_n}$$
is $\eta$-null which means that for any multisequence $c_{i_1,\dots i_n}$ in any complete extension of $K$ with $|c_{i_1,\dots, i_n}|< \eta^{i_1+\dots +i_n}$ the multisequence $\{c_{i_1,\dots,i_n}b_{i_1,\dots, i_n}\}$ converges to zero. 
\end{definition}
\begin{proposition}\label{Convergent plus nilpotent implies unipotent}
Let $A$ be an integral affinoid algebra such that $Y=\textrm{Spm}(A)$ is smooth. Suppose that there exists $A\subseteq L$ such that $L$ is an affinoid algebra over $K$, $\textrm{Spm}(L)$ is smooth and $y_j$'s are invertible in $L$ and the spectral norm of $L$ restricts to the spectral norm of $A$.  Let $\Sigma \subset \mathbb{Z}_p^{n+s}$ be a set which is (NLD) and (NID); if $E$ is an object of $LNM_{Y\times A^n_K([0,1)), \Sigma}$ which is log convergent, then it is $\Sigma$-unipotent.
\end{proposition}
\section{Extension theorem}
Now we come back to the semistable situation and we will prove that the definition of $\Sigma$-unipotent monodromy is well posed.
\begin{proposition}\label{indepunipmonodromy} Let $\mathcal{E}$ be an overconvergent log isocrystal which is in the category $I^{\dag}((U_k,X_k)/\Spf(V))^{log, \Sigma}$.
The notion of $\Sigma$-unipotent monodromy for $\mathcal{E}$ is independent on the choice of the \'etale cover and of the diagram in (\ref{etale locally}) which we have chosen in definition \ref{unipotentmonodromy}. 
\end{proposition}
\begin{proof}
First we will prove that if $\mathcal{E}$ has $\Sigma$-unipotent monodromy for some diagram as in (\ref{etale locally}), then it has $\Sigma$-unipotent monodromy for any diagram as in (\ref{etale locally}). So we suppose that there exists an \'etale covering $\coprod_l\phi_l:\coprod_l \hat{X}_l \rightarrow \hat{X}$ such that every $\hat{X}_l$ has a diagram as in (\ref{etale locally}). As we saw before $\mathcal{E}$ induces on some $W$ strict neighborhood of  $]U_k[_{\hat{X}}$ in $]X_k[_{\hat{X}}$ a locally free $\mathcal{O}_{W}$-module $E$ with an integrable connection. In the situation of (\ref{etale locally}) $W$ contains the set
  $$\{P \in \hat{X}_{;K}\,\,\,|\,\,\,\forall j \,\,\,|y_j(P)| \geq \lambda\}$$
and the restriction of $E$ to 
$$\hat{D}^{\circ}_{i,j;K}\times A^1_K[\lambda,1)$$
is $\cap_{i=1}^r\Sigma_{h(i,j)}$-unipotent for some $\lambda,$ if $(i,j) $ are such that $\phi_l(\hat{D}^{\circ}_{i,j}) \subset \hat{D}^{h(i,j)}$.\\
Using theorem \ref{unipotentiequivalenti} we can extend $E$ to a module with connection on 
$$\hat{D}_{i,j;K}^{\circ}\times A^1_K[0,1)$$
that is $\cap_{i=1}^r\Sigma_{h(i,j)}$-unipotent. In fact, we can restrict $E$ to a module with connection on $\hat{D}_{i,j;K}^{\circ}\times A^1_K([0,0])$ which is $\cap_{i=1}^r\Sigma_{h(i,j)}$-unipotent and use the equivalence of categories of theorem \ref{unipotentiequivalenti} to extend it to a $\cap_{i=1}^r\Sigma_{h(i,j)}$-unipotent log-$\nabla$-module on $\hat{D}_{i,j;K}^{\circ}\times A^1_K[0,1)$ ([0,1) is a quasi open interval). This is true for every $(i,j)$ so $E$ can be extended to a locally free isocrystal on $((\coprod_{(i,j)}\hat{X}^{\circ}_{i,j}, M)/(\Spf(V),N))$ 
which is convergent because it is convergent on $\hat{U}\cap \coprod_{(i,j)}\hat{X}^{\circ}_{i,j}$ that is an open dense $\hat{X}\cap \coprod_{(i,j)}\hat{X}^{\circ}_{i,j}$ (we are applying here proposition \ref{U denso}). Hence we have a locally free isocrystal on $((\hat{X}^{\circ}, M)/(\Spf(V),N))$ which is convergent and such that has exponents along $\hat{D}_{j}$ in $\cap_{i=1}^r\Sigma_{h(i,j)}$ for every $(i,j)$ such $\phi_l(\hat{D}^{\circ}_{i,j})\subset \hat{D}^{h(i,j)}$.\\
For any other diagram as in (\ref{etale locally}), $E$ on $\hat{D}_{i,j;K}^{\circ}\times A^1_K[\lambda, 1)$ with $(i,j)$ such that $\phi_l(\hat{D}_{i,j}^{\circ})\subset \hat{D}^{h(i,j)}$ is the restriction of a convergent module with connection on  $\hat{D}^{\circ}_{i,j:K}\times A^1_K[0,1)$ with exponents in $\cap_{i=1}^r \Sigma^{h(i,j)}$, so that it is $\cap_{i=1}^r\Sigma^{h(i,j)}$-unipotent on $\hat{D}_{i,j;K}^{\circ}\times A^1_K[\lambda, 1)$ by proposition \ref{Convergent plus nilpotent implies unipotent}.\\
Now we prove that the notion of $\Sigma$-unipotence is independent on the choice of the \'etale covering. To do this, it suffices to prove that if $\mathcal{E}$ is a log overconvergent isocrystal with $\Sigma$-unipotent monodromy, for any \'etale morphism $\phi: \hat{X}'\rightarrow \hat{X}$ such that $\hat{X}'$ admits a diagram 
\begin{equation}\label{etale locally X'unip}
\xymatrix{ 
\ \hat{D}^{\prime}=\bigcup_{j=1}^{s}\hat{D}^{\prime}_{j}\ \ar@{^(->}[r] \ar[d] &\ \hat{X}^{\prime} \ar[d]  \\  
\ \bigcup_{j=1}^s\{y^{\prime}_j=0\}\  \ar@{^(->}[r]& \  \Spf V\{x^{\prime}_1,\dots,x^{\prime}_n,y^{\prime}_1,\dots,y^{\prime}_m\}/(x^{\prime}_1\dots x^{\prime}_r-\pi)   \\ 
}  
\end{equation}
as in (\ref{etale locally}), with $\hat{D}':=\phi^{-1}(\hat{D})$, the log-$\nabla$-module $E$ induced on $\hat{X}'$ by $\mathcal{E}$ is $\cap_{i=1}^r\Sigma_{h(i,j)}$-unipotent on $\hat{D}^{' \circ}_{i,j;K}\times A^1_K[\lambda,1)$, for every $(i,j)$ such that $\hat{D}^{' \circ}_{i,j}$ is such that $\phi(\hat{D}^{' \circ}_{i,j})\subset \hat{D}^{h(i,j)}$. We may assume that $\hat{D}^{' \circ}_{i,j;K}=\textrm{Spm}A$ is affinoid. As in the proof of proposition \ref{indepnilpotentresidue} we know that there exists an \'etale covering $\coprod_l\hat{X}'_l\rightarrow \hat{X}'$ such that, for any $l$, $\hat{X}'_l$ admits a diagram as in (\ref{etale locally}) such that $\mathcal{E}$ has $\Sigma$-unipotent monodromy with respect to this diagram. \\
However since we have already showed that the notion of $\Sigma$-unipotent monodromy does not depend on the choice of a diagram as in (\ref{etale locally}), we can say that $\mathcal{E}$ has $\Sigma$-unipotent monodromy with respect to a diagram as in (\ref{etale locally}) for $\hat{X}'_l$ induced by the diagram (\ref{etale locally X'unip}). This means that the log-$\nabla$-module $E$ is $\cap_{i=1}^r\Sigma_{h(i,j)}$-unipotent when it is restricted to $\coprod_l(\hat{D}^{' \circ}_{i,j;K}\times_{\hat{X}_{K}}\hat{X}'_{l,K})\times A^{1}_K[\lambda,1)$ with $(i,j)$ such that $\phi (\coprod_{l}(\hat{D}^{' \circ}_{i,j}\times_{\hat{X}}\hat{X}'_{l}))\subset \hat{D}^{h(i,j)}$.\\
Let us take an affine covering $\coprod_h \hat{C}_{h}\rightarrow \coprod_l (\hat{D}^{' \circ}_{i,j}\times_{\hat{X}}\hat{X}'_{l})$ by affine formal schemes and put $\textrm{Spm}L:=\coprod_h \hat{C}_{h;K}$. Then $E$ is $\cap_{i=1}^r\Sigma_{h(i,j)}$-unipotent on $\textrm{Spm}L\times A^1_K[\lambda,1)$ and the affinoid algebras $A$ and $L$ satisfy the assumption of proposition \ref{Generization of monodromy}. Applying proposition \ref{Generization of monodromy} we can conclude that $E$ is $\cap_{i=1}^r\Sigma_{h(i,j)}$-unipotent on $\hat{D}^{' \circ}_{i,j;K}\times A^1_K[\lambda,1)$ as we wanted.
\end{proof}

We can now state the main result, an extension theorem that generalizes theorem 6.4.5 of \cite{Ke} and theorem 3.16 of \cite{Sh6}. The strategy of the proof is the same as the one in \cite{Ke} and \cite{Sh6} and we will follow step by step the proof of theorem 3.16 of \cite{Sh6}.\\
We will need for the proof the following lemma:
\begin{lemma}\label{denso}
If $W$ is an open dense in $P=\textrm{Spf}(A)$, with $A$ a formal $V$ algebra of topologically finite type such that $A_k$ is reduced, then the spectral seminorm on $\mathcal{O}(]W_k[_{P})$ restricts to the spectral seminorm on $\mathcal{O}(]P_k[_{P})$.
\end{lemma}
\begin{proof}
We can suppose that $W$ is defined by the equation $\{g\neq 0\},$ in particular  that $W=\textrm{Spf}\left(A\left\{\frac{1}{g}\right\}\right).$\\
So we have a map of $V$-algebras 
$$A\longrightarrow A \left\{\frac{1}{g}\right\},$$
which is injective modulo $\pi$ because $W_k$ is dense in $A_k$ that is reduced. By the topological Nakayama's lemma (ex 7.2 of \cite{Ei}) we can conclude that we have an inclusion of $V$-algebras
$$A\hookrightarrow A \left\{\frac{1}{g}\right\}$$
which induces an inclusion of affinoid algebras
$$ \mathcal{O}(]P_k[_{P})=A\otimes K\hookrightarrow  \mathcal{O}(]W_k[_{P})=A \left\{\frac{1}{g}\right\}\otimes K$$
Let us take the Banach norm $|\,\,\,|_P$ on $A\otimes K$ induced by $A$ and the Banach norm  $|\,\,\,|_W$ on $A \left\{\frac{1}{g}\right\}\otimes K$ induced by $A\{\frac{1}{g}\}$, then
$$|a|_P=|a|_W$$
for any $a$ $\in$ $K\otimes A.$ 
By the well known formula (see for example \cite{FvdP} corollary 3.4.6)
$$|a|_{P,\textrm{sp}}=\lim_{n\rightarrow \infty}|a^n|_P^{\frac{1}{n}}$$
where with $|\,\,\,|_{P,\textrm{sp}}$ we denote the spectral norm, we are done. 
\end{proof}
\begin{theorem}\label{log extension}
We fix a set $\Sigma=\prod_{h=1}^k\Sigma_h$ $\in$ $\mathbb{Z}_p^k$, where $k$ is the number of the irreducible components of $\hat{D}=\bigcup_{h=1}^k\hat{D}^h$ in $\hat{X}$ and we require that $\Sigma$ has the properties (NID) and (NLD).  
Let us suppose that locally for the \'etale topology we have a diagram as (\ref{etale locally}), then the restriction functor
$$j^{\dag}:I_{conv}((\hat{X},M)/(\Spf(V),N))^{\Sigma}\longrightarrow I^{\dag}((U_k, M),(X_k,M))/(\Spf(V),N))^{\Sigma}$$
is an equivalence of categories.
\end{theorem} 
\begin{proof}
We will divide the proof in 3 steps.\\	
\textbf{Step 1}: the functor $j^{\dag}$ is well defined.\\
Let $\mathcal{E}$ be in $I_{conv}((\hat{X},M)/(\Spf(V),N))^{\Sigma}$, then we prove that $j^{\dag}(\mathcal{E})$ $\in   I^{\dag}((U_k, M),(X_k,M))/(\Spf(V),N))^{\Sigma}$. Thanks to lemma \ref{indepnilpotentresidue} and lemma \ref{indepunipmonodromy}, we can work \'etale locally. We suppose that $\phi$ is an \'etale map to $\hat{X}$ and we call again $\hat{X}$ an \'etale neighborhood for which we have the diagram as in (\ref{etale locally}); in this situation the log convergent isocrystal $\mathcal{E}$ induces a log-$\nabla$-module with respect to $y_1,\dots ,y_s$ on $\hat{X}_K$ such that has exponents along $\hat{D}_{j;K}$ in $\cap_{i=1}^r\Sigma_{h(i,j)}$ if $i$ and $j$ are such that $\phi(\hat{D}^{\circ}_{i,j})\subset \hat{D}^{h(i,j)}$. By the definition of $\Sigma$-unipotent monodromy (definition \ref{unipotentmonodromy}) we are reduced to prove that, if we restrict $E$ to 
$$\hat{D}_{i,j;K}^{\circ}\times A^1_K [0,1),$$
then it is $\cap_{i=1}^r\Sigma_{h(i,j)}$-unipotent  if $i, j$ are such that $\phi(\hat{D}_{i,j}^{\circ})\subset \hat{D}^{h(i,j)}$.\\
We know by hypothesis that the restriction of $E$ to $\hat{D}^{\circ}_{i,j;K}\times A^1_K [0,1)$ is log convergent and has exponents along $\hat{D}_{j;K}$in $\cap_{i=1}^r\Sigma_{h(i,j)}$, hence we have to prove that this implies $\cap_{i=1}^r\Sigma_{h(i,j)}$-unipotence; from proposition \ref{Convergent plus nilpotent implies unipotent} we know that this implies $\cap_{i=1}^r\Sigma_{h(i,j)}$-unipotence.\\
\textbf{Step 2}: the functor $j^{\dag}$ is fully faithful.\\
We have to prove that given $f:\mathcal{E}\rightarrow \mathcal{F}$, a morphism of log overconvergent isocrystals of $\Sigma$-unipotent monodromy, if there exist extensions of $\mathcal{E}$ and $\mathcal{F}$ to log convergent isocrystals with exponents in $\Sigma$ that we call respectively $\tilde{\mathcal{E}}$ and $\tilde{\mathcal{F}}$, then $f$ extends uniquely to $\tilde{f}:\tilde{\mathcal{E}}\rightarrow \tilde{\mathcal{F}}.$\\
We can work \'{e}tale locally. 
We denote by $\phi$ an \'etale map to $\hat{X}$ and again by $\hat{X}$ an \'etale neighborhood that we consider for which there exists a diagram as in (\ref{etale locally}).\\
Let us take $W$ a strict neighborhood of $]U_k[_{\hat{X}}$ in $]X_k[_{\hat{X}}$; by definition $f$ induces a morphism $\varphi$ of $\nabla$-modules between $E_{\mathcal{E}}$ and $E_{\mathcal{F}}$, the $\nabla$-modules on $W$ that are induced by $\mathcal{E}$ and $\mathcal{F}$ respectively:
$$\varphi:E_{\mathcal{E}}\rightarrow  E_{\mathcal{F}}.$$
We call $E_{\tilde{\mathcal{E}}}$ and $E_{\tilde{\mathcal{F}}}$ the log-$\nabla$-modules on $\hat{X}_K$ induced by $\tilde{\mathcal{E}}$ and $\tilde{\mathcal{F}}$.\\
Let us take the following covering of 
$$\hat{X}_K=\bigcup_{J\subset \{1\dots s\}}A_J$$
where 
$$A_J=\{P\in \hat{X}_K|\,\,\, |y_j(P)|<1 \,\,\,(j\in J)\,\,\, |y_j(P)|\geq \lambda\,\,\, (j \notin J) \}$$
and $\lambda \in (0,1)\cap \Gamma^*$ is such that both $\mathcal{E}$ and $\mathcal{F}$ are defined on the following set:
$$B=\{P\in \hat{X}_K|\,\,\, |y_j(P)|\geq \lambda \,\,\, \forall j \}.$$
The covering of $\hat{X}_{K}$ given by the $A_J$'s restricts to the following covering of $B=\bigcup_{J\subset\{1,\dots,s\}}B_J$, where 
$$B_J=\{P\in \hat{X}_K|\,\,\, \lambda \leq |y_j(P)|<1 \,\,\,(j\in J),\,\,\,|y_j(P)|\geq \lambda\,\,\, (j \notin J)\}.$$ 
The extensions $E_{\tilde{\mathcal{E}}}$ and $E_{\tilde{\mathcal{F}}}$ are log convergent in 
\begin{equation}\label{Jfinoa1}
\{P\in \hat{X}_K| y_j(P)=0 \,\,\,(j\in J),\,\,\,|y_j(P)|\geq \lambda\,\,\, (j \notin J)\}\times A^{|J|}[0,1)
\end{equation}
by proposition 3.6 of \cite{Sh6} and they have exponents in $\prod_j\cap_{i=1}^r\Sigma_{h(i,j)}.$
They extend the restrictions of $E_{\mathcal{E}}$ and $E_{\mathcal{F}}$ on 
$$\{P\in \hat{X}_K| y_j(P)=0 \,\,\,(j\in J)\,\,\, |y_j(P)|\geq \lambda\,\,\, (j \notin J)\}\times A^{|J|}[\lambda,1).$$
By theorem \ref{unipotentiequivalenti}  we can conclude that $\phi$ extends to 
$$A_J=\{P\in \hat{X}_K| y_j(P)=0 \,\,\,(j\in J)\,\,\, |y_j(P)|\geq \lambda\,\,\, (j \notin J)\}\times A^{|J|}[0,1);$$
this means that on this set there exists a unique 
$$\tilde{\phi_J}:E_{\tilde{\mathcal{E}}}\rightarrow E_{\tilde{\mathcal{F}}}$$
that extends $\phi$ on $B_J$.\\
On $A_I\cap A_J$ we have the extensions $\phi_I$ and $\phi_J$, which glue because they coincide on the set 
\begin{equation*}
\begin{split}
B_I\cap B_J=\{P\in \hat{X}_K|\,\,\,  \lambda  \leq |y_j(P)| < 1 & \,\,\,(j\in (I\cup J)-(I\cap J)),\\
& |y_j(P)|\geq \lambda \,\,\, (j \notin  (I\cup J))\}\times A^{|I\cap J|}[\lambda,1)
\end{split}
\end{equation*}
because they extend the map $\phi$ on $B_I\cap B_J.$\\
\textbf{Step 3}: the functor $j^{\dag}$ is essentially surjective.\\
Since we have the \'etale descent property for the category of locally free log convergent isocrystals (remark 5.1.7 of \cite{Sh1}) and the full faithfulness of the functor $j^{\dag}$, we may work \'etale locally to prove the essential surjectivity.\\ 
If $\mathcal{E} \in I^{\dag}((U_k,X_k),\Spf(V))^{log, \Sigma}$, then by definition of log overconvergent isocrystal we know that $\mathcal{E}$ induces 
a module with connection on the following set 
$$\{P \in \hat{X}_K|\,\,\, \forall j: |y_j(P)|\geq \lambda\}$$ 
that we will denote by $E$ that is $\cap_{i=1}^r\Sigma_{h(i,j)}$-unipotent on  
$$\hat{D}^{\circ}_{i,j,K}\times A^1[\lambda,1)=\{P \in \hat{D}_{j,K}|\,\,\, \forall j'\neq j : |y_j'(P)|=1, \,\,\, \forall i'\neq i : |x_i'(P)|=1\}\times A^1[\lambda,1) ,$$
for $i$, $j$ such that $\phi(\hat{D}^{\circ}_{i,j})\subset \hat{D}^{h(i,j)}$ .\\
We will prove that $E$ extends to a log-$\nabla$-module on $$C_{a,\lambda}=\{P \in \hat{X}_K|\,\,\, \forall j>a\,\,\,,|y_{j}(P)|\geq \lambda\}$$ $\forall a=0, \dots, s$ with exponents along $\{y_{j}=0\}$ in $\cap_{i=1}^r\Sigma_{h(i,j)}$ with $i,j$ such that $\phi(\hat{D}^{\circ}_{i,j})\subset \hat{D}^{h(i,j)}$, proceeding by induction on $a$. \\
So we suppose, by induction hypothesis, that $E$ extends to the set $C_{a-1,\lambda}=\{P \in \hat{X}_K|\,\,\, \forall j>a-1\,\,\,,|y_j(P)|\geq \lambda\}$ for some $\lambda$ with exponents along $\{y_{j}=0\}$ in $\cap_{i=1}^r\Sigma_{h(i,j)}$ with $i,j$ such that $\phi(\hat{D}^{\circ}_{i,j})\subset \hat{D}^{h(i,j)}$. \\
We consider the following admissible covering of $\hat{X}_K=A\cup B$, where
$$A=\{P \in \hat{X}_K|\,\,\,|y_{a}(P)|\geq \lambda'\}$$
$$B=\{P \in \hat{X}_K|\,\,\,|y_{a}(P)|< 1 \}$$
with $\lambda' \in  [\lambda,1)\cap \Gamma^*.$\\
Intersecting the covering $A\cup B$ with $C_{a-1,\lambda'}$ we obtain the following admissible covering:
\begin{equation}\label{C_{a-1}}
C_{a-1,\lambda'}=(C_{a-1,\lambda'}\cap A)\cup (C_{a-1,\lambda'}\cap B)=
\end{equation}
\begin{equation*}
\begin{split}
=\{P \in \hat{X}_K|\,\,\,\forall j>a-1\,\,\,|y_{j}(P)|&\geq \lambda'\}\cup \\
&\{P \in \hat{X}_K|\,\,\,\forall j>a\,\,\,|y_{j}(P)|\geq \lambda',\,\,\,\lambda' \leq |y_a(P)|<1 \}=
\end{split}
\end{equation*}
\begin{equation*}
\begin{split}
=\{P \in \hat{X}_K|\,\,\,\forall j>a-1\,\,\,|y_{j}(P)|\geq \lambda'\}\cup\\
&\{P \in \hat{D}_{a,K}|\,\,\,\forall j>a\,\,\,|y_{j}(P)|\geq \lambda'\}\times A^1[\lambda',1),
\end{split}
\end{equation*}
and intersecting with $C_{a,\lambda'}$:
\begin{equation}\label{C_a}
C_{a,\lambda'}=(C_{a,\lambda'}\cap A)\cup (C_{a,\lambda'}\cap B)=
\end{equation}
$$\{P \in \hat{X}_K|\,\,\,\forall j>a-1\,\,\,|y_{j}(P)|\geq \lambda'\}\cup\{P \in \hat{X}_K|\,\,\,\forall j>a\,\,\,|y_{j}(P)|\geq \lambda',\,\,\, |y_a(P)|<1 \} $$
$$=\{P \in \hat{X}_K|\,\,\,\forall j>a-1\,\,\,|y_{j}(P)|\geq \lambda'\}\cup\{P \in \hat{D}_{a,K}|\,\,\,\forall j>a\,\,\,|y_{j}(P)|\geq \lambda'\}\times A^1[0,1) .$$
Comparing the formulas in (\ref{C_{a-1}}) and (\ref{C_a}), we see that it is sufficient to prove that $E$ extends from
$$\{P \in \hat{D}_{a,K}|\,\,\, \forall j>a\,\,\,,|y_j(P)|\geq \lambda'\}\times A^1_K[\lambda',1)$$  
to 
$$\{P \in \hat{D}_{a,K}|\,\,\,\forall j>a\,\,\,|y_{j}(P)|\geq \lambda'\}\times A^1[0,1),$$
for some $\lambda'$ $\in$ $[\lambda,1)\cap \Gamma^*$ in a log-$\nabla$-module such that it has exponents along $\{y_{j}=0\}$ in $\cap_{i=1}^r\Sigma_{h(i,j)}$ with $i,j$ for which $\phi(\hat{D}_{i,j} ^{\circ})\subset \hat{D}^{h(i,j)}$.\\
As we saw before $E$ is $\cap_{i=1}^r\Sigma_{h(i,a)}$-unipotent on 
\begin{equation}
\begin{split}
\hat{D}^{\circ}_{i,a,K}\times &A^1_K([\lambda,1))=\\
&\{P \in \hat{D}_{a,K}|\,\,\, \forall j'\neq a : |y_j'(P)|=1, \,\,\, \forall i'\neq i : |x_i'(P)|=1\}\times A^1_K([\lambda,1)),
\end{split}
\end{equation}
so it is $\cap_{i=1}^r\Sigma_{h(i,a)}$-unipotent also on 
\[
\begin{split}
\coprod_i\hat{D}^{\circ}_{i,a,K}\times & A^1_K([\lambda,1))= \\
&
\coprod_i\{P \in \hat{D}_{a,K}|\,\,\, \forall j'\neq a : |y_j'(P)|=1, \,\,\, \forall i'\neq i : |x_i'(P)|=1\}\times A^1_K([\lambda,1)) .
\end{split}
\]
We now want  to apply proposition \ref{Overconvergent generization}; following the notation given in the proposition in our case we have that $P$ is the pull-back of 
$$\textrm{Spf}\frac{V\{x_1,\dots,x_n,y_1\dots \hat{y_a} \dots y_m\}}{x_1\dotsm x_r-\pi}\{y_j^{-1}|j<a\}\{\prod_{i,i'\,\,\,i\neq i'}(x_i-x_i')^{-1}\}$$
by the morphism 
$$\hat{X}\rightarrow \Spf V\{x_1,\dots, x_n, y_1, \dots, y_m\}/(x_1\dotsm x_r - \pi),$$
$Y_k$ is the open defined in $P_k$ by the following equation 
$$\{y_{a+1}\dotsm y_s\neq 0\}$$ 
and 
\begin{equation*}
\begin{split}
W\times A^1_K(I)=\{P \in \hat{D}_{a,K}|&\,\,\, \forall j < a : |y_j(P)|=1,\\
& \,\,\, \forall j>a\,\,\,,|y_j(P)|\geq \lambda  \,\,\, \forall i'\neq i : |x_i'(P)|=1\} \times  A^1_K([\lambda,1)).
\end{split}
\end{equation*}
The hypothesis of proposition \ref{Overconvergent generization} are fulfilled.\\
The restriction of $E$ to  
$$\{P \in \hat{D}_{a,K}|\,\,\, \forall j>a\,\,\,,|y_j(P)|\geq \lambda \}\times A^1_K([\lambda,1))$$
is a log-$\nabla$-module with exponents in $\prod_j\cap_{i=1}^r\Sigma_{h(i,j)}$,  that is $\prod_j\cap_{i=1}^r\Sigma_{h(i,j)}$-unipotent on 
$$]Y_k[_P=\coprod_i\{P \in \hat{D}_{a,K}|\,\,\, \forall j'\neq a : |y_j'(P)|=1, \,\,\, \forall i'\neq i : |x_i'(P)|=1\}\times A^1_K([\lambda,1));$$
so applying proposition \ref{Overconvergent generization} we know that for every $[b,c]$ $\subset [\lambda,1)$ there exists a $\lambda'$ (we suppose that it verifies $\lambda' \in (c,1)$ for gluing reasons) 
such that $E$ is $\prod_j\cap_{i=1}^r\Sigma_{h(i,j)}$-unipotent on
\begin{equation*}
\begin{split}
\coprod_i\{P \in \hat{D}_{a,K}|&\,\,\, \forall j < a : |y_j(P)|=1,\\
& \,\,\, \forall j > a : |y_j(P)|\geq \lambda', \,\,\, \forall i'\neq i : |x_i'(P)|=1\}\times A^1([b,c]).
\end{split}
\end{equation*}
Now we apply proposition \ref{Generization of monodromy} with 
\begin{equation*}
\begin{split}
\textrm{Spm}(L)=\coprod_i\{P \in \hat{D}_{a,K}|&\,\,\, \forall j < a : |y_j(P)|=1,\\
 & \forall j > a : |y_j(P)|\geq \lambda', \,\,\, \forall i'\neq i : |x_i'(P)|=1\}
\end{split}
\end{equation*}
and 
$$Y=\{P \in \hat{D}_{a,K}| \,\,\, \forall j > a : |y_j(P)|\geq \lambda' \};$$
we are in the hypothesis of that proposition thanks to lemma \ref{denso}, hence we deduce that $E$ is $\prod_j\cap_{i=1}^r\Sigma_{h(i,j)}$-unipotent on
$$\{P \in \hat{D}_{a,K}|\,\,\, \forall j > a : |y_j(P)|\geq \lambda'\}\times A^1((b,c)).$$
By theorem \ref{unipotentiequivalenti} we see that $E$ can be extended to a $\prod_j\cap_{i=1}^r\Sigma_{h(i,j)}$-unipotent log-$\nabla$-module,  in particular to a log-$\nabla$-module with exponents in $\prod_j\cap_{i=1}^r\Sigma_{h(i,j)}$ on
$$\{P \in \hat{D}_{a,K}|\,\,\, \forall j > a : |y_j(P)|\geq \lambda'\}\times A^1([0,c)).$$
Now we glue this with $E$ and we obtain a log-$\nabla$-module with exponents in the set $\prod_j\cap_{i=1}^r\Sigma_{h(i,j)}$ on
$$\{P \in \hat{D}_{a,K}|\,\,\, \forall j>a\,\,\,,|y_j(P)|\geq \lambda'\}\times A^1_K([0,1))$$  
as we wanted.\\
Therefore we have a log-$\nabla$-module defined on the space $\hat{X}_K$ and we now prove that it is convergent.\\
We know that the restriction of $E$ to $\hat{U}_K$ is log convergent because it is an extension of an overconvergent log isocrystal on $\hat{U}_K$, hence it belongs to the category
$$I_{\textrm{conv}}((\hat{U}, M),(\Spf(V),N))^{lf}.$$
Since $\hat{U}$ is an open dense in $\hat{X}$, we have a module with log connection defined in the whole space that is convergent on an open dense of the space; we can apply proposition \ref{U denso} and conclude that $E$ is convergent.\\
\end{proof}

\section{Main theorem}
As we saw in proposition \ref{log extension} there is an equivalence of categories 
$$j^{\dag}:I_{conv}((\hat{X},M)/(\Spf(V),N))^{lf,\Sigma}\longrightarrow I^{\dag}((U_k,X_k)/\Spf(V))^{log, \Sigma}.$$ 
Now we want to compare this to the category $MIC((X_K,M)/K)^{lf}$ that we defined in definition \ref{connections}. We can define the notion of exponents also in the algebraic case, giving the analogous definition that we gave before definition \ref{residueanalytic}, replacing the rigid analytic space $\hat{X}_K$ with the algebraic space $X_K$, the divisor $\hat{D}_K$  with the divisor $D_K$ and the $\mathcal{O}_{\hat{X}_K}$-module $\omega^1_{(\hat{X}_K,M)/K}$ with the $\mathcal{O}_{X_K}$-module $\omega^1_{(X_K,M)/K}$.\\ 
We fix a set $\Theta=\prod_{p=1}^f \Theta_p \subset \bar{K}^f$, where $f$ is the number of irreducible components of the divisor $D_K=\cup_{p=1}^fD^{p;K}$.  
We say that (E,$\nabla$) in $MIC((X_K,M_D)/K)$ has residue along $D_{K}$ in $\Theta$ if \'etale locally there exists a diagram analogous to  (\ref{etale locally})
such that for every $l$ the log-$\nabla$-module $X_{l;K}$ induced by ($E$, $\nabla$) has exponents along $D_{j,l;K}$ in $\Theta_{p(j)}$ for every $j$ such that $\phi_{l;K}(\hat{D}_{j,l})\subset D^{p(j)}$.\\
We will denote the category of locally free module with integrable log connection with exponents in $\Theta$ by $MIC((X_K,M_D)/K)^{\Theta}$.
We can prove as in lemma \ref{indepnilpotentresidue} that the notion of exponents in $\Theta$ can be given \'etale locally and that is independent on the choice of a diagram as in (\ref{etale locally}). 
If $(E,\nabla)$ is in $MIC((X_K,M_D)/K)$ we restrict locally \'etale in a situation for which there exists a diagram analogous to (\ref{etale locally}) for the algebraic setting and we look at the exponents of $(E,\nabla)$ along $D_{j;K}$. In particular we consider the log-$\nabla$-module $(\hat{E},\hat{\nabla})$ on $\hat{X}_K$ induced by the log infinitesimal locally free isocrystal $\Psi (E,\nabla)$ (where $\Psi$ is the functor defined in proposition \ref{derhaminf}) and the residue of it along $\hat{D}_{j;K}$.
We have a map 
$$\textrm{End}_{\mathcal{O}_{D_{j;K}}}(E|_{D_{j;K}})\rightarrow \textrm{End}_{\mathcal{O}_{\hat{D}_{j,K}}}(\hat{E}|_{\hat{D}_{j;K}}),$$
that sends the residue of $(E,\nabla)$ along $D_{j;K}$ to the residue of $(\hat{E},\hat{\nabla})$ along $\hat{D}_{j;K}.$
Moreover the map is injective because the map
$$\Gamma(D_{j;K},\mathcal{O}_{D_{j;K}})\rightarrow \Gamma(\hat{D}_{j,K},\mathcal{O}_{\hat{D}_{j;K}})$$
is injective.
This means that \'etale locally $(E,\nabla)$ and the log-$\nabla$-module $(\hat{E},\hat{\nabla})$ have the same exponents along $D_{j;K}$ and $\hat{D}_{j;K}$ respectively. \\ 
The relation between $\Theta$ and $\Sigma$ is as follows. \\
Given a log infinitesimal isocrystal $\mathcal{E}$ with exponents in $\Sigma$ then the functor $\Psi^{-1}$ associates to it a module with integrable log connection $(E,\nabla)$ such that it has exponents in $\Theta=\prod_{p=1}^f\Theta_p$, where the $p$-th component $\Theta_p$ is given by $\cap_{i=1}^r\Sigma^{h(i,j)}$ where $j$ is such that $\phi_{l,K}(D_{j;K})\subset D^{p;K}.$ Viceversa given a module with integrable log connection $(E, \nabla)$ such that it has exponents in $\Theta$, the functor $\Psi$ associates to it a log infinitesimal isocrystal $\mathcal{E}$ with exponents in $\Sigma=\prod_{h=1}^k\Sigma_h$  where the $h$-th component is given by $\Theta^{p(j)}$ where $j$ is such that $\phi(\hat{D}^{\circ}_{i,j})\subset\hat{D}^h$.\\
From this it follows that the functor $\Psi$ induces an equivalence of categories
$$MIC((X_K,M_D)/(K,\textrm{triv}))^{lf,\Theta}\longrightarrow I_{inf}((\hat{X},M)/(\Spf V,N))^{lf,\Sigma}.$$
If we start from a log overconvergent isocrystal $\mathcal{E}$ with $\Sigma$-unipotent monodromy as in \ref{unipotentmonodromy}, we apply the equivalence of category given by the functor $j^{\dag}$ of theorem \ref{log extension} and the observations written above, we can conclude that there is fully faithful functor 
$$I^{\dag}((U_k, M),(X_k,M))/(\Spf(V),N))^{\Sigma}\longrightarrow MIC((X_K,M_D)/(K,\textrm{triv}))^{lf,\Theta}.$$

The logarithmic extension theorem of Andr\ac e and Baldassarri (theorem 4.9 of \cite{AnBa}) gives an equivalence of category between $MIC((X_K,M_D)/(K,\textrm{triv}))^{lf,\Theta}$ and the category of coherent modules with connection on $U_K$ regular along $D_K$, that we denote by $MIC(U_K/K)^{reg}$ gives us the general result.\\
\begin{theorem}
There is a fully faithful functor
$$I^{\dag}((U_k, M),(X_k,M))/(\Spf(V),N))^{\Sigma}\longrightarrow MIC(U_K/K)^{reg}.$$
\end{theorem}

\section*{Acknowledgements}
This work is a generalization of the author's PhD thesis. It is a pleasure to thank her advisors: prof. Francesco Baldassarri and prof. Atsushi Shiho. Much of this work was done while the author was visiting the University of Tokyo; she wants to thank all the people there for their hospitality. In particular she wants to acknowledge prof. Atsushi Shiho for the helpfulness in guiding her work and for the careful reading of her thesis and prof. Takeshi Saito for many useful discussions. The author would like to express her gratitude also to prof. Bruno Chiarellotto and prof. Marco Garuti for valuable advises. Many thanks go to the referee for his useful comments and remarks.


\begin{thebibliography}{}
\bibitem[AnBa]{AnBa} Y. Andr\'e, F. Baldassarri, \emph{De Rham Cohomology of Differential Modules on Algebraic Varieties}, Progress in Mathematics
\textbf{189}, Birkh\"{a}user, 2001. Second edition in preparation.
\bibitem[Ba]{Ba} F. Baldassarri, \emph{Radius of convergence of p-adic connections and the Berkovich ramification locus}, arXiv:1209.0081v3, (2012). To appear in Math. Ann. DOI:10.1007/s00208-012-0866-1.
\bibitem[BaCh]{BaCh} F. Baldassarri, B. Chiarellotto, \emph{Formal and $p$-adic theory of differential systems with logarithmic singularities depending upon parameters}, Duke Math. J., \textbf{72}, 241-300, (1993).

\bibitem[Be]{Be} P. Berthelot, \emph{ Cohomologie rigide et cohomologie rigide \`a support propre.  
Premi\`ere partie}, pr\'epublication de l'IRMAR, 96-03, (1993). 


\bibitem[BeOg]{BeOg} P. Berthelot, A. Ogus, \emph{Notes on crystalline cohomology}, Mathematical
Notes, Princeton University Press, (1978).
\bibitem[Bo]{Bo} S. Bosch, \emph{Lectures on Formal and Rigid Geometry}, Preprintreihe des Mathematischen Instituts, Heft 378, Westf\"alische Wilhelms-Universit\"at, M\"unster, (2005). 
\bibitem[ChFo]{ChFo} B. Chiarellotto, M. Fornasiero, \emph{Logarithmic de Rham, Infinitesimal and Betti cohomologies}, J. Math. Sci. Univ. Tokyo, \textbf{13}, 205-257, (2006).
\bibitem[ChTs]{ChTs} B. Chiarellotto, N. Tsuzuki, \emph{ Cohomological descent of rigid cohomology for \'etale 
coverings}, Rend. Sem. Mat. Univ. Padova, \textbf{109}, 63-215, (2003). 
\bibitem[CtMeII]{CtMeII} G. Christol, Z. Mebkhout, \emph{Sur le th\'eor\`eme de l'indice des \'equations diff\'erentielles $p$-adiques II}, Annals of math., 2nd Ser., \textbf{146}, No. 2, 345-410, (1997).
\bibitem[CtMeIV]{CtMeIV} G. Christol, Z. Mebkhout, \emph{Sur le th\'eor\`eme de l'indice des \'equations diff\'erentielles $p$-adiques IV}, Invent. math., \textbf{143}, 629-672, (2001).
\bibitem[De]{De} P. Deligne, \emph{Cat\'egories Tannakiennes}, in Grothendieck Festschrift II,
Progress in Mathematics, Birkhauser, pp. 111-195, (1990).
\bibitem[DeMi]{DeMi} P. Deligne; J. S. Milne, \emph{Tannakian Categories}, in Hodge cycles, motives, and Shimura varieties, Lecture Notes in Mathematics, 900. Springer-Verlag, Berlin-New York, (1982).
\bibitem[DGS]{DGS} B. Dwork, G. Gerotto, and F. Sullivan, \emph{An Introduction to G-Functions}, Annals of Math. Studies
\textbf{133}, Princeton University Press, Princeton, (1994).
\bibitem[Ei]{Ei} D.Eisenbud, \emph{Commutative algebra with a view toward algebraic geometry}, \textbf{GTM} 150, Springer-Verlag, New York, (1995).

\bibitem[FvdP]{FvdP} J. Fresnel, M. van der Put, \emph{Rigid analytic geometry and its applications}, Progress in Mathematics
\textbf{218}, Birkh\"{a}user, (2004).
\bibitem[Gr]{Gr} A. Grothendieck, \emph{Crystals and the de Rham cohomology of schemes}, in \emph{Dix Exposes sur la Cohomologie des Schemas},  pp. 306-358, North Holland publishing company, (1968). 
\bibitem[Ha]{Ha} R.Hartshorne, \emph{Algebraic Geometry}, \textbf{GTM} 52, Springer-Verlag, New York, (1977).
\bibitem[Ka]{Ka} K. Kato, \emph{Logarithmic structure of Fontaine-Illusie}, in Algebraic Analysis, Geometry and Number Theory, The John Hopkins Univ. Press, 191-224, (1988).
\bibitem[Kz]{Kz} N. M. Katz, \emph{Nilpotent connection and the monodromy theorem: an application of a result of Turrittin}, Pub. Math. IHES, \textbf{39}, 175-232, (1970).
\bibitem[Ke]{Ke} K. S. Kedlaya,\emph{Semistable reduction for overconvergent F-isocrystals,I: Unipotence and logarithmic extensions}, Compositio Math., \textbf{143}, 1164-1212, (2007).
\bibitem[LS]{LS} B. Le Stum, \emph{Rigid Cohomology}, Cambridge Tracts in Mathematics 172, Cambridge University Press, (2007).
\bibitem[NaSh]{NaSh} Y. Nakkajima, A. Shiho, \emph{Weight filtration on log crystalline cohomologies of families of open smooth varieties}, Lecture Notes in Mathematics, \textbf{1959}, Springer-Verlag, (2008).
\bibitem[Og]{Og} A. Ogus, \emph{F-isocrystals and de Rham cohomology II - convergent
isocrystals}, Duke Math. \textbf{51}, 765-850, (1984).
\bibitem[Sh1]{Sh1} A. Shiho, \emph{Crystalline fundamental groups I- Isocrystals on log crystalline site and log convergent site}, Journal of Mathematical Sciences, University of Tokyo, \textbf{7}, 509-656, (2000).
\bibitem[Sh2]{Sh2} A. Shiho, \emph{Crystalline fundamental groups I- Log convergent cohomology and rigid cohomology}, Journal of Mathematical Sciences, University of Tokyo, \textbf{9}, 1-163, (2002).
\bibitem[Sh3]{Sh3} A. Shiho, \emph{Relative log convergent cohomology and relative rigid cohomology I}, arXiv:0707.1742v3, (2007).
\bibitem[Sh4]{Sh4} A. Shiho, \emph{Relative log convergent cohomology and relative rigid cohomology II}, arXiv:0707.1743v2, (2007).
\bibitem[Sh5]{Sh5} A. Shiho, \emph{Relative log convergent cohomology and relative rigid cohomology III}, arXiv:0805.3229v1, (2008).
\bibitem[Sh6]{Sh6} A. Shiho, \emph{On logarithmic extension of overconvergent isocrystals}, Math. Ann., \textbf{348}, no. 2, 467-512, (2010).
\bibitem[Sh7]{Sh7} A. Shiho, \emph{Cut-by-curves criterion for the log extendability of overconvergent isocrystals}, Math. Zeit.,  \textbf{269}  no.1-2, 59-82, (2011).


\end{thebibliography}
\end{document}